\documentclass[11pt]{article}
\usepackage[utf8]{inputenc}
\usepackage[margin=1in]{geometry}

\usepackage{amsmath, amsthm, thmtools}

\usepackage{charter}
\DeclareMathAlphabet{\mathcal}{OMS}{zplm}{m}{n}
\usepackage{graphicx}
\usepackage{amsfonts}

\usepackage{mathtools, amsfonts, amssymb, mathrsfs}
\usepackage{enumerate,enumitem}
\usepackage{subcaption}
\usepackage{color,soul}
\usepackage{algorithm, algpseudocode, algorithmicx}
\usepackage{comment}
\usepackage{xspace}
\usepackage{mleftright}
\usepackage{booktabs}

\usepackage[dvipsnames]{xcolor}
\usepackage[pagebackref]{hyperref}
\hypersetup{
    colorlinks=true,      
    linkcolor=Green,       
    citecolor=NavyBlue,      
    filecolor=Plum,    
    urlcolor=Plum         
}

\algrenewcommand{\algorithmiccomment}[1]{\hfill \textcolor{blue}{$\vartriangleright$ \textit{#1}}}

\newcommand{\real}{\mathbb{R}}
\newcommand{\complex}{\mathbb{C}}

\DeclareMathOperator{\diag}{diag}

\newcommand{\mat}[1]{\boldsymbol{#1}}
\renewcommand{\vec}[1]{\boldsymbol{#1}}

\newcommand{\norm}[1]{\mleft\| #1 \mright\|}

\newcommand{\QR}{\textsf{QR}\xspace}

\newcommand{\SVD}{\textsf{SVD}\xspace}
\DeclareMathOperator{\cond}{cond}

\newcommand{\expmat}[1]{\begin{bmatrix} #1 \end{bmatrix}}
\newcommand{\twobytwo}[4]{\expmat{#1 & #2 \\ #3 & #4}}
\newcommand{\twobyone}[2]{\expmat{#1 \\ #2}}
\newcommand{\onebytwo}[2]{\expmat{#1 & #2}}

\newcommand{\Id}{\mathbf{I}}

\newcommand{\order}{\mathcal{O}}

\DeclareMathOperator*{\argmin}{argmin}
\newcommand{\set}[1]{\mathsf{#1}}
\newcommand{\e}{\mathrm{e}}

\renewcommand{\hat}[1]{\widehat{#1}}
\renewcommand{\tilde}[1]{\widetilde{#1}}
\newcommand{\actionbox}[1]{\begin{tcolorbox}[colback=white,colframe=black,width=\columnwidth,boxsep=5pt,arc=4pt]
  #1
\end{tcolorbox}}

\usepackage[most]{tcolorbox}

\DeclareMathOperator{\err}{err}
\DeclareMathOperator{\BE}{BE}
\DeclareMathOperator{\fl}{fl}
\newcommand{\Rhat}{\smash{\mat{\hat{R}}}}

\newcommand{\yn}[1]{\errmessage{Use of \textbackslash{yn} is not allowed}}
\newcommand{\rr}[1]{#1}
\newcommand{\rrr}[2]{#2}
\usepackage[normalem]{ulem}
\newcommand{\rrrr}[1]{}

\usepackage[nameinlink,capitalise]{cleveref}
\usepackage[sort&compress,numbers]{natbib}
\usepackage{doi}

\renewcommand*{\backref}[1]{}
\renewcommand*{\backrefalt}[4]{%
  \ifcase #1 %
    (No citations.)
  \or
    (Cited on page #2.)
  \else
    (Cited on pages #2.)
  \fi
}
\usepackage[symbol]{footmisc}

\makeatletter
\def\th@plain{%
  \thm@notefont{}
  \itshape 
}
\def\th@definition{%
  \thm@notefont{}
  \normalfont 
}
\makeatother

\crefname{equation}{}{}
\crefname{section}{section}{sections}
\crefname{appendix}{appendix}{appendices}
\newcommand*{\email}[1]{\href{mailto:#1}{\nolinkurl{#1}} } 
\crefname{inftheorem}{Informal Theorem}{Informal Theorems}

\declaretheorem[name=Theorem,numberwithin=section]{theorem}
\declaretheorem[name=Informal Theorem,numberlike=theorem]{inftheorem}
\declaretheorem[name=Proposition,numberlike=theorem]{proposition}
\declaretheorem[name=Fact,numberlike=theorem]{fact}
\crefname{fact}{Fact}{Facts}
\Crefname{fact}{Fact}{Facts}

\declaretheorem[name=Lemma,numberlike=theorem]{lemma}
\declaretheorem[name=Corollary,numberlike=theorem]{corollary}

\theoremstyle{remark}

\theoremstyle{definition}
\declaretheorem[name=Definition,numberlike=theorem]{definition}
\usepackage{mdframed}
\newmdtheoremenv{question}{Question}

\numberwithin{equation}{section}

\title{Fast randomized least-squares solvers can be just as accurate and stable as classical direct solvers\thanks{ENE is supported by the U.S. Department of Energy, Office of Science, Office of Advanced Scientific Computing Research, Department of Energy Computational Science Graduate Fellowship under Award Number DE-SC0021110 and, under aegis of Joel Tropp, by NSF FRG 1952777, Office of Naval Research BRC Award N000142412223, the Carver Mead New Horizons Fund. YN is supported by EPSRC grants EP/Y010086/1 and EP/Y030990/1.}}
\author{Ethan N. Epperly\thanks{Department of Computing and Mathematical Sciences, California Institute of Technology, Pasadena, CA 91125 USA (\email{eepperly@caltech.edu}, \url{https://ethanepperly.com})} \and Maike Meier\thanks{Bernoulli Institute for Mathematics, Computer Science and Artificial Intelligence, University of Groningen, Groningen, 9700 AB, The Netherlands (\email{m.meier@rug.nl})} \and Yuji Nakatsukasa\thanks{Mathematical Institute, University of Oxford, Oxford, OX2 6GG UK (\email{nakatsukasa@maths.ox.ac.uk})}}
\date{\today}

\begin{document}

\maketitle

\begin{abstract}
    One of the greatest success stories of randomized algorithms for linear algebra has been the development of fast, randomized algorithms for highly overdetermined linear least-squares problems. However, none of the existing algorithms is backward stable, preventing them from being deployed as drop-in replacements for existing QR-based solvers. This paper introduces sketch-and-precondition with iterative refinement (SPIR) and FOSSILS, two \emph{provably} backward stable randomized least-squares solvers. SPIR and FOSSILS combine iterative refinement with a preconditioned iterative method applied to the normal equations and converge at the same rate as existing randomized least-squares solvers. This work offers the promise of incorporating randomized least-squares solvers into existing software libraries while maintaining the same level of accuracy and stability as classical solvers.
\end{abstract}

\section{Motivation}

In recent years, researchers in the field of \emph{randomized numerical linear algebra} (RNLA) have developed a suite of randomized methods for linear algebra problems such as linear least-squares \cite{Sar06,RT08,AMT10,PW16} and low-rank approximation \cite{DDH07,HMT11,TW23}.

We will study randomized algorithms for the overdetermined linear least-squares problem:
\begin{equation} \label{eq:ls}
    \vec{x} = \argmin_{\vec{x} \in \real^n} \norm{\vec{b} - \mat{A}\vec{x}} \quad \text{where } \mat{A}\in\real^{m\times n}, \: \vec{b} \in \real^m, \: m\gg n.
\end{equation}
Here, and throughout, $\norm{\cdot}$ denotes the Euclidean norm of a vector or the spectral norm of a matrix.
The classical method for solving \cref{eq:ls} requires a \QR decomposition of $\mat{A}$ at $\order(mn^2)$ cost.
For highly overdetermined problems, with $m\gg n$, RNLA researchers have developed methods for solving \cref{eq:ls} which require fewer than $\order(mn^2)$ operations. 
Indeed, in exact arithmetic, sketch-and-precondition methods \cite{RT08,AMT10} and iterative sketching methods \cite{PW16,OPA19,LP21,Epp24} both solve \cref{eq:ls} to high accuracy in roughly $\order(mn+n^3)$ operations, a large improvement over the classical method if $m\gg n\gg 1$ and much closer to the theoretical lower bound of $\Omega(mn)$.
Despite these seemingly attractive algorithms, programming environments such as MATLAB, numpy, and scipy still use classical \QR-based methods to solve \cref{eq:ls}, even for problems with $m\gg n\gg 1$.
This prompts us to ask:

\actionbox{\textbf{Motivating question:} Could a randomized solver be the default in a programming environment (e.g., MATLAB) for solving large highly overdetermined least-squares problems?}

Replacing an existing solver in MATLAB or numpy, each of which have millions of users, should only be done with assurances that the new solver is just as accurate and reliable as the previous method.
The classic Householder \QR algorithm is backward stable~\cite[Ch.~20]{Hig02}:
\begin{definition}[Backward \rrr{stability}{error and backward stability}] \label{def:backward}
    \rrrr{A numerically computed least-squares solution $\vec{\hat{x}}$ is \emph{backward stable} if it is the exact solution to a slightly perturbed problem}
    \rr{The backward error $\BE(\vec{\hat{x}})$ of a numerically computed solution $\vec{\hat{x}}$ to \cref{eq:ls} is the size of the minimum perturbation to $\mat{A}$ that makes $\vec{\hat{x}}$ the solution to the least-squares problem:}
    \rrrr{\begin{equation} \label{eq:backward_stable}
        \vec{\hat{x}} = \argmin_{\vec{\hat{x}} \in \real^n} \norm{(\vec{b} + \vec{\Delta b}) - (\mat{A} + \mat{\Delta A}) \vec{\hat{x}}} \quad \text{where } \norm{\mat{\Delta A}} \lesssim \norm{\mat{A}} u,\:\norm{\vec{\Delta b}} \lesssim  \norm{\vec{b}} u.
    \end{equation}}
    \begin{equation} \label{eq:backward_stable}
        \rr{\BE(\vec{\hat{x}}) = \min\left\{\norm{\mat{\Delta A}}_{\rm F}\,:\, \vec{\hat{x}} = \argmin_{\vec{\hat{x}\in \real^n}}\norm{\vec{b} - (\mat{A}+ \mat{\Delta A}) \vec{\hat{x}}}\right\}.}
    \end{equation}
    \rr{This solution is \emph{backward stable} if the backward error is bounded by a multiple of the machine precision $\BE(\vec{\hat{x}}) \le c\, \norm{\mat{A}}_{\rm F} u$, where the prefactor $c$ is a (small, low-degree) polynomial in $m$ and $n$.}
    \rrrr{A least-squares solver is backward stable if it always computes a backward stable solution.}
\end{definition}
\rr{Backward stability is the gold-standard guarantee of numerical stability for a least-squares algorithm.
If a method is backward stable, it will also be \emph{forward stable} in the sense of possessing a small \emph{forward error} $\norm{\vec{x} - \vec{\hat{x}}}$; see \cref{sec:analysis-preview} for discussion on different notions of stability.}
To replace \QR-based least-squares solvers in general-purpose software, we need a randomized least-squares solver that is also backward stable.

\begin{figure}[t]
  \centering
  \includegraphics[width=0.45\textwidth]{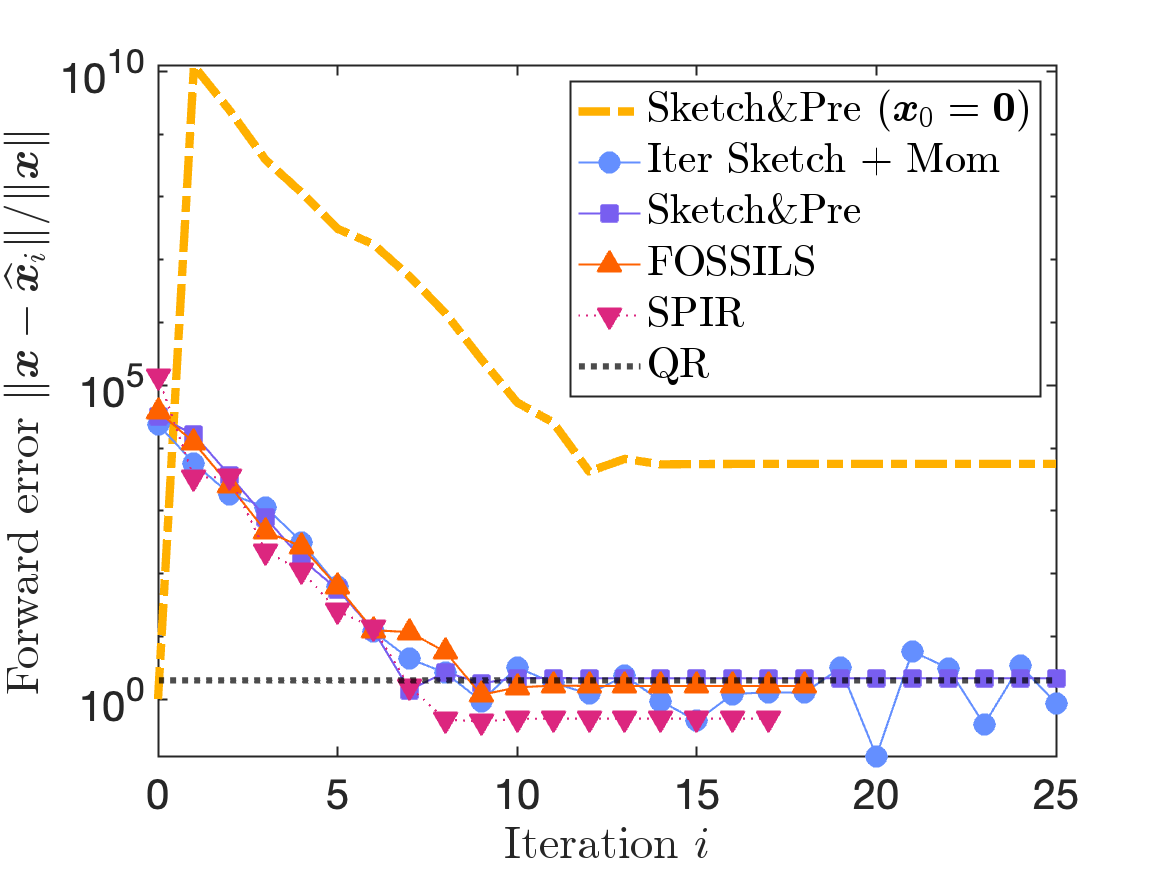}
  \includegraphics[width=0.45\textwidth]{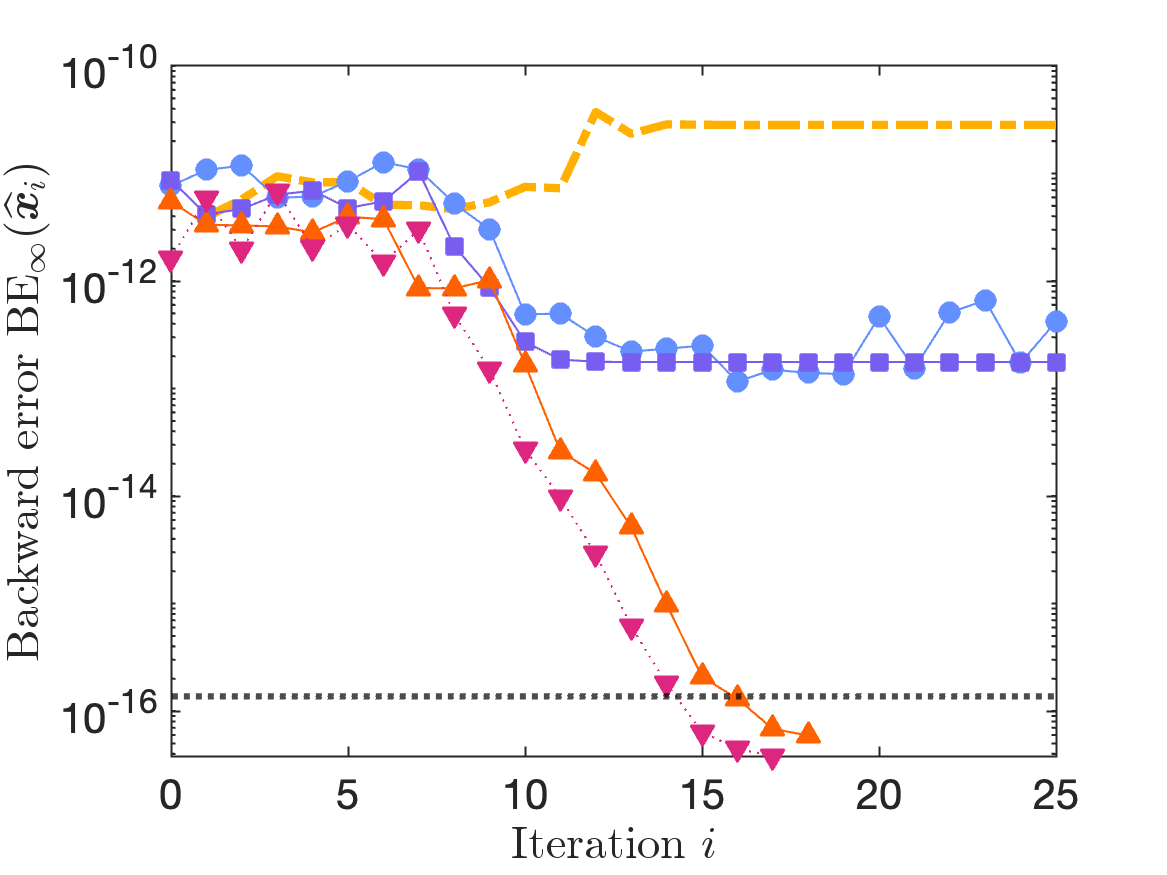} 
    
  \caption{Forward (\emph{left}) and backward (\emph{right}, \cref{eq:backward_error}) error for different randomized least-squares solvers.
  A reference accuracy for Householder \QR is shown as a dashed line.
  All of the randomized methods, except sketch-and-precondition with the zero initialization, are forward stable (achieving comparable forward error to Householder \QR), but only SPIR and FOSSILS are backward stable.} \label{fig:comparison}
\end{figure}

Sadly, existing fast, randomized least-squares methods are not backward stable \cite{MNTW24,Epp24}; see \cref{fig:comparison} for a demonstration.
We test two existing solvers: sketch-and-precondition (with both the zero and sketch-and-solve initialization) and iterative sketching with momentum; see \cref{sec:related} for description of these methods.
We test on a $4000\times 50$ random least-squares problem with residual norm $\norm{\vec{b} - \mat{A}\vec{x}} = 10^{-6}$ and condition number $\kappa \coloneqq \cond(\mat{A}) \coloneqq \sigma_{\rm max}(\mat{A})/\sigma_{\rm min}(\mat{A}) = 10^{12}$.
We generated these problems with a matrix $\mat{A}$ with logarithmically spaced singular values \rr{(i.e., $\sigma_i = 10^{p_i}$ where the $p_i$ are $n$ equally spaced points between 0 and $-\log_{10}\kappa$)} and Haar random singular vectors, and the residual $\vec{b} - \mat{A}\vec{x}$ is uniformly random with the specified norm; see \cite[sec.~1.1.1]{Epp24} for details.
All randomized methods use a sparse sign embedding with dimension of $d = 12n$ (see \cref{sec:implementation}).

Looking first at the left panel of \cref{fig:comparison}, we see all of these methods---except sketch-and-precondition with the zero initialization
---are empirically \emph{forward stable}, achieving roughly the same \emph{forward error} $\norm{\vec{x} - \vec{\hat{x}}}$ as the solution by \QR factorization.
However, \textbf{none of the existing fast, randomized methods is backward stable}.
This is demonstrated in the right panel of \cref{fig:comparison}, which shows that the \emph{backward error} of all existing randomized methods fails to match the backward error of \QR.
(The backward error is a measure of the minimal size of the perturbations $\mat{\Delta A},\mat{\Delta b}$ needed to satisfy \cref{eq:backward_stable}; see \cref{eq:backward_error} for a formal definition of the backward error.)

As existing randomized least-squares solvers are not backward stable, they are not appropriate as replacements for \QR-based solvers in general software, at least not for all least-squares problems.
This motivates us to address the following question:
\actionbox{\textbf{Research question:} Is there a randomized least-squares solver that is both \emph{fast}, running in $\tilde{\order}(mn + n^3)$ operations, and \emph{backward stable}?}
\noindent 
In this article, we answer this question in the affirmative by proposing two \textbf{provably backward stable} randomized least-squares solvers with a fast (nearly) $\order(mn + n^3)$ runtime: 
\begin{itemize}
    \item \textbf{FOSSILS} (\textbf{F}ast \textbf{O}ptimal \textbf{S}table \textbf{S}ketchy \textbf{I}terative \textbf{L}east \textbf{S}quares), a backward stable version of iterative sketching with momentum.
    \item \textbf{SPIR} (\textbf{S}ketch-and-\textbf{P}recondition with \textbf{I}terative \textbf{R}efinement), a backward stable version of sketch-and-precondition.
\end{itemize}

The fast convergence and stability of both SPIR and FOSSILS are shown in \cref{fig:comparison}.
In the left panel, we see that SPIR and FOSSILS converge to the forward error $\norm{\vec{x} - \vec{\hat{x}}}$ of the \QR method at the same rate are both sketch-and-precondition and iterative sketching with momentum.
In the right panel, we see that the backward errors of SPIR and FOSSILS continue to drop past the point where other methods stagnate, converging to the same backward error as the \QR method.

\subsection{\rr{Backward stability: Does it matter?}}

\rr{When implemented carefully, fast randomized least-squares solvers are already (strongly) forward stable (\cref{sec:analysis-preview}), and this stability guarantee is enough for many applications.
Why care about backward stability?
Here are several reasons: 
\begin{enumerate}
    \item \textbf{``As accurate as possible.''} The mere act of representing the problem data $\mat{A}$ and $\vec{b}$ using floating point numbers introduces perturbations on the order of $u$.
    Thus, being the exact solution to a slightly perturbed least-squares problem \cref{eq:backward_stable} is ``as accurate as one could possibly hope to solve a least-squares problem''.
    \item \textbf{Subroutine.}
    Linear algebraic primitives, such as least-squares solvers, are often used as subroutines in larger computations.
    The numerical stability of an end-to-end computation can depend on the backward stability of its subroutines.
    Indeed, even the \emph{forward stability} of randomized least-squares solvers such as iterative sketching depends crucially on the \emph{backward stability} of triangular solves.
    \item \textbf{Componentwise errors.}
    \Cref{thm:backward-componentwise} shows that 
    a method is backward stable if \emph{each component} of the solution $\vec{\hat{x}}$, in the basis of $\mat{A}$'s right singular vectors, is computed accurately.
    As \cref{eq:strong-forward-componentwise,eq:strong-forward-componentwise2} show, componentwise errors can be larger for a forward stable method.
    \item \textbf{Residual orthogonality.} 
    The characterizing geometric property of the least-squares solution is that the residual $\vec{b}-\mat{A}\vec{x}$ is orthogonal to the range of the matrix $\mat{A}$, i.e., $\mat{A}^\top (\vec{b} - \mat{A}\vec{x}) = \vec{0}$.
    The following corollary of \cref{thm:backward-componentwise} compares the size of $\mat{A}^\top (\vec{b} - \mat{A}\vec{\hat{x}})$ for both backward and forward stable methods:
    \begin{corollary}[Residual orthogonality]
        Let $\vec{\hat{x}}$ a backward stable solution to \cref{eq:ls}.
        Then
        \begin{equation*}
            \norm{\smash{\mat{A}^\top (\vec{b} - \mat{A}\vec{\hat{x}})}} \lesssim \norm{\mat{A}}(\norm{\vec{b}} + \norm{\mat{A}} \norm{\vec{\hat{x}}}) u.
        \end{equation*}
        For a strongly forward stable method, this quantity can be as large as
        \begin{equation*}
            \norm{\smash{\mat{A}^\top (\vec{b} - \mat{A}\vec{\hat{x}})}} \gtrsim \norm{\mat{A}}(\norm{\mat{A}} \norm{\vec{x}} + \cond(\mat{A})\norm{\vec{b}-\mat{A}\vec{x}}) u.
        \end{equation*}
    \end{corollary}
    This result shows that, for ill-conditioned ($\cond(\mat{A})\gg 1$) and highly inconsistent ($\norm{\vec{b}-\mat{A}\vec{x}} \gg u$) least-squares problems, strong forward stability is not enough to ensure that $\norm{\smash{\mat{A}^\top(\vec{b}-\mat{A}\vec{\hat{x}})}}$ is small.
    Backward stability, however, is enough.
    See \cref{tab:residual-orthogonality} for a demonstration.
\end{enumerate}}
\begin{table}[t]
    \centering
    \caption{\textbf{\textit{Residual orthogonality.}} Median of $\norm{\mat{A}^\top (\vec{b} - \mat{A}\vec{\hat{x}})}$ for $\vec{\hat{x}}$ computed by different algorithms for $100$ randomly generated least-squares problems.
    Random problems are generated using the same procedure as \cref{fig:comparison} with condition number $\kappa = 10^{12}$ and residual norm $\norm{\vec{b} - \mat{A}\vec{x}} = 10^{-3}$.}
    \begin{tabular}{llc}\toprule
    \textbf{Algorithm} & \textbf{Type of stability} & $\norm{\mat{A}^\top (\vec{b} - \mat{A}\vec{\hat{x}})}$ \\\midrule
    Sketch-and-precondition & Strong forward & 3.9e-9\\ 
    Iterative sketching with momentum & Strong forward & 1.5e-10 \\\midrule 
    Householder \QR & Backward & 5.2e-14 \\
    FOSSILS
    & Backward & 4.0e-14
    \\
    SPIR
    & Backward & 5.3e-14
    \\\bottomrule
    \end{tabular}
    \label{tab:residual-orthogonality}
\end{table}

Ultimately, our view is that---for well-established software packages like MATLAB, numpy, and scipy---the accuracy and stability of the existing solvers must be treated as a fixed target.
In order to use a randomized solver as a drop-in replacement for an existing direct solver, the randomized method must be as accurate and stable as the existing method.
Our hope is that, by developing SPIR and FOSSILS and proving their stability, randomized least-squares solvers may have advanced to a level of reliability that they could be serious candidates for incorporation into software packages such MATLAB, numpy, scipy, and (Rand)LAPACK.
\rr{Even if one is only interested in forward stability, the results of this paper are still of interest, as our paper proves forward stability for sketch-and-precondition (with the sketch-and-solve initialization); see \cref{infthm:spir}.}

\subsection{\rr{The role of randomness}}

\rr{Despite the success of randomization in linear algebra, there remain a lingering belief among some researchers that randomized methods are \emph{inherently} less accurate than deterministic ones.
The instabilities of existing least-squares solvers can be seen as consistent with this perspective.}

\rr{Let us push back on this viewpoint right from the start.
The instabilities identified in this paper are a general property of preconditioned iterative methods for least squares, and they manifest with deterministic preconditioning strategies as well.
Randomness enters our analysis in just two places: the construction of the preconditioner and the construction of the initialization.
The rest of our analysis is deterministic and guarantees backward stable least-squares solutions provided with any high-quality preconditioner and initialization, whether constructed deterministically or randomly (\cref{sec:generalizing}).}

\rr{With this said, let us also not minimize the role of randomness in the algorithms of this paper.
To obtain a backward stable solution in a single iterative refinement step, our analysis requires access to a near-perfect preconditioner (satisfying $\cond(\mat{A}\mat{R}^{-1}) \le \mathrm{const}$) and a near-perfect initialization (satisfying something like $\norm{\vec{b} - \mat{A}\vec{x}_0} \lessapprox \mathrm{const} \cdot \norm{\vec{b} - \mat{A}\vec{x}}$).
The randomized solvers considered in this work are the \emph{only} setting we are aware of where one has both of these ingredients.}

\subsection{Outline}

\Cref{sec:overview} provides an overview of the main ideas of this paper.
It begins by showing how to construct randomized least-squares solvers by combining randomized preconditioning and iterative refinement, introduces the SPIR and FOSSILS algorithms, and provides an overview of our backward stability analysis.
\Cref{sec:related} discusses related work, \cref{sec:FOSSILS} presents implementation guidance for SPIR and FOSSILS, and \cref{sec:experiments} provides numerical experiments.
We develop a general backward stability analysis for a class of randomized least-squares solvers in \cref{sec:analysis} and use this framework to analyze SPIR and FOSSILS in \cref{sec:FOSSILS-analysis,sec:SPIR-analysis}.

\subsection{Notation}

\rr{We adopt the standard model of floating point arithmetic \cite[sec.~2.2]{Hig02}, in which arithmetic operations $\star \in \{+,-,\times,/\}$ are performed up to relative error $u$: $\fl (a \star b)= (a\star b) (1+\delta)$ for $|\delta|\leq u$.}
The relation $\rr{a} \lesssim \rr{b}$ or $\rr{b}\gtrsim \rr{a}$ indicate that $\rr{a} \le \rr{c}(m,n) \rr{b}$ for a prefactor $\rr{c}(m,n)$ that is a polynomial function of $m$ and $n$.
\rr{We will write $a \ll b$ to indicate that $a \ll b/c(m,n)$ for some \emph{sufficiently large} polynomial function $c(m,n)$.}
We make no efforts to explicitly track the prefactors $\rr{c}(m,n)$ in our analysis, but we believe a careful analysis will reveal they are small and low-degree.
The symbol $u$ denotes the unit roundoff, roughly equal to $2\times 10^{-16}$ in double precision arithmetic.
The numerically computed version of a quantity $\vec{f}$ is denoted $\fl(\vec{f})$. 

The (spectral norm) condition number of a matrix $\mat{B}$ is $\cond(\mat{B}) \coloneqq \sigma_{\rm max}(\mat{B}) / \sigma_{\rm  min}(\mat{B})$.
The condition number of $\mat{A}$ is denoted $\kappa \coloneqq \cond(\mat{A})$.
The double lines $\norm{\cdot}$ indicate the $\ell_2$ norm of a vector or spectral norm of a matrix; $\norm{\cdot}_{\rm F}$ is the Frobenius norm.
Throughout this work, we refer to an algorithm for least-squares as ``fast'' if it runs in $\order((mn + n^3) \log m)$ operations or fewer.
We often assume $\norm{\mat{A}}=\norm{\vec{b}}=1$, which loses no generality but often helps simplify the discussion.

\section{Overview of algorithms and main results} \label{sec:overview}

Fast, randomized least-squares solvers were originally introduced by Rokhlin and Tygert in 2008 \cite{RT08}.
Over the past sixteen years, much additional research has been dedicated to developing randomized algorithms for least-squares problems \cite{AMT10,drineas2011faster,PW16,OPA19,LP21,CFS21,MNTW24,Epp24}.
Despite this substantial interest, a fast and backward stable method has remained elusive.
Given this challenge, it may be surprising that our proposed fast, backward stable least-squares solvers are based on a straightforward combination of two classical ingredients: randomized preconditioning and iterative refinement.
We review these ideas in \cref{sec:randomized-preconditioned,sec:iterative-refinement}.

\subsection{Ingredient 1: Randomized preconditioning} \label{sec:randomized-preconditioned}

The core ingredient to most randomized least-squares solvers is \emph{sketching}.

\begin{definition}[Sketching matrix]
    A (random) matrix $\mat{S} \in \real^{d\times m}$ is called a \emph{sketching matrix} or \emph{subspace embedding} for a matrix $\mat{A}$ with \emph{distortion} $0 < \eta < 1$ if
    \begin{equation} \label{eq:sketching-matrix}
        (1-\eta) \norm{\mat{A}\vec{y}} \le \norm{\mat{S}(\mat{A}\vec{y})} \le (1+\eta) \norm{\mat{A}\vec{y}} \quad \text{for all } \vec{y} \in \real^n.
    \end{equation}
    Most constructions for sketching matrices are probabilistic and satisfy \cref{eq:sketching-matrix} with high probability.
\end{definition}

Note that $d\geq n$ is required for~\cref{eq:sketching-matrix} to hold.
A sketching matrix $\mat{S}$ reduces vectors $\vec{z} \in \real^m$ to vectors $\mat{S}\vec{z} \in \real^d$ in a lower dimensional space $d \ll m$ while approximately preserving the lengths of all vectors of the form $\vec{z} = \mat{A}\vec{y}$.
There are many constructions for sketching matrices \cite[sec.~9]{MT20}; we recommend using sparse sign embeddings, which we discuss in \cref{sec:implementation}.

The core insight of Rokhlin and Tygert \cite{RT08} is that sketching matrices can be used to precondition (i.e., reduce the condition number) of a matrix $\mat{A}$ (see also \cite[Prop.~5.4]{KT24}):

\begin{fact}[Randomized preconditioning] \label{fact:whitening}
    Let $\mat{S}$ be an embedding of distortion $0 < \eta < 1$ for a full-rank matrix $\mat{A}$ and consider any factorization $\mat{S}\mat{A} = \mat{Q}\mat{R}$ for $\mat{Q}$ with orthonormal columns and $\mat{R}$ square. 
    Then
    \begin{align}
        \frac{1}{1+\eta} \le \sigma_{\rm min}(\mat{A}\mat{R}^{-1}) &\le \sigma_{\rm max}(\smash{\mat{A}\mat{R}^{-1}}) \le \frac{1}{1-\eta}. \label{eq:AR-whiten}
    \end{align}
    In particular, the condition number of $\mat{A}\mat{R}^{-1}$ is bounded $\cond(\mat{A}\mat{R}^{-1}) \le (1+\eta)/(1-\eta)$.
\end{fact}

To use this result, we begin by sketching $\mat{A}$ and \QR factorizing
\begin{equation} \label{eq:sketch_qr}
    \mat{S}\mat{A} = \mat{Q}\mat{R}.
\end{equation}
Then, we make the change of variables $\vec{y} = \mat{R}\vec{x}$, leading to the equation:
\begin{equation}\label{eq:ls-precond}
    \vec{y} = \argmin_{\vec{y} \in \real^n} \norm{\vec{b} - (\mat{A}\mat{R}^{-1})\vec{y}}.
\end{equation}
To obtain a fast least-squares solver, we can solve \cref{eq:ls-precond} using an iterative method, which we should converge rapidly since $\mat{A} \mat{R}^{-1}$ is well-conditioned.
We then recover $\vec{x} = \mat{R}^{-1}\vec{y}$ by a triangular solve.
In many settings, the preferred iterative algorithms for a least squares problem include Krylov subspace methods LSQR~\cite{PS82a} and LSMR~\cite{FM11}.

In this work, we shall adopt a different perspective than many previous works on randomized least-squares in that we shall solve \cref{eq:ls-precond} by passing to the \emph{preconditioned normal equations}:
\begin{equation} \label{eq:precond-normal}
    (\mat{R}^{-\top} \mat{A}^\top \mat{A} \mat{R}^{-1}) \vec{y} = \mat{R}^{-\top} (\mat{A}^\top \vec{b}).
\end{equation}
Working directly with the normal equations is often ill-advised, but is safe in this context since randomized preconditioning has ensured that $\mat{A}\mat{R}^{-1}$ is well-conditioned.
Iterative methods for square positive definite systems such as \cref{eq:precond-normal} include Krylov subspace methods conjugate gradient (CG)~\cite{HS52} and MINRES~\cite{PS75} and (accelerated) gradient methods such as the Polyak heavy ball method~\cite{Pol64}.
Note that conjugate gradient and MINRES on the normal equations are mathematically equivalent to LSQR and LSMR, respectively.

\subsection{Ingredient 2: Iterative refinement} \label{sec:iterative-refinement}

Existing methods based on randomized preconditioning are forward, but not backward, stable.
In particular, if~\cref{eq:precond-normal} is solved with, for instance, conjugate gradient, the iterations \rr{can} stagnate at a solution that is not backward stable \rr{(as was empirically observed in~\cite{MNTW24,Epp24})}.
An additional ingredient is needed to obtain a backward stable scheme.

The classical approach to improving the quality of a computed solution in numerical linear algebra is \emph{iterative refinement} (see \cite{Wil63} and \cite[ch.~12 and sec.~20.5]{Hig02} for a discussion of iterative refinement in general and \cite{golub1966note,carson2024comparison} for iterative refinement for least-squares in, particular).
Starting from a low-accuracy solution $\vec{x}_0$, iterative refinement for least-squares proceeds as follows: For $i = 0,1,\ldots$,
\begin{enumerate}
    \item Compute the residual $\vec{r}_i \gets \vec{b}-\mat{A}\vec{x}_i$.
    \item Solve for the error $\vec{\delta x}_i \coloneqq \vec{x} - \vec{x}_i$ by solving the least-squares problem 
    \begin{equation} \label{eq:ls-itref}
        \vec{\delta x}_i = \argmin_{\vec{\delta x} \in \real^n} \norm{\vec{r}_i - \mat{A}\vec{\delta x}}.
    \end{equation}
    \item Update the solution $\vec{x}_{i+1} \gets \vec{x}_i + \vec{\delta x}_i$.
\end{enumerate}
The idea of iterative refinement is simple, to improve the quality of an approximate solution $\vec{x}_i$, solve for the error $\vec{\delta x}_i = \vec{x} - \vec{x}_i$ and update $\vec{x}_{i+1} \gets \vec{x}_i + \vec{\delta x}_i$.
Were the system in step 2 to be solved exactly, $\vec{x}_{i+1}$ would be the exact solution to the least-squares problem \cref{eq:ls}.

Classically, the inexact solve in step 2 of iterative refinement is performed using a classical direct solver such as \QR factorization computed in lower numerical precision (i.e., single precision), and the residual and update steps 1 and 3 are performed in higher precision (e.g., double precision).
This \emph{mixed precision iterative refinement} for least-squares was studied by Golub and Wilkinson \cite{golub1966note}, who came to the disappointing conclusion that this approach is only effective for nearly consistent least-squares problems (i.e., problems for which $\norm{\vec{b} - \mat{A}\vec{x}}$ is small).

Our setting is different than the classical mixed precision setting: All steps 1--3 are performed in the same precision, but the solves in step 2 are solved iteratively through the randomly preconditioned normal equations:
\begin{equation}\label{eq:precond-normal-RHSc}
    \vec{c}_i = \mat{R}^{-\top} (\mat{A}^\top \vec{r}_i), \quad (\mat{R}^{-\top} \mat{A}^\top \mat{A} \mat{R}^{-1}) \, \vec{\delta y}_i = \vec{c}_i, \quad \vec{\delta x}_i = \mat{R}^{-1}\,\vec{\delta y}_i.
\end{equation}
As we will see, iterative refinement will be much more effective in this setting than the classical mixed precision setting studied by Golub and Wilkinson.

\subsection{Algorithms} \label{sec:algs}

To obtain backward stable randomized least-squares solvers, we combine randomized preconditioning for the normal equations together with iterative refinement.
We begin by drawing a sketching matrix $\mat{S}$ and sketching and \QR factorizing as in \cref{eq:sketch_qr}.
Then, we construct a cheap ``sketch-and-solve'' initial guess for the least-squares solution 
\begin{equation} \label{eq:sketch-and-solve-2.3}
    \vec{x}_0 \coloneqq \argmin_{\vec{x}_0 \in \real^n} \norm{\mat{S}\vec{b} - (\mat{S}\mat{A}) \vec{x}_0} = \mat{R}^{-1} (\mat{Q}^\top(\mat{S}\vec{b})).
\end{equation}
Sketch-and-solve was proposed by Sarl\'os \cite{Sar06} and suggested as an initialization for iterative randomized least-squares solvers by Rokhlin and Tygert \cite{RT08}.
Then, we use solve the preconditioned normal equations \cref{eq:precond-normal-RHSc} to perform a step of iterative refinement:
\begin{enumerate}
    \item Compute $\vec{r}_0 \gets \vec{b} - \mat{A}\vec{x}_0$.
    \item Solve $(\mat{R}^{-\top} \mat{A}^\top \mat{A} \mat{R}^{-1})\, \vec{\delta y}_0 = \mat{R}^{-\top} (\mat{A}^\top \vec{r}_0)$ using an iterative method.
    \item Update $\vec{x}_1 \gets \vec{x}_0 + \vec{\delta x}_0$ for $\vec{\delta x}_0 \gets \mat{R}^{-1} \, \vec{\delta y}_0$.
\end{enumerate}
We will show that, under minimal stability assumptions on the iterative solver used, the result $\vec{x}_1$ of this single iterative refinement step is a forward stable solution to \cref{eq:ls}.
To obtain a backward stable solution, we perform a \emph{single} additional step of iterative refinement, as above, obtaining a backward stable solution $\vec{x}_2$.
We refer to this procedure, shown in \cref{alg:meta-algorithm}, as the meta-algorithm for fast, backward stable least-squares.

\begin{algorithm}[ht!]
	\caption{General meta-algorithm for fast, backward stable randomized least-squares} \label{alg:meta-algorithm}
	\begin{algorithmic}[1]
		\Require Matrix $\mat{A}\in\real^{m\times n}$, vector $\vec{b}\in\real^m$, subroutine  $\Call{IterativeSolver}{}$
		\Ensure Least-squares solution $\vec{x}_2\in\real^n$
        \State Draw a sketching matrix $\mat{S}$ for $\onebytwo{\mat{A}}{\vec{b}}$
        \State $(\mat{Q},\mat{R}) \gets \Call{QR}{\mat{S}\mat{A}}$ \Comment{Sketch and \QR factorize}
        \State $\vec{x}_0 \gets \mat{R}^{-1}(\mat{Q}^\top (\mat{S}\vec{b}))$ \Comment{Sketch-and-solve initialization}
        \For{$i=0,1$} \Comment{Two steps of iterative refinement with iterative solver}
        \State $\vec{c}_i \gets \mat{R}^{-\top}(\mat{A}^\top(\vec{b} - \mat{A}\vec{x}_i))$
        \State $\vec{\delta y}_{i} \gets \Call{IterativeSolver}{\textsc{Apply} = \vec{z} \mapsto \mat{R}^{-\top}(\mat{A}^\top(\mat{A}(\mat{R}^{-1}\vec{z})),\textsc{RightHandSide} = \vec{c}_i}$ 
        \State $\vec{x}_{i+1} \gets \vec{x}_i + \mat{R}^{-1}\vec{\delta y}_{i}$ \Comment{Iterative refinement}
        \EndFor
	\end{algorithmic}
\end{algorithm}

By choosing a different iterative solver routine \Call{IterativeSolver}{} 
(which solves a linear system $\mat{Mx}=\vec{c}$, where a routine to compute matrix-vector product $\vec{z}\mapsto \mat{M}\vec{z}$ is given)
for step 2 of iterative refinement, we obtain different backward stable randomized least-squares solvers.
In the interest of concreteness, we highlight two particular algorithms corresponding to specific choices for this \Call{IterativeSolver}{}: FOSSILS and sketch-and-precondition with iterative refinement (SPIR).

\textbf{FOSSILS} refers to \cref{alg:meta-algorithm} with the Polyak heavy ball method (\cite{Pol64}, \cref{alg:polyak}) for the \Call{IterativeSolver}{} routine, leading to the following iteration to solve \cref{eq:precond-normal-RHSc}: 
\begin{equation} \label{eq:polyak_iteration}
    \vec{\delta y}_{(j+1)} = \vec{\delta y}_{(j)} + \alpha (\vec{c}_i - \mat{R}^{-\top}(\mat{A}^\top (\mat{A}(\mat{R}^{-1}\,\vec{\delta y}_{(j)})))) + \beta(\vec{\delta y}_{(j)} - \vec{\delta y}_{(j-1)}) \quad j = 1,2,\ldots.
\end{equation}
\rr{When implemented with a sketching matrix with distortion $\eta$,} the optimal value of $\alpha,\beta$ are known \cite{OPA19,LP21}:
\begin{equation} \label{eq:optimal-coeffs}
    \alpha = (1-\eta^2)^2, \quad \beta = \eta^2,
\end{equation}
The notation $\vec{\delta y}_{(j)}$ is reserved for inner iterations with the iterative solver, whereas $\vec{\delta y}_{0}$ and $\vec{\delta y}_{1}$ are the updates in the outer refinement steps.
\rr{We note that determination of the (optimal) coefficients \cref{eq:optimal-coeffs} requires a numerical value of $\eta$ for the particular embedding, which is a limitation of the FOSSILS method.
We use the estimates $\eta\approx \sqrt{n/d}$ or, more conservatively, $\eta\approx 1.1\sqrt{n/d}$ in this work based on the recommendations from \cite[sec.~9.7]{MT20}
See also the discussion in \cref{sec:implementation}.}

\begin{algorithm}[ht!]
	\caption{FOSSILS: Un-optimized implementation \hfill \textcolor{red}{[\textit{In practice, we recommend \cref{alg:FOSSILS-recommended}}]}} \label{alg:FOSSILS-basic}
	\begin{algorithmic}[1]
		\Require Matrix $\mat{A}\in\real^{m\times n}$, right-hand side $\vec{b}\in\real^m$, steps-per-refinement $q > 0$ \Comment{$q \approx \log(1/u)$}
		\Ensure Backward stable least-squares solution $\vec{x}_2\in\real^n$
        \State Draw an embedding $\mat{S}$ for $\operatorname{range}(\mat{A})$ of distortion $\eta$
        \State $\beta \gets \eta^2$, $\alpha \gets (1-\beta)^2$
        \State $(\mat{Q},\mat{R})\gets \Call{QR}{\mat{S}\mat{A}}$ \Comment{Sketch and \QR factorize}
        \State $\vec{x}_0\gets \mat{R}^{-1}(\mat{Q}^\top (\mat{S}\vec{b}))$ \Comment{Sketch-and-solve initialization}
        \For{$i=0,1$} \Comment{Two steps of iterative refinement with Polyak method on normal equations}
        \State $\vec{c}_i \gets \mat{R}^{-\top}(\mat{A}^\top(\vec{b} - \mat{A}\vec{x}_i))$
        \State $\vec{\delta y}_{i} \gets \Call{Polyak}{\mat{A},\mat{R},\vec{c}_i,\alpha,\beta,q}$ \Comment{See \cref{alg:polyak}}
        \State $\vec{x}_{i+1} \gets \vec{x}_i + \mat{R}^{-1}\,\vec{\delta y}_{i}$ \Comment{Iterative refinement}
        \EndFor
	\end{algorithmic}
\end{algorithm}
    
\textbf{Sketch-and-precondition with iterative refinement (SPIR)} refers to \cref{alg:meta-algorithm} with conjugate gradient as the \Call{IterativeSolver}{} routine.
As the name suggests, SPIR consists of Rokhlin and Tygert's sketch-and-precondition method \cite{RT08} run twice, with the first run use to initialize the second.
One can also implement SPIR using LSQR applied to the preconditioned least-squares problem \cref{eq:ls-precond} instead of conjugate gradient applied to \cref{eq:precond-normal}.
Both variants seem similarly stable in practice, though SPIR tends to be more robust; see \cref{sec:fossils-vs-spir}.

\begin{algorithm}[ht!]
	\caption{SPIR: Un-optimized implementation \hfill \textcolor{red}{[\textit{For practical implementation, see \cref{sec:implementation}}]}} \label{alg:SPIR-basic}
	\begin{algorithmic}[1]
		\Require Matrix $\mat{A}\in\real^{m\times n}$, right-hand side $\vec{b}\in\real^m$
		\Ensure Backward stable least-squares solution $\vec{x}_2\in\real^n$
        \State Draw an embedding $\mat{S}$ for $\operatorname{range}(\mat{A})$ of distortion $\eta$
        \State $(\mat{Q},\mat{R})\gets \Call{QR}{\mat{S}\mat{A}}$ \Comment{Sketch and \QR factorize}
        \State $\vec{x}_0\gets \mat{R}^{-1}(\mat{Q}^\top (\mat{S}\vec{b}))$ \Comment{Sketch-and-solve initialization}
        \For{$i=0,1$} \Comment{Two iterative refinement steps with CG on preconditioned normal equations}
        \State $\vec{c}_i \gets \mat{R}^{-\top}(\mat{A}^\top(\vec{b} - \mat{A}\vec{x}_i))$
        \State $\vec{\delta y}_{i} \gets \Call{ConjugateGradient}{\textsc{Apply} = \vec{z} \mapsto \mat{R}^{-\top}(\mat{A}^\top(\mat{A}(\mat{R}^{-1}\vec{z})),\textsc{RightHandSide} = \vec{c}}$ 
        \State $\vec{x}_{i+1} \gets \vec{x}_i + \mat{R}^{-1}\,\vec{\delta y}_{i}$ \Comment{Iterative refinement}
        \EndFor
	\end{algorithmic}
\end{algorithm}

Both SPIR and FOSSILS converge at the same (optimal \cite{LP20}) rate, and they are both fast and reliable in practice, with SPIR being somewhat more robust.
When implemented using a sparse sign embedding (see \cref{sec:implementation}) with conservative parameter settings \cref{eq:cohen}, both of these algorithms require $\order(mn \log(n/u) + n^3 \log n)$ operations.
We recommend either for deployment in practice, with the preference for SPIR for most use cases; see \cref{sec:fossils-vs-spir} for discussion.

\subsection{Stability analysis} \label{sec:analysis-preview}

To present our main results, we begin by formally defining forward stability.
(Recall that backward stability was defined above in \cref{def:backward}.)
To do so, we begin with Wedin's perturbation theorem \cite{Wed73}.
Here is a simplified version:

\begin{fact}[Least-squares perturbation theory] \label{fact:wedin}
    Consider a perturbed least-squares problem \cref{eq:ls}:
    \begin{equation*}
        \vec{\hat{x}} = \argmin_{\vec{\hat{x}} \in \real^n} \norm{(\vec{b} + \vec{\Delta b}) - (\mat{A}+\mat{\Delta A}) \vec{\hat{x}} } \quad \text{with } \norm{\mat{\Delta A}} \le \norm{\mat{A}} \varepsilon,\:\norm{\vec{\Delta b}} \le  \norm{\vec{b}} \varepsilon.
    \end{equation*}
    Then, provided $\cond(\mat{A})\varepsilon \le 0.1$, the following bounds hold:
    \begin{align}
        \norm{\vec{x} - \vec{\hat{x}}} &\le 2.23 \, 
        \cond(\mat{A}) 
        \left(  \norm{\vec{x}} + \frac{\cond(\mat{A})}{\norm{\mat{A}}} \norm{\vec{b} - \mat{A}\vec{x}} \right) \varepsilon, \label{eq:forward_error} \\
        \norm{\mat{A}(\vec{x} - \vec{\hat{x}})} &\le 2.23 \left(  \norm{\mat{A}} \norm{\vec{x}} + \cond(\mat{A}) \norm{\vec{b} - \mat{A}\vec{x}} \right) \varepsilon.\label{eq:residual_error} 
    \end{align}
\end{fact}

\begin{definition}[Forward stability]
    A least-squares solver is \emph{forward stable} if the computed solution $\vec{\hat{x}}$ satisfies the bound \cref{eq:forward_error} for $\varepsilon \lesssim u$.
    If it also satisfies \cref{eq:residual_error} for $\varepsilon \lesssim u$, then we say the method is \emph{strongly forward stable}.
\end{definition}

Put simply, a least-squares solver is forward stable if the forward error $\norm{\vec{x} - \vec{\hat{x}}}$ is comparable to that achieved by a backward stable method.
Backward stability is a strictly stronger property than forward stability; backward stability implies (strong) forward stability, but not the other way around.
For more on stability for least-squares problems, standard references are \cite{Hig02,Bjo96}.

We now provide an overview of our main stability results.
Formal statements and proofs are deferred to \cref{sec:analysis,sec:FOSSILS-analysis,sec:SPIR-analysis}.
Our general stability result concerns the meta-algorithm \cref{alg:meta-algorithm}:

\begin{inftheorem}[Stability of meta-algorithm] \label{infthm:meta-stability}
    Assume the matrix $\mat{A}$ is numerically full-rank $\kappa u \ll 1$ and assume that the \Call{IterativeSolver}{} subroutine satisfies the stability guarantee
    \begin{equation} \label{eq:iterative-solver-stability}
        \norm{ \fl(\Call{IterativeSolver}{\vec{c}}) - (\Rhat^{-\top}\mat{A}^\top\mat{A}\Rhat^{-1})^{-1}\vec{c} } \lesssim \kappa u \cdot \norm{\vec{c}},
    \end{equation}
    where $\mat{\hat{R}} = \fl(\mat{R})$ denotes the numerically computed preconditioner.
    Then $\vec{x}_1$ computed by \cref{alg:meta-algorithm} is strongly forward stable, and the final output $\vec{x}_2$ is backward stable.
\end{inftheorem}

A formal statement of this result is provided in \cref{thm:meta-stability}, and a generalization to general, possibly non-randomized preconditioners is given in \cref{sec:generalizing}.
We anticipate our analysis could have implications for stability of other least-squares algorithms.

On its face, our result \cref{infthm:meta-stability} might seem to contradict Golub and Wilkinson's conclusion \cite{golub1966note} that iterative refinement is not generally effective for least-squares problems.
Our qualitative explanation is that SPIR and FOSSILS come with error that is \emph{structured}; in particular, the computation of $\mat{A}^\top(\vec{b} - \mat{A}\vec{x}_i)$ removes the components of the residual $\vec{b} - \mat{A}\vec{x}_i$ orthogonal to the column space of $\mat{A}$ to working accuracy.
This assumption is not satisfied in Golub and Wilkinson's setting.
We note that iterative refinement for least-squares has been recently revisited~\cite{carson2024comparison}.

We now turn to our two concrete proposals for fast, backward stable randomized least-squares solvers: SPIR and FOSSILS.
In light of \cref{infthm:meta-stability}, analyzing FOSSILS just requires proving \cref{eq:iterative-solver-stability} for the Polyak method, which we do in \cref{lem:polyak}.
Thus, we obtain:

\begin{inftheorem}[Backward stability of FOSSILS] \label{infthm:fossils}
    If the matrix $\mat{A}$ is numerically full-rank so that $\kappa u \ll 1$ and appropriate parameter choices are made, FOSSILS is backward stable and runs in $\order(mn \log (n/u) + n^3 \log n)$ operations or fewer.
\end{inftheorem}

Formally speaking, the numerical stability of different SPIR variants must be separately established for each different Krylov method (e.g., conjugate gradient, LSQR, etc.) used for iterative refinement.
Given that these variants are similarly stable in practice, we will only analyze SPIR for one choice of Krylov method, the Lanczos method for matrix inversion (see \cref{alg:lanczos}), which is equivalent to both conjugate gradient and LSQR in exact arithmetic \cite[p.~44]{Gre97a}.
Our choice of Lanczos-based SPIR as the target for analysis is pragmatic, as it allows us to employ existing analysis of the Lanczos method in floating point arithmetic \cite{MMS18}.
We prove the following backward stability guarantee for a Lanczos-based implementation of SPIR:

\begin{inftheorem}[Backward stability of SPIR] \label{infthm:spir}
    If the matrix $\mat{A}$ is numerically full-rank so that $\kappa u \ll 1$, $n \ge \log(1/u)$, and appropriate parameter choices are made, a Lanczos-based implementation of SPIR is backward stable and runs in $\order(mn \log (n/u) + n^3 \log n)$ operations or fewer.
\end{inftheorem}

Our results have implications beyond least-squares.
In \rr{subsequent work~\cite{epperly2025stable}, it is 
demonstrated} that an approach based on preconditioning the normal equations and iterative refinement can be used to give backward stable solutions to square, nonsymmetric systems of linear equations.
Note that, in this setting, more than two refinement steps may be needed to obtain backward stability, \rr{as a sketch-and-solve initialization is not available}.

\subsection{Our approach to proving backward stability}

We now briefly comment on how we prove \cref{infthm:meta-stability,infthm:fossils,infthm:spir}.
For direct \QR-based least-squares solvers, backward stability is established by directly analyzing the algorithm, interpreting each rounding error made during computation as a perturbation to the input.
For iterative methods, this approach breaks down, leading us to develop a different approach.
We are unaware of any existing work demonstrating backward stability for an iterative least-squares algorithm.

To establish backward stability, we will use the following characterization:

\begin{theorem}[Backward stability, componentwise errors] \label{thm:backward-componentwise}
    Consider the least-squares problem \cref{eq:ls} and normalize so that $\norm{\mat{A}} = \norm{\vec{b}} = 1$.
    Let $\mat{A} = \sum_{i=1}^n \sigma_i \vec{u}_i^{\vphantom{\top}}\vec{v}_i^\top$ be an \SVD.
    Then $\vec{\hat{x}}$ is a backward stable solution to \cref{eq:ls} if and only if the component of the error $\vec{\hat{x}} - \vec{x}$ in the direction of each singular vector satisfies the bound
    \begin{equation} \label{eq:componentwise-error}
        \left| \vec{v}_i^\top (\vec{\hat{x}} - \vec{x}) \right| \lesssim \sigma_i^{-1}\cdot \left(1+\norm{\vec{\hat{x}}}\right) u + \sigma_i^{-2} \cdot \norm{\vec{b} - \mat{A}\vec{\hat{x}}}u \quad \text{for } i =1,\ldots,n.
    \end{equation}
\end{theorem}

To the best of our knowledge, this result is new, though the derivation is short using the analysis of \cite{GJT12}.
See \cref{app:backward-componentwise} for a proof.

\subsection{\rr{Backward stability versus strong forward stability}}

\rr{In addition to being useful in our proof of stability,} \cref{thm:backward-componentwise} helps to highlight the difference between forward stability and \rr{backward} stability.
\rr{We have the following result, proven in \cref{sec:forward-componentwise}:}

\begin{proposition}[\rr{Strong forward stability, componentwise errors}] \label{cor:forward-componentwise}
    \rr{Use the normalization $\norm{\mat{A}} = \norm{\vec{b}} = 1$\rr{, assume that $\mat{A}$ is numerically full-rank $\cond(\mat{A})u\ll 1$, and let $\vec{\hat{x}} \in \real^n$.}
    (The notation $\ll$ is defined formally in \cref{sec:analysis}.)
    The following are equivalent:
    \begin{enumerate}[label=(\arabic*)]
        \item $\vec{\hat{x}}$ is strongly forward stable.
        \item The componentwise errors admit the following bound in $\vec{x}$:
        \begin{equation} \label{eq:strong-forward-componentwise}
        \left| \vec{v}_i^\top (\vec{\hat{x}} - \vec{x}) \right| \lesssim \sigma_i^{-1}\cdot \norm{\vec{x}} u + \sigma_i^{-1}\sigma_n^{-1} \cdot \norm{\vec{b} - \mat{A}\vec{x}}u \quad \text{for } i =1,\ldots,n.
    \end{equation}
        \item The componentwise errors admit the following bound in $\vec{\hat{x}}$: 
        \begin{equation} \label{eq:strong-forward-componentwise2}
        \left| \vec{v}_i^\top (\vec{\hat{x}} - \vec{x}) \right| \lesssim \sigma_i^{-1}\cdot \left(1+\norm{\vec{\hat{x}}}\right) u + \sigma_i^{-1}\sigma_n^{-1} \cdot \norm{\vec{b} - \mat{A}\vec{\hat{x}}}u \quad \text{for } i =1,\ldots,n.
    \end{equation}
    \end{enumerate}}
\end{proposition}

Comparing~\cref{eq:componentwise-error,eq:strong-forward-componentwise}, we see that a backward stable method has componentwise errors that are small in the direction of the dominant right singular vectors of $\mat{A}$, \rr{whereas a strongly forward stable method can have larger errors in these directions.
One consequence of this result is that it shows that a strongly forward stable method is backward stable unless the problem is both ill-conditioned and has a large residual.
We have the following result, proven in \cref{sec:strong-forward-enough}:
\begin{corollary}[When is strong forward stability enough?] \label{cor:strong-forward-enough}
    Enforce the normalization $\norm{\mat{A}} = \norm{\vec{b}} = 1$, assume $\mat{A}$ is numerically full-rank $\cond(\mat{A})u \ll 1$, and suppose $\cond(\mat{A}) \norm{\vec{b} - \mat{A}\vec{x}} \lesssim 1+\norm{\vec{x}}$.
    Then the solution produced by a strongly forward stable algorithm is backward stable.
\end{corollary}}

\rr{Taken together, \cref{cor:forward-componentwise,cor:strong-forward-enough} show that a method that is strongly forward stable is \emph{almost} backward stable, except that it produces large backward errors on ill-conditioned problems with large residual.
However, replacing existing backward stable solvers with randomized methods requires they work on \emph{all} problem instances, motivating us to develop the SPIR and FOSSILS methods that are unconditionally backward stable.}

\subsection{\rr{Informal explanations for stability results}}

\rr{The stability differences between different versions of sketch-and-precondition may seem mysterious.
Why is the method unstable with zero initialization, (strongly) forward stable with sketch-and-solve initialization, and backward stable after one step of iterative refinement?}

\rr{At base, the reason for these stability differences come from the size of the quantity 
\begin{equation}
\label{eq:c-def-qual}
    \vec{c} \coloneqq \mat{R}^{-\top}\mat{A}^\top (\vec{b} - \mat{A}\vec{z})
\end{equation}
for different values of $\vec{z}$.
When the preconditioned normal equations 
\begin{equation} \label{eq:pre-norm-qual}
    (\mat{R}^{-\top}\mat{A}^\top\mat{A}\mat{R}^{-1})\vec{\delta y} = \vec{c}
\end{equation}
are solved numerically by an iterative method, the solution is accurate up to a forward error $\lesssim \kappa u \norm{\vec{c}}$.
Thus, attaining numerical stability is associated with a small value of $\norm{\vec{c}}$.}

\rr{The norm of $\vec{c}$ shrinks the closer $\vec{z}$ is to being a solution to the least-squares problem \cref{eq:ls}.
Assume $\norm{\mat{A}} = \norm{\vec{b}} = 1$.
\begin{itemize}
    \item If $\vec{z} = \vec{x}$ is the exact solution, then $\vec{c} = \vec{0}$.
    \item If $\vec{z}$ is strongly forward stable, then $\norm{\vec{c}} \lesssim (\norm{\vec{x}} + \kappa \norm{\vec{b} - \mat{A}\vec{x}})u$.
    \item If $\vec{z}$ is the sketch-and-solve solution, then $\norm{\vec{c}} \lesssim \norm{\vec{b} - \mat{A}\vec{x}}$.
    \item If $\vec{z} = \vec{0}$ is the zero initialization, then $\norm{\vec{c}} \lesssim \norm{\vec{b}}$.
\end{itemize}
We see that the better the initialization $\vec{z}$, the smaller $\norm{\vec{c}}$ is and the smaller the smaller the numerical errors in solving \cref{eq:pre-norm-qual} are.
Thus, the norm of $\vec{c}$ provides a unifying explanation for the differing stability properties of sketch-and-precondition methods.
A numerical comparison for the norm of $\vec{c}$ with different initializations is provided in \cref{tab:norm-c}.}

\begin{table}[t]
    \centering
    \caption{\rr{\textbf{\textit{The norm of $\vec{c}$.}} Norm of vector $\vec{c}$ in \cref{eq:pre-norm-qual} for different initializations $\vec{z}$ on the least-squares problem from \cref{fig:comparison}.}}
    \begin{tabular}{lc}\toprule
    \textbf{Initialization} & $\norm{\vec{c}}$ \\\midrule
    Zero initialization $\vec{z} = \vec{0}$ & 2.0e-1 \\
    Sketch-and-solve initialization $\vec{z} = \vec{x}_0$ & 3.0e-7 \\
    Strongly forward stable initialization $\vec{z} = \vec{x}_1$ & 1.9e-11
    \\\bottomrule
    \end{tabular}
    \label{tab:norm-c}
\end{table}

\rr{Let us further comment on the zero initialization $\vec{x}_0 = \vec{0}$ in particular, since it is perhaps the most mysterious.
With $\vec{x}_0 = \vec{0}$, the solution after one LSQR step has an enormous norm, and thus error. 
This is a generic phenomenon with problems with exact solution $\|\vec{x}\|\ll \|\vec{b}\|/\sigma_{\rm min}(\mat{A})$, which is common.
To see why this happens, note that the approximate solution $\vec{\hat{x}}_{(1)}$ produced by one preconditioned LSQR step takes the form $\vec{\hat{x}}_{(1)}=\mat{R}^{-1}\vec{\hat{y}}_{(1)}$, where $\|\mat{R}^{-1}\|\approx 1/\sigma_{\rm min}(\mat{A})\gg 1/\sigma_{\rm max}(\mat{A})$ and $\vec{\hat{y}}_{(1)}$ is the first LSQR iterate.
Since only step of LSQR has been performed, $\vec{\hat{y}}_{(1)}$ is still far from the true solution and has hardly has any correct digits.
Thus, we can think of $\vec{\hat{y}}_{(1)}$ as being a ``generic'' (i.e., random) vector.
Hence $\|\vec{\hat{x}}_{(1)}\|=\|\mat{R}^{-1}\vec{\hat{y}}_{(1)}\|\approx \|\mat{R}^{-1}\|\|\vec{\hat{y}}_{(1)}\|$ has a large norm, as claimed, leading to large errors.}

\section{Related work} \label{sec:related}

This section provides an overview of existing fast randomized least-squares solvers and their stability properties.

\subsection{Fast randomized least-squares solvers} \label{sec:sketch-and-precondition}

By now, there are a number of fast, randomized least-squares solvers based on Rokhlin and Tygert's randomized preconditioning idea (\cref{sec:randomized-preconditioned}).

\textbf{Sketch-and-precondition} \cite{RT08,AMT10} refers to methods which solve the least-squares problem \cref{eq:ls} using a Krylov method with the randomized preconditioner $\mat{R}$ in \cref{eq:sketch_qr}.
Rokhlin and Tygert's original proposal used conjugate gradient as the Krylov method \cite{RT08} and corresponds to \cref{alg:SPIR-basic} with only one execution of the refinement loop.

The sketch-and-precondition was popularized by the paper \cite{AMT10} under the name Blendenpik, which proposed a streamlined method in which the least-squares problem \cref{eq:ls} is solved by LSQR with the randomized preconditioner $\mat{R}$.
(The sketch-and-solve initialization \cref{eq:sketch-and-solve-2.3}, necessary for the method to be forward stable, is included as an optional setting in their code.)
There are a few more variants of sketch-and-precondition.
LSRN \cite{MSM14} is a sketch-and-precondition variant that uses the Chebyshev semi-iterative method, a Gaussian embedding $\mat{S}$, and $\mat{R} = \mat{\Sigma} \mat{V}^\top$ computed from an \SVD $\mat{S}\mat{A} = \mat{U}\mat{\Sigma}\mat{V}^\top$ as the preconditioner. 
The paper \cite{CFS21} developed a sketch-and-precondition variant using a general orthogonal factorization of $\mat{S}\mat{A}$ and released an implementation \texttt{Ski-LLS}.

In 2016, Pilanci and Wainwright \cite{PW16} developed a new class of fast randomized least-squares solvers known as \textbf{iterative sketching} methods.
In their original form, iterative sketching methods were dissimilar from sketch-and-precondition methods, though refinements \cite{OPA19,LP21} have brought the two approaches closer together.
The basic iterative sketching method can be derived by applying \emph{gradient descent} to the preconditioned system \cref{eq:ls-precond}: 
\begin{equation*}
    \vec{y}_{(j+1)} = \vec{y}_{(j)} + \mat{R}^{-\top} \mat{A}^\top (\vec{b} - \mat{A} (\mat{R}^{-1}\vec{y}_{(j)})) \quad \text{for } j=0,1,2,\ldots.
\end{equation*}
Transforming $\vec{x}_{(j)} = \mat{R}^{-1} \vec{y}_{(j)}$, we obtain the iteration for the basic iterative sketching method:
\begin{equation} \label{eq:iterative-sketching-basic}
    \vec{x}_{(j+1)} = \vec{x}_{(j)} + \mat{R}^{-1}(\mat{R}^{-\top} \mat{A}^\top (\vec{b} - \mat{A} \vec{x}_{(j)})) \quad \text{for } j=0,1,2,\ldots.
\end{equation}
The \QR factorization-based implementation of iterative sketching \cref{eq:iterative-sketching-basic} was developed in a sequence of papers \cite{PW16,OPA19,LP21,Epp24}.
We emphasize that the iterative sketching method presented only multiplies $\mat{A}$ by a sketching matrix once, making the name ``iterative sketching'' for this method something of a misnomer.

The basic iterative sketching method has a much slower convergence rate than sketch-and-precondition, and it requires a large embedding dimension of at least $d \approx 12n$ for the method to be guaranteed convergence to a forward stable solution.
(The method diverges \emph{in practice} when $d \ll 12n$. \rr{We note that all other randomized iterative methods are seen empirically to converge geometrically with an appropriate choice of sketching matrix and $d=cn$ for any $c > 1$, though the rate of convergence worsens with small $c$.})

\rr{To remedy these issues, concurrent} works \cite{OPA19,LP21} proposed accelerating iterative sketching with \emph{momentum}:
\begin{equation} \label{eq:ism}
    \vec{x}_{(j+1)} = \vec{x}_{(j)} + \alpha \mat{R}^{-1} (\mat{R}^{-\top} \mat{A}^\top (\vec{b} - \mat{A} \vec{x}_{(j)})) + \beta (\vec{x}_{(j)} - \vec{x}_{(j-1)}) \quad \text{for } i=0,1,2,\ldots.
\end{equation}
The optimal values of $\alpha$ and $\beta$ are given by \cref{eq:optimal-coeffs}.
Iterative sketching with momentum is known to be an optimal first-order method for least-squares in an appropriate sense \cite{LP20}.

Our proposed backward stable randomized least-squares solvers SPIR and FOSSILS fit neatly into the sketch-and-precondition/iterative sketching classification, with FOSSILS being an iterative sketching-type method and SPIR being a sketch-and-precondition method.
The main innovation of our methods is the use of iterative refinement to obtain backward stability.
We also note that FOSSILS applies the Polyak method to the normal equations \cref{eq:precond-normal}, whereas iterative sketching with momentum applies the preconditioned Polyak heavy ball method to the original least-squares problem \cref{eq:ls}; this difference also has effects on numerical stability.

\subsection{Numerical stability of randomized least-squares solvers} \label{sec:existing-stability-least-squares}

Critical investigation of the numerical stability of randomized least-squares solvers was initiated by Meier, Nakatsukasa, Townsend, and Webb \cite{MNTW24}, who demonstrated that \rr{sketch-and-precondition is not backward stable.
Their numerical results showed that, with trivial initialization $\vec{x}_0 = \vec{0}$, sketch-and-precondition is not even forward stable, and they noted that the sketch-and-solve initialization (as suggested in~\cite{RT08}) improves the stability properties considerably, but not to the level of backward stability.}
This finding is reproduced in \cref{fig:comparison}.
To remedy this instability, Meier et al.\ introduced sketch-and-apply, a backward stable version of sketch-and-precondition that runs in $\order(mn^2)$ operations, the same asymptotic cost as Householder \QR.

Meier et al.'s work left open the question of whether \emph{any} least-squares solvers with a fast $\tilde{\order}(mn + n^3)$ runtime was stable, even in the forward sense.
This question was resolved by Epperly, who proved that iterative sketching is strongly forward stable \cite{Epp24}.
\rr{Taken together, the numerical results of \cite{MNTW24,Epp24} suggested that sketch-and-precondition with the sketch-and-solve initialization is strongly forward stable.}
\rr{Empirically, we find there is a class of problems with large residual and moderately large condition number for which neither iterative sketching nor sketch-and-precondition algorithms provide backward stable solutions. This region can be observed in the right panel of~\cref{fig:grids_fossils}.}

\rr{The current paper effectively resolves the line of inquiry of \cite{MNTW24,Epp24} by establishing the existence of a fast backward stable least-squares solver and rigorously analyzing the stability properties of sketch-and-precondition with different forms of initialization (\cref{thm:spir-stability}).}

\section{SPIR and FOSSILS: Implementation guidance} \label{sec:FOSSILS}

This section begins by discussing posterior estimates for the backward error (\cref{sec:qa}).
After which, we present implementations of SPIR and FOSSILS that use these estimates to adaptively determine the number of algorithm steps (\cref{sec:implementation}).
We also comment on other implementation details, such as the choice of random embedding, detecting and mitigating numerical rank deficiency, and column scaling.
\Cref{sec:lsqr-vs-cg} comments on the choice of LSQR versus conjugate gradient as the Krylov method for SPIR, and \cref{sec:fossils-vs-spir} provides a comparison of SPIR and FOSSILS.

\subsection{Quality assurance} \label{sec:qa}

In order to deploy randomized algorithms safely in practice, it is best practice to equip them with posterior error estimates \cite{ET24}.
In this section, we will provide fast posterior estimates for the backward error to a least-squares problem.
These estimates can be used to verify, at runtime, whether FOSSILS, SPIR, or any other solver has produced a backward stable solution to a least-squares problem.
We will use this estimate in \cref{sec:implementation} to adaptively determine the number of SPIR or FOSSILS iterations required to achieve convergence to machine-precision backward error.

\rr{Generalizing from \cref{def:backward},} the backward error for a computed solution $\vec{\hat{x}}$ to a least-squares problem \cref{eq:ls} \rr{may be} defined \rr{as}
\begin{equation} \label{eq:backward_error}
    \BE_\theta(\vec{\hat{x}}) \coloneqq \min \left\{ \norm{\onebytwo{\mat{\Delta A}}{\theta\cdot \vec{\Delta b}}}_{\rm F} : \vec{\hat{x}} = \argmin_{\vec{\hat{x}}} \norm{(\vec{b}+\vec{\Delta b}) - (\mat{A}+\mat{\Delta A})\vec{\hat{x}}\}} \right\}.
\end{equation}
\rr{This formula generalizes \cref{def:backward} by allowing perturbations to both $\mat{A}$ and $\vec{b}$, whose relative influence is measured by a} parameter $\theta \in (0,\infty]$.
Under the normalization $\norm{\mat{A}} = \norm{\vec{b}} = 1$, a method is backward stable if and only if the output $\vec{\hat{x}}$ satisfies $\BE_1(\vec{\hat{x}}) \lesssim u$.

We can compute $\BE_\theta(\vec{\hat{x}})$ stably in $\order(m^3)$ operations using the Wald\'en--Karlson--Sun--Higham formula \cite[eq.~(1.14)]{WKS95}.
This formula is expensive to compute, leading Karlson and Wald\'en \cite{KW97} and many others \cite{Gu98,Grc03,GSS07} to investigate easier-to-compute proxies for the backward error.
This series of works resulted in the following estimate \cite[sec.~2]{GSS07}:
\begin{equation*}
    \hat{\BE}_\theta(\vec{\hat{x}}) \coloneqq \frac{\theta}{\sqrt{1+\theta^2\norm{\vec{\hat{x}}}^2}} \norm{\left( \mat{A}^\top\mat{A} + \frac{\theta^2\norm{\vec{b} - \mat{A}\vec{\hat{x}}}^2}{1+\theta^2\norm{\vec{\hat{x}}}^2} \Id \right)^{-1/2} \mat{A}^\top (\vec{b} - \mat{A}\vec{\hat{x}})}.
\end{equation*}
We will call $\hat{\BE}_\theta(\vec{\hat{x}})$ the \emph{Karlson--Wald\'en estimate} for the backward error.
The Karlsen--Wald\'en estimate can be evaluated in $\order(mn^2)$ operations or, with a \QR factorization or \SVD of $\mat{A}$, only $\order(mn)$ operations \cite{GSS07}.
The quality of the Karlson--Walden\'en estimate was precisely characterized by Gratton, Jir\'anek, and Titley-Peloquin \cite{GJT12}, who showed the following:

\begin{fact}[Karlson--Wald\'en estimate] \label{fact:KW-estimate}
    The Karlson--Wald\'en estimate $\hat{\BE}(\vec{\hat{x}})$ agrees with the backward error $\BE(\vec{\hat{x}})$ up to a constant factor $\hat{\BE}_\theta(\vec{\hat{x}}) \le \BE_\theta(\vec{\hat{x}}) \le \sqrt{2}\cdot \hat{\BE}_\theta(\vec{\hat{x}})$.
\end{fact}

For randomized least-squares solvers, the $\order(mn^2)$ cost of the Karlson--Wald\'en estimate is still prohibitive.
To address this challenge, we propose the following \emph{sketched Karlson--Wald\'en estimate}:
\begin{equation} \label{eq:sketched-kw}
    \hat{\BE}_{\theta,{\rm sk}}(\vec{\hat{x}}) \coloneqq \frac{\theta}{\sqrt{1+\theta^2\norm{\vec{\hat{x}}}^2}} \norm{\left( (\mat{S}\mat{A})^\top(\mat{S}\mat{A}) + \frac{\theta^2\norm{\vec{b} - \mat{A}\vec{\hat{x}}}^2}{1+\theta^2\norm{\vec{\hat{x}}}^2} \Id \right)^{-1/2} \mat{A}^\top (\vec{b} - \mat{A}\vec{\hat{x}})}.
\end{equation}
This estimate satisfies the following bound, proven in \cref{app:sketched-KW-proof}:

\begin{proposition}[Sketched Karlson--Wald\'en estimate] \label{prop:sketched-KW}
    Let $\mat{S}$ be a sketching matrix for $\mat{A}$ with distortion $0 < \eta < 1$.
    The sketched Karlson--Wald\'en estimate $\hat{\BE}_{\theta,\mathrm{sk}}(\vec{\hat{x}})$ satisfies the bounds
    \begin{equation*}
        (1-\eta)\cdot \hat{\BE}_{\theta,{\rm sk}}(\vec{\hat{x}}) \le \BE_\theta(\vec{\hat{x}}) \le \sqrt{2}(1+\eta) \cdot \hat{\BE}_{\theta,{\rm sk}}(\vec{\hat{x}}).
    \end{equation*}
\end{proposition}

To evaluate $\hat{\BE}_{\theta,{\rm sk}}(\vec{\hat{x}})$ efficiently in practice, we compute an \SVD of the sketched matrix
\begin{equation} \label{eq:sketch_svd}
    \mat{S}\mat{A} = \mat{U}\mat{\Sigma} \mat{V}^\top,
\end{equation}
from which $\hat{\BE}_{\theta,{\rm sk}}(\vec{\hat{x}})$ can be computed for any $\vec{\hat{x}}$ in $\order(mn)$ operations using the formula
\begin{equation*}
    \hat{\BE}_{\theta,{\rm sk}}(\vec{\hat{x}}) = \frac{\theta}{\sqrt{1+\theta^2\norm{\vec{\hat{x}}}^2}} \norm{\left( \mat{\Sigma}^2 + \frac{\theta^2\norm{\vec{b} - \mat{A}\vec{\hat{x}}}^2}{1+\theta^2\norm{\vec{\hat{x}}}^2} \Id \right)^{-1/2} \mat{V}^\top \mat{A}^\top (\vec{b} - \mat{A}\vec{\hat{x}})}.
\end{equation*}

\rr{A limitation of the sketched Karlson--Wald\'en is that the precise numerical value of the distortion $\eta$ is hard to estimate at runtime.
Fortunately, well-designed random embeddings such as sparse sign embeddings \cite[sec.~9.2]{MT20} or randomized trigonometric transforms \cite[sec.~9.3]{MT20} (with rerandomization \cite[Rem.~9.2]{MT20}) have been \emph{extensively} tested and achieve a distortion $\eta \approx \sqrt{n/d}$ with neglible failure probability.
Provided one uses one of these types of embeddings and sets the embedding dimension sufficiently high (say, $d\ge 2n$), we find the sketched Karlson--Wald\'en estimate produces extremely reliable estimates of the backward error (accurate up to a small multiple).}

\subsection{General implementation recommendations} \label{sec:implementation}

In this section, we provide guidance for implementing SPIR and FOSSILS and propose a few tweaks to make SPIR and FOSSILS as efficient as possible in practice.
As always, specific implementation details should be tuned to the available computing resources and the format of the problem data (e.g., dense, sparse, etc.).
Automatic tuning of randomized least-squares solvers was considered in the recent paper \cite{CDD+23}.
We combine all of the implementation recommendations from this section into a FOSSILS implementation \cref{alg:FOSSILS-recommended}.
An implementation of SPIR using these principles is similar and is omitted for space.

\paragraph{Column scaling.}
As a simple pre-processing step, we recommend beginning SPIR or FOSSILS by scaling the matrix $\mat{A}$ so that all of its columns have $\ell_2$ norm one.
This scaling can improve the conditioning of $\mat{A}$, producing an optimal conditioning  up to a factor of $ \sqrt{n}$ among all possible column scalings (see \cite{sluis1969condition}, \cite[Thm.~7.5]{Hig02}).
Scaling the columns in this way ensures that SPIR and FOSSILS are \emph{columnwise} backward stable \cite[p.~385]{Hig02}. 

\paragraph{Choice of embedding.}
For computational efficiency, we must use a sketching matrix with a fast multiply operation, of which there are many \cite[sec.~9]{MT20}.
Based on the timing experiments in \cite{DM23,Epp25}, we recommend using \emph{sparse sign embeddings} \cite[sec.~9.2]{MT20}, i.e., a matrix of the form
\begin{equation*}
    \mat{S} = \zeta^{-1/2} \begin{bmatrix}
        \vec{s}_1 & \cdots & \vec{s}_m
    \end{bmatrix},
\end{equation*}
whose columns $\vec{s}_1,\ldots,\vec{s}_m \in \real^m$ are independent random vectors with $\zeta$ nonzero entries placed in uniformly random positions.
The nonzero entries of $\vec{s}_i$ take the values $\pm 1$ with equal probability.
Cohen \cite{Coh16} showed that a sparse sign embedding is a sketching matrix for $\mat{A} \in \real^{m\times n}$ provided
\begin{equation} \label{eq:cohen}
    d \ge \mathrm{const} \cdot \frac{n \log n}{\eta^2} \quad \text{and} \quad \zeta \ge \mathrm{const} \cdot \frac{\log n}{\eta}.
\end{equation}
See also \cite{CDDR23,CDD25,Tro25,CEMT25} for more recent analyses of sparse sketching matrices with different tradeoffs between $d$ and $\zeta$.
The paper \cite{CFS21} shows that sparse sign embeddings with a small constant value of $\zeta$ are effective for incoherent matrices.
(This incoherence assumption can be enforced algorithmically by first applying a randomized trigonometric transform \cite[sec.~4.3]{CFS21}, but this can have significant effects on the runtime \cite{DM23,Epp25}.)

Following \cite[sec.~3.3]{TYUC19}, we use $\zeta = 8$ nonzeros per column in our experiments.
In our experiments with dense matrices of size $m = 10^5$ and $10^1 \le n \le 10^3$, we found the runtime to not be very sensitive to the embedding dimension $d \in [10n,30n]$.
We find the value $d = 12n$ tends to work across a range of $n$ values.

In order to set the parameters $\alpha$ and $\beta$ for FOSSILS, we need a numeric value for the distortion $\eta$ of the embedding.
Here, a valuable heuristic---justified rigorously for Gaussian embeddings---is $\eta \approx \sqrt{n/d}$.
Therefore, we use $\eta = \sqrt{n/d}$ to choose $\alpha$ and $\beta$ in our experiments.
We remark that, for smaller values of $d$ (e.g., $d=4n$), this value of $\eta$ sometimes led to issues; a slightly larger value $\eta = 1.1\sqrt{n/d}$ fixed these problems.

\paragraph{Replacing \QR by \SVD.} 
To implement the posterior error estimates from \cref{sec:qa}, we need an \SVD $\mat{S}\mat{A} = \mat{U}\mat{\Sigma}\mat{V}^\top$.
To avoid computing two factorizations of $\mat{S}\mat{A}$, we recommend using an \SVD-based implementation of SPIR and FOSSILS, removing the need to ever compute a \QR factorization of $\mat{S}\mat{A}$.
To do this, observe that \cref{fact:whitening} holds for any orthogonal factorization $\mat{S}\mat{A} = \mat{Q}\mat{R}$.
Thus, an \SVD variant of FOSSILS using  $\mat{Q} = \mat{U}$ and $\mat{R} = \mat{\Sigma}\mat{V}^\top$ achieves the same rate of convergence as the original \QR-based SPIR and FOSSILS procedures. 
In our experience, the possible increased cost of the \SVD over \QR factorization pays for itself by allowing us to terminate SPIR and FOSSILS early with a posterior guarantee for backward stability.
(We note that our analysis is specialized to a \QR-based implementation; \QR and \SVD implementations appear to be similarly stable in practice.)
An \SVD-based preconditioner is also used in LSRN~\cite{MSM14}.

\paragraph{Adaptive stopping.}
Finally, to optimize the runtime and to assure backward stability, we can use posterior error estimates to adaptively stop the iterative solver.
We need a different stopping criteria for each of the two refinement steps:
\begin{enumerate}
    \item As will become clear from the analysis in \cref{sec:analysis}, the \emph{backward stability} of SPIR and FOSSILS requires $\vec{x}_1$ produced during the first refinement step to be \emph{ strongly forward stable}.
    Therefore, we can use a variant of the stopping criterion from \cite[sec.~3.3]{Epp24} to stop the inner solver during the first refinement step: when performing an update $\vec{\delta y}_{(j+1)} \gets \vec{\delta y}_{(j)} + \vec{\Delta}$,
    \begin{equation*}
        \text{STOP} \quad \text{when} \quad \norm{\vec{\Delta}} \le u(\gamma \cdot \norm{\mat{\Sigma}} \norm{\vec{x}_0} + \rho\cdot\cond(\mat{\Sigma})\norm{\vec{\vec{b}-\mat{A}\vec{x}_0}}).
    \end{equation*}
    We use $\gamma = 1$ and $\rho = 0.04$ in our experiments.
    \item At the end of the second refinement step, we demand that the computed solution be backward stable.
    We stop the iteration when the sketched Karlson--Wald\'en estimate of the backward error of the solution $\vec{x}_2$ with parameter $\theta = \norm{\mat{A}}_{\rm F} / \norm{\vec{b}}$ is below $u\norm{\mat{A}}_{\rm F}$.
    To balance the expense of forming the error estimate against the benefits of early stopping, we recommend computing the Karlson--Wald\'en estimate only once every five iterations. 
\end{enumerate}

\paragraph{Numerical rank deficiency.}
In floating point arithmetic, randomized iterative methods can fail catastrophically for problems that are numerically rank-deficient (i.e., $\kappa \gtrsim u^{-1}$).
Indeed, without modification, sketch-and-precondition, iterative sketching, SPIR, and FOSSILS all exit with \texttt{NaN}'s when run on a matrix of all ones, as does Householder \QR itself.

Ultimately, it is a question of software design---not of mathematics---what to do when a user requests a least-squares solver to solve a least-squares problem with a numerically rank-deficient $\mat{A}$.
For SPIR and FOSSILS, we adopt the following approach to deal with numerically rank-deficient problems.
First, $\cond(\mat{\Sigma})$ is a good proxy for the condition number $\kappa$ of $\mat{A}$ \cite[Fact~2.3]{NT24}:
\begin{equation*}
    \frac{1-\eta}{1+\eta} \cdot \cond(\mat{\Sigma})\le \kappa \le \frac{1+\eta}{1-\eta} \cdot \cond(\mat{\Sigma}).
\end{equation*}
Therefore, we can test for numerical rank deficiency at runtime.
If a problem is determined to be numerically rank-deficient, one can fall back on a column pivoted \QR- or \SVD-based solver, if desired. \rr{Considering the necessity of the SVD for quality assurance at run-time, we implement a truncated SVD-based preconditioner. In particular, having computed the SVD of the sketch $\mat{SA} = \mat{U_1\Sigma_1 V_1^\top} + \mat{U_2\Sigma_2 V_2^\top}$, where the singular values of $\mat{\Sigma_1}$ are all greater than some small constant $\tau$ (we set $\tau = 30u$), the inverted preconditioner is $\mat{V_1\Sigma_1^{-1}}$. This is also suggested in the generic sketch-and-precondition algorithm presented in~\cite{CFS21}.}
Alternatively, \rrrr{if the user still wants a fast least-squares solution, computed by SPIR or FOSSILS,} we can solve a regularized least-squares problem
\begin{equation*}
    \vec{\tilde{x}} = \argmin_{\vec{x} \in \real^n} \norm{\vec{b} - \mat{A} \vec{x}}^2 + \mu^2 \norm{\vec{x}}^2,
\end{equation*}
where $\mu > 0$ is a regularization parameter chosen to make this problem numerically well-posed \rr{(say, $\mu \coloneqq 10 \norm{\mat{A}}_{\rm F} u$)}. One can instead regularize by adding a Gaussian perturbation to $\mat{A}$ of size $\order(u\norm{\mat{A}})$ \cite{MNTW24}, but this may destroy benefitial structure such as sparsity.

\rr{We find that the truncated preconditioner approach results in slightly smaller forward errors than the direct regularization approach, while maintaining backward stability. }
See \cref{alg:FOSSILS-recommended} for \rr{implementation} details.

\begin{algorithm}[!th]
	\caption{FOSSILS: Recommended implementation} \label{alg:FOSSILS-recommended}
	\begin{algorithmic}[1]
		\Require Matrix $\mat{A}\in\real^{m\times n}$, right-hand side $\vec{b}\in\real^m$
		\Ensure Backward stable least-squares solution $\vec{x}\in\real^n$
		\State $\mat{d} \gets \Call{ColumnNorms}{\mat{A}}$, $\mat{A} \gets \mat{A}\diag(\vec{d})^{-1}$ \Comment{Normalize columns}
		\State $d\gets 12n$, $\zeta \gets 8$ \Comment{Parameters for embedding}
		\State $\beta \gets d/n$, $\alpha\gets(1-\beta)^2$ \Comment{Or $\beta = 1.1 \cdot d/n$, to be safe}
		\State $\mat{S} \gets \Call{SparseSignEmbedding}{m,d,\zeta}$
		\State $(\mat{U}_{\rr{\rm full}},\mat{\Sigma}_{\rr{\rm full}},\mat{V}_{\rr{\rm full}})\gets \Call{SVD}{\mat{S}\mat{A}}$ \Comment{Sketch and \SVD}
		\State $\mathrm{normest} \gets \mat{\Sigma}_{\rr{\rm full}}(1,1)$, $\mathrm{condest} \gets \mat{\Sigma}_{\rr{\rm full}}(1,1) / \mat{\Sigma}_{\rr{\rm full}}(n,n)$
		\If{$\mathrm{condest} > 1/(30u)$}
		  \State Print ``Warning! Condition number estimate is $\langle\mathrm{condest}\rangle$''
            \State \rr{$\mathrm{rank} \gets |\{\mathrm{diag}(\mat{\Sigma}_{\rm full}) > 30u\cdot \mathrm{normest}\}|$} \Comment{Truncate preconditioner if rank deficient}
            \State \rr{$\mat{U} \gets \mat{U}_{\rr{\rm full}}(:,\,1:\mathrm{rank}), \mat{\Sigma} \gets \mat{\Sigma}_{\rr{\rm full}}(1:\mathrm{rank},\,1:\mathrm{rank}), \mat{V} \gets \mat{V}_{\rr{\rm full}}(:,\,1:\mathrm{rank})$ }
		\Else 
		\State \rr{$\mat{U} \gets \mat{U}_{\rr{\rm full}},\mat{\Sigma} \gets \mat{\Sigma}_{\rr{\rm full}},\mat{V} \gets \mat{V}_{\rr{\rm full}}$}
		\EndIf
		\State $\vec{x}\gets \mat{V}(\mat{\Sigma}^{-1}(\mat{U}^\top (\mat{S}\vec{b})))$ 
		\Comment{Sketch-and-solve initialization}
		\State $\vec{r} \gets \vec{b} - \mat{A}\vec{x}$
		\State $\mat{P} \gets \mat{V}\mat{\Sigma}^{-1}$ \Comment{Preconditioner, taking the role of $\mat{R}^{-1}$}
		\For{$i=0,1$} \Comment{Two steps of iterative refinement with inner solver}
		\State $\vec{\delta y}, \vec{\delta y}_{\rm old},\vec{c} \gets \mat{P}^{\top}(\mat{A}^\top(\vec{b} - \mat{A}\vec{x}_i)$
		\For{$j=0,1,\ldots,99$} \Comment{Polyak heavy ball method, 100 iterations max}
		\State $\vec{z} \gets \mat{P}\vec{\delta y}$
		\State $\vec{\Delta} \gets \alpha(\vec{c} - \mat{P}^{\top}(\mat{A}^\top (\mat{A}\vec{z})) + \mu^2 \vec{z}) + \beta(\vec{\delta y} - \vec{\delta y}_{\rm old})$
		\State $(\vec{\delta y}, \vec{\delta y}_{\rm old}) \gets (\vec{\delta y} + \vec{\Delta},\vec{\delta y})$
		\If{$i = 0$ and $\norm{\vec{\Delta}} \le (10 \cdot \mathrm{normest} \cdot \norm{\vec{x}} + 0.4 \cdot \mathrm{condest} \cdot \norm{\vec{r}})u$} 
		\State \textbf{break} \Comment{Break after update is small during first inner solve}
		\ElsIf{$i=1$ and $\Call{Mod}{i,5} = 0$} \Comment{Check backward error every five iterations}
		\State $\vec{\hat{x}} \gets \vec{x} + \mat{P}\vec{\delta y}$
		\State $\vec{\hat{r}} \gets \vec{b} - \mat{A}\vec{\hat{x}}$
		\State $\theta \gets \norm{\mat{A}}_{\rm F}/\norm{\vec{b}}$
		\State $\hat{\mathrm{BE}}_{\rm sk} \gets \tfrac{\theta}{\sqrt{1+\theta^2\norm{\vec{\hat{x}}}^2}} \cdot \norm{ (\mat{\Sigma}_{\rr{\rm full}}^2+\tfrac{\theta^2\norm{\vec{\hat{r}}}^2}{1+\theta^2\norm{\vec{\hat{x}}}^2}\Id)^{-1/2}\mat{V}_{\rr{\rm full}}^\top(\mat{A}^\top\vec{\hat{r}}) }$ \Comment{Sketched KW \cref{eq:sketched-kw}}
		\If{$\hat{\mathrm{BE}}_{\rm sk} < \norm{\mat{A}}_{\rm F}u$}
		\State \textbf{break} \Comment{Break when backward stability achieved during second inner solve}
		\EndIf
		\EndIf
		\EndFor
		\State $\vec{x} \gets \vec{x} + \mat{P}\,\vec{\delta y}$ \Comment{Iterative refinement}
		\EndFor
		\State $\vec{x} \gets \diag(\vec{d})^{-1}\vec{x}$ \Comment{Undo column scaling}
	\end{algorithmic}
\end{algorithm}

\subsection{SPIR: Conjugate gradient or LSQR?}
\label{sec:lsqr-vs-cg}

In \cref{sec:algs}, we proposed a conjugate gradient-based implementation of SPIR (\cref{alg:SPIR-basic}).
This CG-based implementation is convenient to analyze because it is covered by the general theoretical analysis developed in \cref{sec:analysis}, allowing us to give a partial analysis of SPIR in \cref{sec:SPIR-analysis}.
In practice, however, an LSQR-based implementation of SPIR---in which \cref{eq:ls-itref} is solved by LSQR with preconditioner $\mat{R}$---seems to give an implementation of SPIR with similar stability properties to the CG-based implementation.

\begin{figure}[t]
  \centering
  \includegraphics[width=0.45\textwidth]{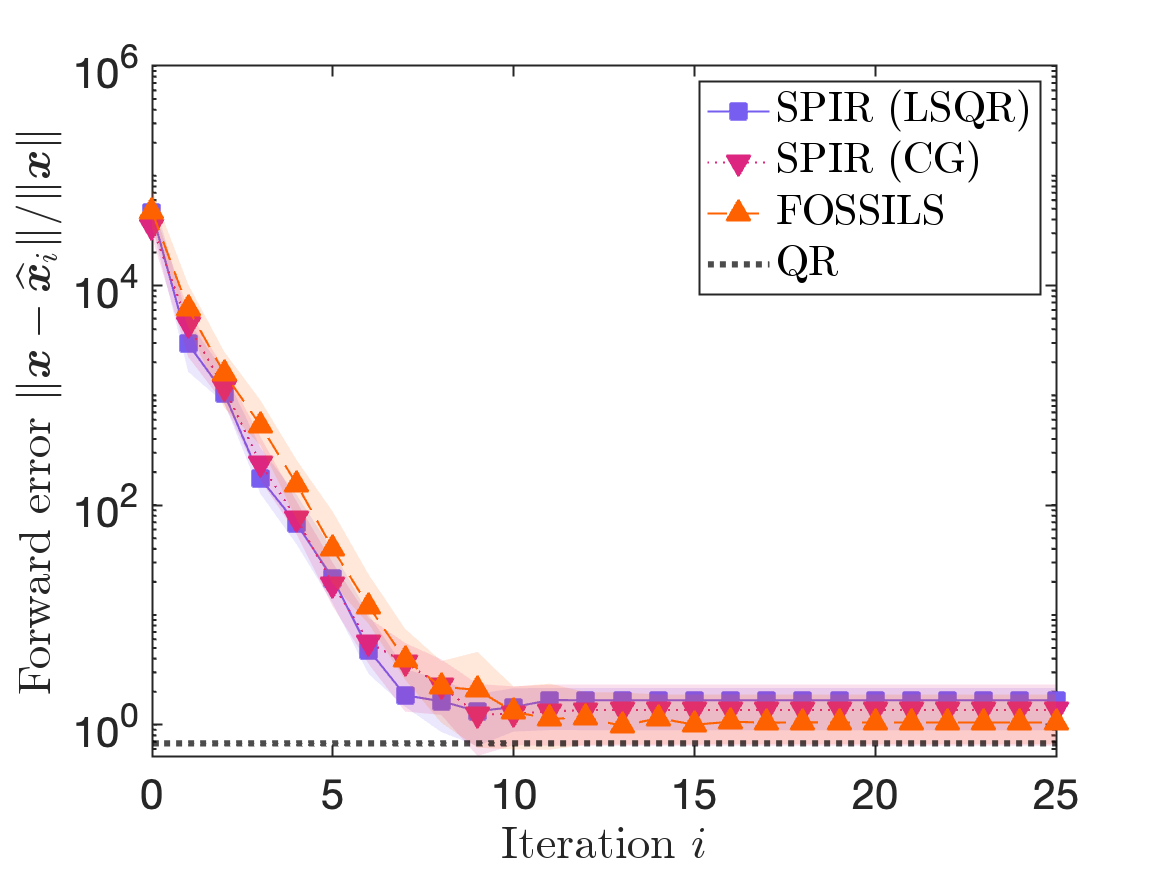}
  \includegraphics[width=0.45\textwidth]{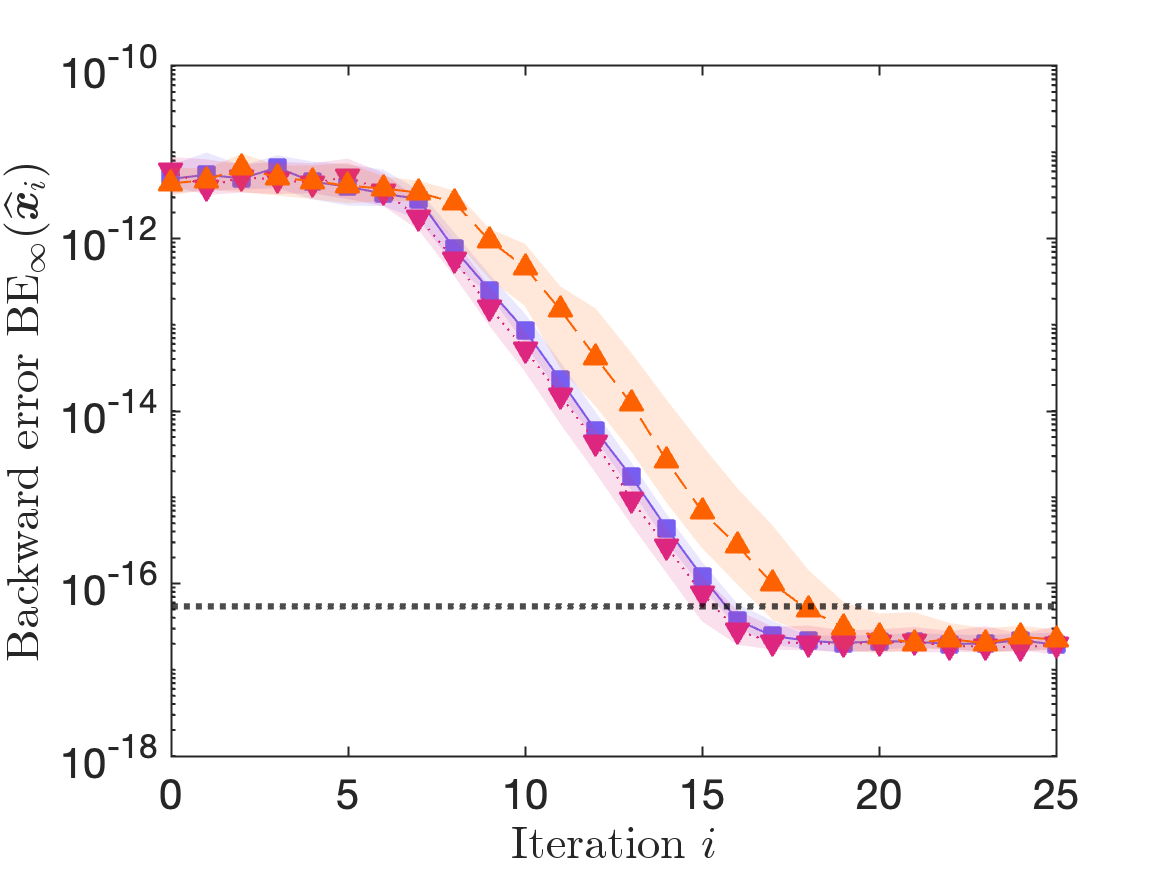} 
    
  \caption{Median forward (\emph{left}) and backward (\emph{right}) error for FOSSILS and LSQR- and CG-based versions of SPIR on the example from \cref{fig:comparison} over 25 trials, along with reference errors for Householder \QR.
  All methods use $d \coloneqq 12n$.
  Shaded regions show 20\% and 80\% quantiles.} \label{fig:spir}
\end{figure}

\Cref{fig:spir} shows the median forward (\emph{left}) and backward (\emph{right}) errors for CG- and LSQR-based implementations of SPIR.
20\% and 80\% quantiles are shown as shaded regions.
The behavior of the CG- and LSQR-based implementations are essentially indistinguishable in this example.
With small values of the embedding $d$ (e.g., $d=2n$), we have found slightly better backward error for the LSQR-based implementation.

\subsection{SPIR vs.\ FOSSILS: Which to use?}
\label{sec:fossils-vs-spir}

Which algorithm should you use: SPIR or FOSSILS?
With optimal parameters \cref{eq:optimal-coeffs}, both SPIR and FOSSILS achieve the same optimal rate of convergence (in both theory and practice).
The methods have strengths and weaknesses relative to each other.

\begin{itemize}
    \item \textbf{Benefits of FOSSILS over SPIR.} FOSSILS is simpler and easier to parallelize than SPIR.
    (See \cite{MSM14} for a discussion of parallelization.)
    Additionally, we only have a formal proof of backward stability for a Lanczos-based implementation of SPIR, so---strictly speaking---the backward stability of CG- and LSQR-based implementations of SPIR is only conjectural.
    \item \textbf{Benefits of SPIR over FOSSILS.} 
    The main weakness of FOSSILS is the need to know the distortion $\eta$ of the embedding to set the parameters $\alpha$ and $\beta$.
    The method can be slowly or even non-convergent if these parameters are set incorrectly.
    Fortunately, using the value $\eta = \sqrt{n/d}$ works well for sufficiently large values of $d$, with larger values of $\eta$ (i.e., $\eta = 1.1\sqrt{n/d}$) giving more robustness but also somewhat more slower convergence.
    By contrast, SPIR \rr{avoids} these issues \rr{and is reliable even when $d$ is set to a small multiple of $n$}.
\end{itemize}
Ultimately, we view both methods as promising candidates for deployment in software, but believe that SPIR is likely the method of choice except in certain situations.

\section{Numerical experiments} \label{sec:experiments}

In this section, we present numerical experiments demonstrating FOSSILS\rr{' and SPIR's} speed, reliability, and stability.
Code for all experiments can be found at
\actionbox{\centering \url{https://github.com/eepperly/Stable-Preconditioned-Solvers}}

\subsection{Confirming stability}

\Cref{fig:comparison_2} investigates the stability of randomized least-squares solvers for problems of increasing difficulty.
As Wedin's theorem (\cref{fact:wedin}) shows, the sensitivity of a least-squares problem depends on the condition number $\kappa$ and residual norm $\norm{\vec{b}-\mat{A}\vec{x}}$.
To generate a family of problems of increasing difficulty, we consider a sequence of randomly generated least-squares problems with 
\begin{equation} \label{eq:difficulty}
    \mathrm{difficulty} = \kappa = \frac{\norm{\vec{b}-\mat{A}\vec{x}}}{u} \in [10^0,10^{16}].
\end{equation}
For each value of the $\mathrm{difficulty}$ parameter, we generate $\mat{A}$, $\vec{x}$, and $\vec{b}$ as in \cref{fig:comparison}.
We test sketch-and-precondition, iterative sketching with momentum, SPIR, FOSSILS, and (Householder) \QR.
For sketch-and-precondition, we set $d \coloneqq 12n$ and run for $100$ iterations.
For iterative sketching with momentum, we use the implementation recommended in \cite[sec.~3.3 and app.~B]{Epp24}.
For SPIR and FOSSILS, we follow the recommendations in \cref{sec:implementation}.

\begin{figure}[t]
  \centering
  \includegraphics[width=0.45\textwidth]{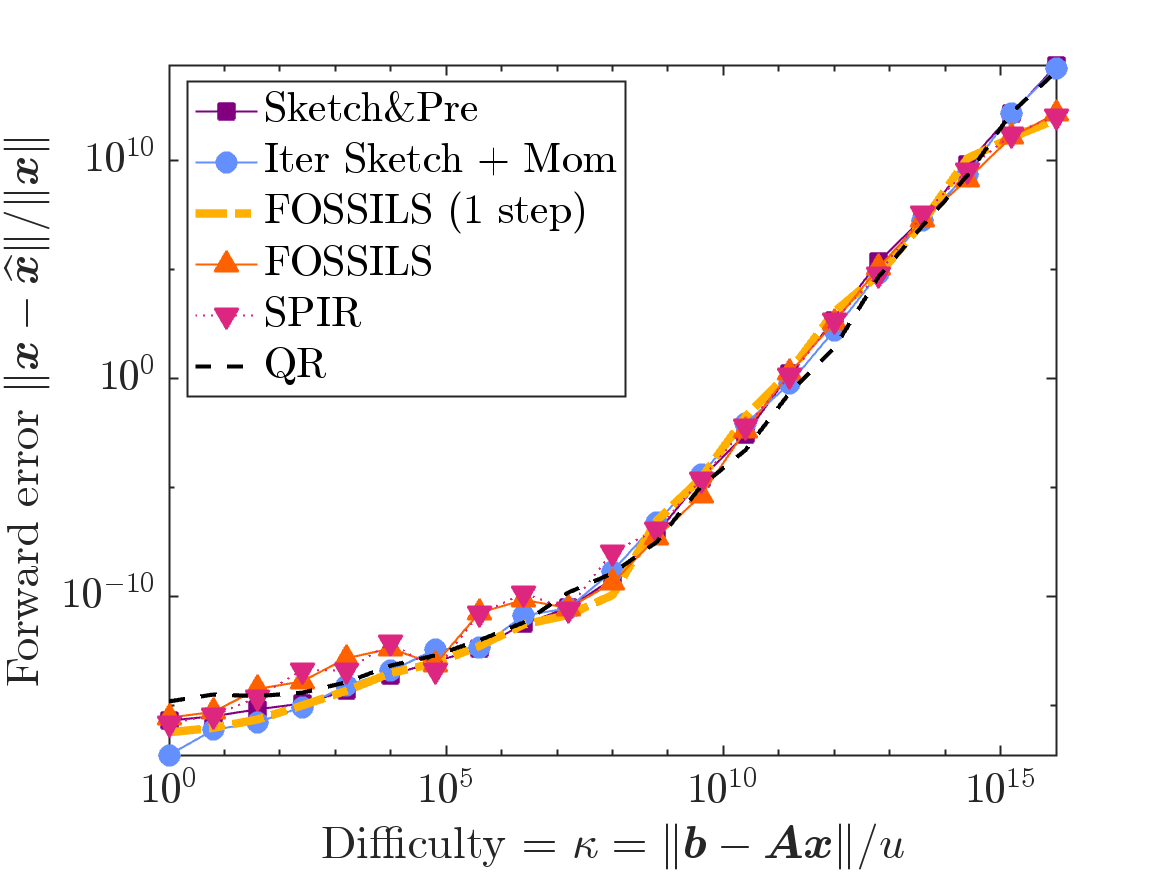}
  \includegraphics[width=0.45\textwidth]{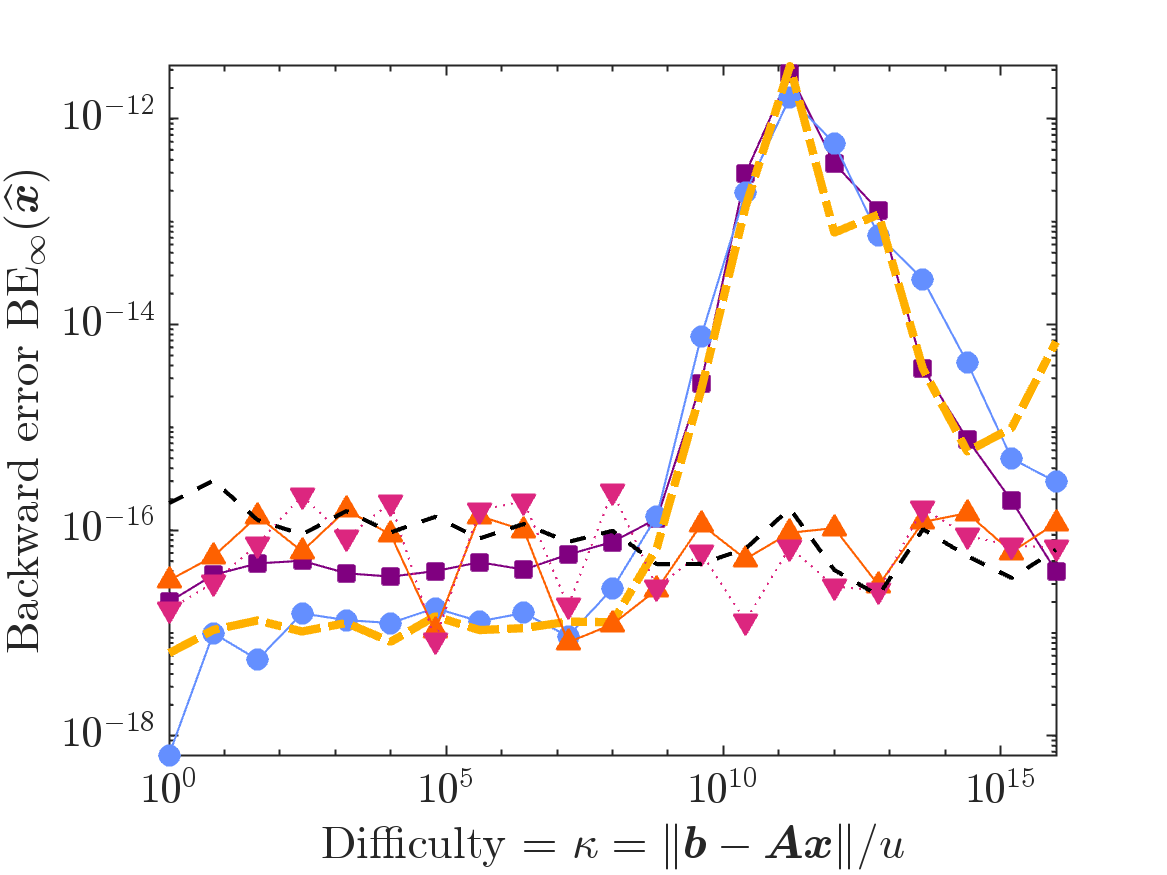} 
    
  \caption{Forward error (\emph{left}) and backward error (\emph{right}) for different least-squares solvers of different difficulties \cref{eq:difficulty}.
  FOSSILS is the only fast backward stable randomized least-squares solver, achieving comparable backward error to Householder \QR.} \label{fig:comparison_2}
\end{figure}

\Cref{fig:comparison_2} shows the forward (\emph{left}) and backward (\emph{right}) errors for these solvers on a sequence of randomly generated least-squares problems of increasing difficulty.
The major finding is the confirmation of backward stability of SPIR and FOSSILS for problems of all difficulties, achieving comparable backward error to \QR for all problems.
We note in particular that, for highly ill-conditioned problems $\kappa \gtrapprox 10^{14}$, the regularization approach developed in \cref{sec:implementation} yields a convergent and backward stable scheme.

To see that two refinement steps are necessary for FOSSILS, \cref{fig:comparison_2} also shows the forward and backward error for the output $\vec{x}_1$ of a single call to the FOSSILS outer solver, listed as FOSSILS (1 step).
We run FOSSILS (1 step) for 100 iterations to ensure convergence.
Even with 100 iterations, FOSSILS (1 step) shows similar stability properties to iterative sketching and sketch-and-precondition, being forward but not backward stable.
This demonstrates that both refinement steps in FOSSILS are necessary to obtain a backward stable scheme.

We further confirm the backward stability of SPIR in~\cref{fig:grids_fossils}, where we show a contour plot of the backward error for SPIR (\emph{left}) and iterative sketching with momentum (\emph{right}) for combinations of parameters $\kappa,\norm{\vec{b}-\mat{A}\vec{x}}/u \in [1,10^{16}]$.
As we see in the left panel, the backward error for SPIR is always a small multiple of the unit roundoff, regardless of the conditioning of $\mat{A}$ or the residual norm $\|\vec{b}-\mat{A}\vec{x}\|$.
We omit similar-looking plots for FOSSILS and sketch-and-precondition.

\begin{figure}[t]
  \centering
  \includegraphics[width=0.8\textwidth]{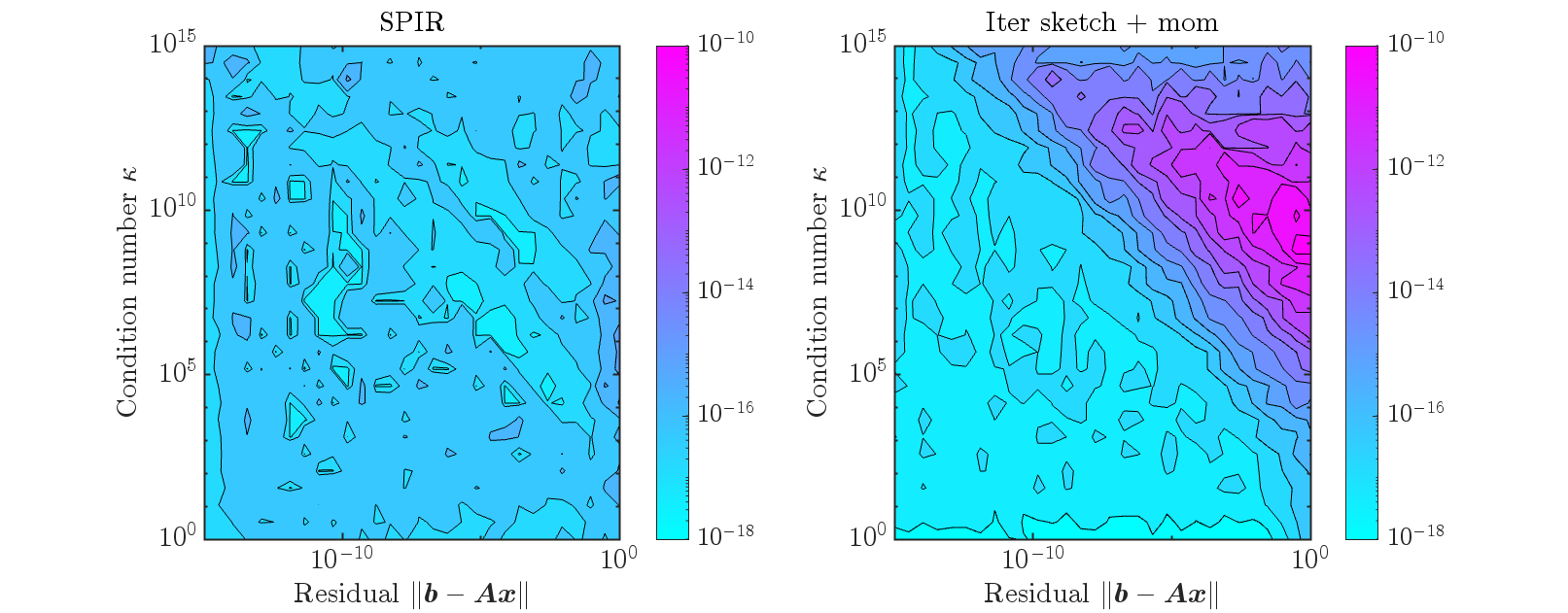}
  \caption{The backward error for SPIR (\emph{left}) and iterative sketching with momentum (\emph{right}) for sequences of randomly generated least-squares problems (generated as in \cref{fig:comparison}) varying condition number $\kappa$ and scaled residual norm $\|\vec{b}-\mat{A}\vec{x}\|/u$.} \label{fig:grids_fossils}
\end{figure}

\subsection{Runtime and iteration count} \label{sec:runtime}

\begin{figure}[t]
  \centering
  
  \includegraphics[width=0.45\textwidth]{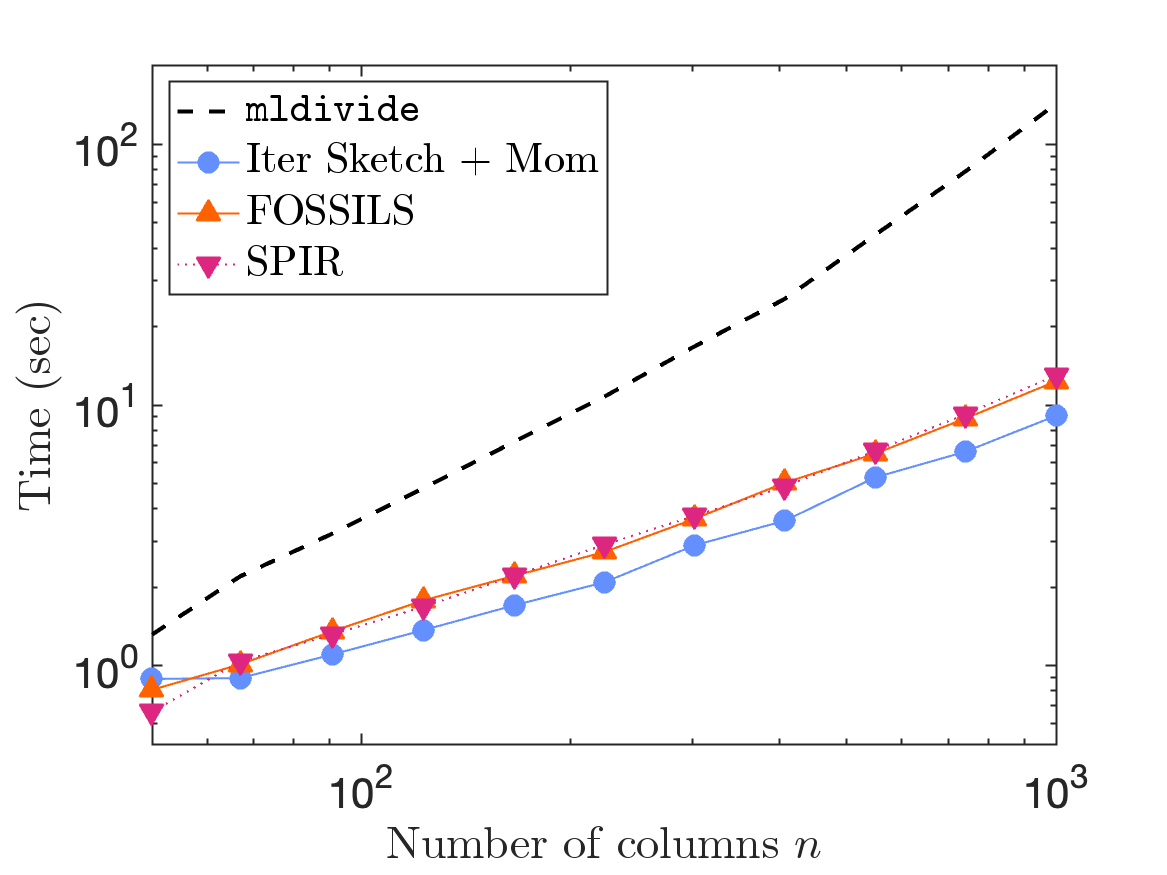}
  \includegraphics[width=0.45\textwidth]{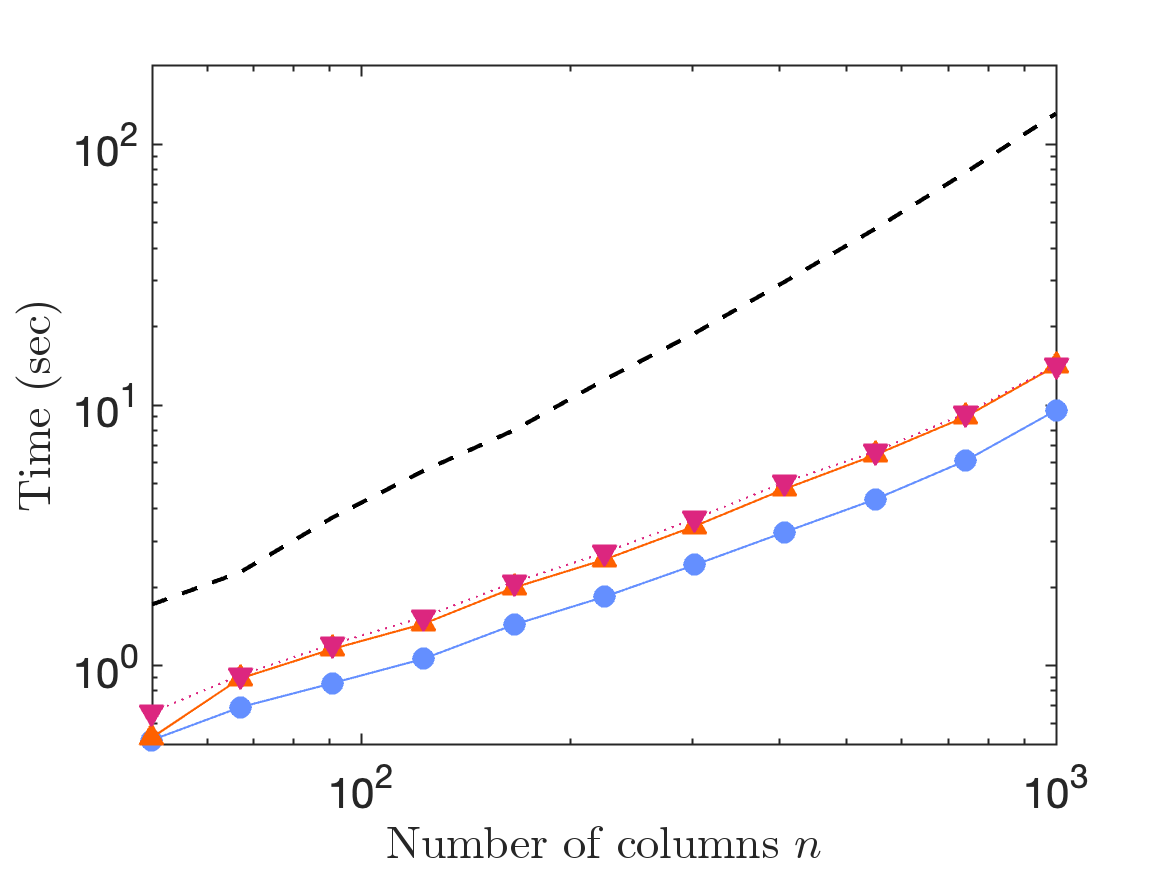}
    
  \caption{Runtime of SPIR, FOSSILS, iterative sketching with momentum, and MATLAB's \QR-based \texttt{mldivide} solver for kernel regression (real-valued, \emph{left}) and Prony (complex-valued, \emph{right}) dense least-squares problems.} \label{fig:dense}
\end{figure}

In this section, we present a runtime comparison of FOSSILS to MATLAB's QR-based direct solver \texttt{mldivide} and iterative sketching with momentum.
For our first experiments, we test on dense least-squares problems from applications:
\begin{itemize}
    \item \textbf{Kernel regression for exotic particle detection.}
    We consider least-squares problems for fitting the SUSY dataset \cite{BSW14} using a linear combination of kernel functions.
    We adopt the same setup as in \cite[sec.~3.4]{Epp24}, yielding real-valued least-squares problems of dimension $m = 10^6$ and $n \in [10^1,10^3]$.
    The condition numbers of these problems range from $\approx 10^1$ (for $n = 10^1$) to $\approx 10^7$ (for $n = 10^3$) and the residual is large $\norm{\vec{b}-\mat{A}\vec{x}}/\norm{\vec{b}} \approx 0.5$.
    \item \textbf{Prony's method and quantum eigenvalue algorithms.} 
    We consider a family of least-squares problems of dimension $m = 5\times 10^5$ and $n \in [10^1,10^3]$ that arise from applying Prony's method to estimate eigenvalues from noisy measurements from a quantum device.
    See \cref{app:prony} for details.
    The resulting least-squares problems are complex-valued $\mat{A} \in \complex^{m\times n}, \vec{b} \in \complex^m$ with a condition number of $\approx 10^6$ and residual of $\norm{\vec{b}-\mat{A}\vec{x}}/\norm{\vec{b}} \approx 10^{-6}$.
\end{itemize}
We run these experiments on a Mac Pro with a 2.7 GHz 12-Core Intel Xeon E5 processor and 64 GB 1866 MHz DDR3 RAM.
Results are shown in \cref{fig:dense}.
For all randomized methods, we use embedding dimension $d \coloneqq 12n$.
We find that FOSSILS is faster than \texttt{mldivide} for both problems at roughly $n\approx 50$, achieving a \textbf{maximal speedup of 11$\times$ (kernel regression) and 9$\times$ (Prony)} at $n=10^3$.
In particular, with our recommended implementation, the backward stable SPIR and FOSSILS solvers are nearly as fast as the forward stable iterative sketching with momentum method.

\begin{table}[t]
\centering
\caption{\textbf{\textit{Runtime comparison with sparse matrices.}} Runtime for SPIR and MATLAB's \texttt{mldivide} for ten problems from the SuiteSparse Matrix Collection \cite{DH11}.
Here, \texttt{nnz} reports the number of nonzero entries in the matrix and speedup reports the ratio of the \texttt{mldivide} and SPIR runtimes.
For problems with a beneficial sparsity pattern (problems 5--9), \texttt{mldivide} is significantly faster than SPIR; for other problems, SPIR is faster by $2.1\times$ to $6\times$. \label{tab:sparse}}
\begin{tabular}{clcccccc}\toprule
\# & Matrix & $m$ & $n$ & \texttt{nnz} & SPIR & \texttt{mldivide} & Speedup \\\midrule
1 & \texttt{JGD\_BIBD/bibd\_20\_10} & 1.8e5 & 1.9e2 & 8.3e6 & 1.3 & 2.7 & 2.1$\times$\\
2 & \texttt{JGD\_BIBD/bibd\_22\_8} & 3.2e5 & 2.3e1 & 9.0e6 & 1.6 & 3.8 & 2.4$\times$\\
3 & \texttt{Mittelmann/rail2586} & 9.2e5 & 2.6e3 & 8.0e6 & 7.0 & 40 & 5.8$\times$ \\
4 & \texttt{Mittelmann/rail4284} & 1.1e6 & 4.3e3 & 1.1e7 & 25 & 150 & 6$\times$ \\
5 & \texttt{LPnetlib/lp\_osa\_30} & 1.0e5 & 4.4e3 & 6.0e5 & 24 & 0.33 & 0.01$\times$ \\
6 & \texttt{Meszaros/stat96v1} & 2.0e5 & 6.0e3 & 5.9e5 & 58 & 0.10 & 0.002$\times$ \\
7 & \texttt{Dattorro/EternityII\_A} & 1.5e5 & 7.4e3 & 7.8e5 & 100 & 1.8 & 0.02$\times$ \\
8 & \texttt{Yoshiyasu/mesh\_deform}& 2.3e5 & 9.4e3 & 8.5e5 & 200 & 0.3 & 0.002$\times$ \\
9 & \texttt{Dattorro/EternityII\_Etilde} & 2.0e5 & 1.0e4 & 1.1e6 & 240 & 3.3 & 0.01$\times$ \\
10 & \texttt{Mittelmann/spal\_004}& 3.2e5 & 1.0e4 & 4.6e7 & 260 & 550 & 2.1$\times$ \\
\bottomrule
\end{tabular}
\end{table}

As a second timing experiment, we test SPIR and \texttt{mldivide} on problems for the SuiteSparse Matrix Collection \cite{DH11}.
We consider all rectangular problems in the collection with at least $m = 10^5$ rows and $n = 10^2$ columns, transposing wide matrices so that they are tall.
We remove any matrices that are structurally rank-deficient and consider \rr{the ten problems meeting our criteria with the smallest values of $n$}.
For each matrix, we generate $\vec{b}$ to have independent standard normal entries.
We use a smaller embedding dimension of $d = 3n$ for SPIR.
Timings for \texttt{mldivide} and SPIR are shown in \cref{tab:sparse}.
As these results show, the runtime of \texttt{mldivide} depends sensitively on the sparsity pattern.
Problems 5--9 possess a low fill-in \QR factorization and \texttt{mldivide} is 20$\times$ to a 1000$\times$ faster than SPIR.
However, \textbf{SPIR is 2.1$\times$ to 6$\times$ faster than \texttt{mldivide} on problems 1--4 and 10, which lack a benevolent sparsity pattern}.
We note that the stability of randomized iterative methods for sparse matrices is poorly understood; for matrices with few nonzero entries, sketch-and-precondition has been observed to be stable even with the zero initialization.

\begin{figure}[ht!]
    \centering
    \includegraphics[width = 0.8\linewidth]{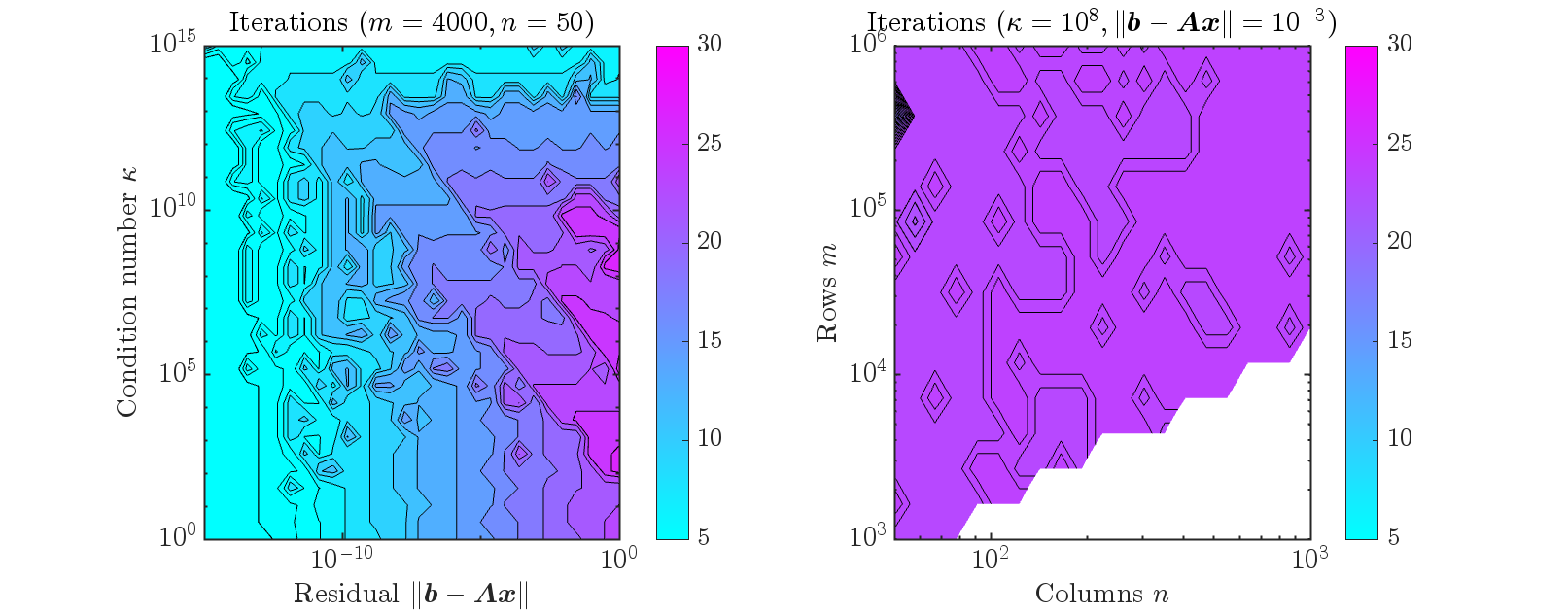}
    \caption{Total number of iterations (in both first and second refinement steps) until SPIR reaches convergence for different values of $\kappa,\norm{\vec{b} - \mat{A}\vec{x}}$ (\emph{left}) and for different values of $m\ge n$ (\emph{right}).}
    \label{fig:itcount}
\end{figure}

Finally, \cref{fig:itcount} shows the total number of conjugate gradient iterations needed for SPIR (across both refinement steps) for different parameters and problem sizes.
This figure reports two different parameter studies.
In the left panel, \cref{fig:itcount} shows the number of iterations for different values of the condition number $\kappa$ and residual norm $\|\vec{b}-\mat{A}\vec{x}\|$ for randomly generated problems of size $4000\times 50$.
For these problems, the number of iterations tends to be smaller for problems large $\kappa$ or small $\norm{\vec{b} - \mat{A}\vec{x}}$.
In all cases, however, the number of iterations is at most 30.
In the right panel, \cref{fig:itcount} shows the number of iterations for randomly generated problems of different sizes $10^3\leq m\leq 10^6$ and $50\leq n\leq 10^3$, while fixing $\kappa=10^8$ and $\norm{\vec{b}-\mat{A}\vec{x}} = 10^{-3}$.
We observe that the iteration count remains consistent across these $m,n$ values.

\section{Preconditioned iterative normal equations: Stability analysis} \label{sec:analysis}

In this section, we analyze the general meta-algorithm \cref{alg:meta-algorithm} for fast, backward stable randomized least-squares. 
Our result will be the following formalization of \cref{infthm:meta-stability}:

\begin{theorem}[Stability of the meta-algorithm] \label{thm:meta-stability}
    Assume the following:
    \begin{enumerate}[label=(\Alph*)]
        \item \textbf{Numerically full-rank.} The condition number $\kappa$ satisfies $\kappa u \ll 1$ (see formal definition below).
        \item \textbf{Non-pathological rounding errors.} The output satisfies the non-pathological rounding error assumption, defined below in \cref{sec:non-pathological}.
        \item \textbf{Stable sketching.} The matrix $\mat{S}$ is a sketching matrix for $\mat{A}$ with $d \le n^3$ rows, distortion $\eta \le 0.9$, and a forward stable multiply operation
        \begin{equation} \label{eq:stable-sketch}
            \norm{\fl(\mat{S}\mat{A}) - \mat{S}\mat{A}} \lesssim \norm{\mat{A}}u.
        \end{equation}
        \item \textbf{Stability of inner solver.} The $\Call{IterativeSolver}{}$ subroutine satisfies the stability guarantee
	\begin{equation} \label{eq:inner-solver-guarantee}
		\norm{ \fl(\Call{IterativeSolver}{\vec{c}}) - (\Rhat^{-\top}\mat{A}^\top\mat{A}\Rhat^{-1})^{-1}\vec{c} } \lesssim \kappa u \cdot \norm{\vec{c}}.
	\end{equation}
    \end{enumerate}
    Then $\vec{x}_1$ computed by \cref{alg:meta-algorithm} is strongly forward stable, and $\vec{x}_2$ is backward stable.
\end{theorem}

The rest of this section will present a proof of this result.
Analysis of SPIR and FOSSILS appears in \cref{sec:FOSSILS-analysis,sec:SPIR-analysis}.
See \cref{sec:generalizing} for a slight generalization of \cref{thm:meta-stability}.
To simplify the presentation, we introduce notation and assumptions which we enforce throughout this section:

\paragraph{Notation and assumptions.}
We remind the reader \rr{of the $\lesssim$ and $\ll$ notations defined in the introduction.}
Vectors $\vec{e}$, $\vec{e}'$, $\vec{e}_1$ $\vec{e}_2$, $\vec{e}_{1,2}$, etc.\ denote arbitrary vectors of norm $\lesssim u$.
We emphasize that we make no attempts to compute or optimize the value of the constants in our proofs; as our numerical experiments demonstrate, SPIR and FOSSILS appear to have similar backward and forward error to \QR in practice.

The numerically computed version of a quantity $\vec{f}$ is denoted by $\fl(\vec{f})$ or $\vec{\hat{f}}$ and the error is $\err(\vec{f}) \coloneqq \vec{\hat{f}} - \vec{f}$.
A vector is \emph{exactly represented} if $\fl(\vec{f}) = \vec{f}$.
We assume throughout that quantities are evaluated in the sensible way, e.g., $\mat{A}\mat{R}^{-1}\vec{z}$ is evaluated as $\mat{A}(\mat{R}^{-1}\vec{z})$.
All triangular solves, e.g., $\mat{R}^{-1}\vec{z}$, are computed by back substitution.

By homogeneity, we are free to normalize the least-squares problem \cref{eq:ls} so that
\begin{equation} \label{eq:normalization}
    \norm{\mat{A}} = \norm{\vec{b}} = 1.
\end{equation}
\textbf{We use this normalization throughout the remainder of this paper.}
The condition number of $\mat{A}$ is $\kappa \coloneqq \cond(\mat{A}) = \sigma_{\rm max}(\mat{A})/\sigma_{\rm min}(\mat{A})$, which we assume satisfies $\kappa u \ll 1$.

\subsection{The non-pathological rounding error assumption} \label{sec:non-pathological}

Wedin's theorem (\cref{fact:wedin}) tells us that a backward stable solution $\vec{\hat{x}}$ obeys the forward error bound
\begin{equation*}
    \norm{\vec{\hat{x}} - \vec{x}} \lesssim \kappa (\norm{\vec{x}} + \kappa\norm{\vec{b}-\mat{A}\vec{x}})u.
\end{equation*}
If the magnitude of the rounding errors agrees with Wedin's theorem and the errors are assumed generic (i.e., random), then we would expect the norm of the computed solution to have size
\begin{equation} \label{eq:non-pathological-rounding}
    \norm{\vec{\hat{x}}} \gtrsim \norm{\vec{x}} + \kappa (\norm{\vec{x}} + \kappa \norm{\vec{b}-\mat{A}\vec{x}})u.
\end{equation}
We call \cref{eq:non-pathological-rounding} the \textbf{non-pathological rounding error assumption}.

We believe the non-pathological rounding error assumption essentially always holds for the output of SPIR or FOSSILS.
In ordinary floating point arithmetic, rounding errors are not random \cite[sec.~1.17]{Hig02}, but they often behave as if they are \cite{HM19}.
If the non-pathological rounding error assumption were to \emph{not} hold, the rounding errors would have to occur in a very particular way to lead to a numerical solution of very small norm.
We have never seen a violation of the non-pathological rounding error assumption in our numerical experiments.
Finally, we note that we have posterior estimates for the backward error (\cref{sec:qa}), so a small backward error for the computed solution can always be certified at runtime at small cost.

Finally, we note one case where the non-pathological rounding error is always true:

\begin{proposition}[Non-pathological rounding errors]
    The non-pathological rounding error assumption \cref{eq:non-pathological-rounding} holds provided that the true solution $\vec{x}$ satisfies $\norm{\vec{x}} \gg \kappa(\norm{\vec{x}} + \kappa\norm{\vec{b}-\mat{A}\vec{x}})u$. 
\end{proposition}

This result tells us that the non-pathological rounding error assumption could fail only on problems where the relative forward error of the computed solution could be high $\norm{\vec{\hat{x}} - \vec{x}} / \norm{\vec{x}}\gtrsim 1$.

\subsection{Stability of sketching}

In this section, we \rrr{prove}{state} an analog of the preconditioning result \cref{fact:whitening} in finite-precision. Its proof can be found in~\cref{app:stability-sketching-proof}.

\begin{lemma}[Stability of sketching] \label{lem:stability-sketching}
	Assume 
 the hypotheses of \cref{thm:meta-stability} 
 and let $\Rhat$ be the computed R-factor of \QR decomposition of $\mat{S}\mat{A}$ computed using Householder \QR.
    Then
    \begin{align}
        \norm{\Rhat^{-1}}  &\le 20\kappa, & \norm{\Rhat} &\le 2, \label{eq:R-whiten-fp} \\
        \frac{1}{1+\eta} - \sigma_{\rm min}(\mat{A}\Rhat^{-1}) &\lesssim \kappa u,& \norm{\smash{\mat{A}\Rhat^{-1}}} - \frac{1}{1-\eta} &\lesssim \kappa u. \label{eq:AR-whiten-fp}
    \end{align}
\end{lemma}

\subsection{Basic stability results}

We begin by presenting some standard results from rounding error analysis, making use of our ``$\vec{e}$ notation'' introduced at the start of \cref{sec:analysis}.
The first concern stability for addition and multiplication \cite[secs.~2--3]{Hig02}:

\begin{fact}[Basic stability results] \label{fact:basic-stability}
    For exactly representable vectors $\vec{w}$ and $\vec{z}$ in $\real^m$ or $\real^n$ and exactly representable $\gamma \in \real$, we have
    \begin{equation*}
        \norm{\err(\vec{z} \pm \vec{w})}\lesssim \norm{\vec{z} \pm \vec{w}}u, \quad \norm{\err(\gamma\cdot \vec{z})} \lesssim |\gamma| \norm{\vec{z}} u.
    \end{equation*}
    For matrix multiplication by $\mat{A}$ normalized as in \cref{eq:normalization}, we have
    \begin{equation*}
        \norm{\err(\mat{A}\vec{z})} \lesssim \norm{\vec{z}} u, \quad \norm{\err(\mat{A}^\top\vec{z})} \lesssim \norm{\vec{z}} u.
    \end{equation*}
\end{fact}

We also have the following characterization of the backward stability of triangular solves. \rr{~\cref{prop:backward-triangular} is proved in~\cref{app:backward-triangular-proof}.}

\begin{proposition}[Backward stability of triangular solves] \label{prop:backward-triangular}
    Instate the asssumptions of \rrr{\cref{thm:fossils-stability}}{\cref{thm:meta-stability}} and let $\vec{z}\in\real^n$ be exactly representable. Then    
    \begin{equation*}
        \err(\Rhat^{-1}\vec{z}) = \norm{\Rhat^{-1}\vec{z}} \cdot \Rhat^{-1}\vec{e}, \quad \err(\Rhat^{-\top}\vec{z}) = \norm{\smash{\Rhat^{-\top}\vec{z}}} \cdot \Rhat^{-\top}\vec{e}'. \label{eq:RT_solve}
    \end{equation*}
\end{proposition}

By combining \cref{fact:basic-stability,prop:backward-triangular}, we can analyze the effect of multiplying by $\mat{A}\Rhat^{-1}$ and variations thereof:

\begin{proposition}[Multiplication by whitened basis] \label{prop:whitened-stability}
   Instate the asssumptions of \rrr{\cref{thm:fossils-stability}}{\cref{thm:meta-stability}} and normalization \cref{eq:normalization}.
   For an exactly representable vector $\vec{z}$ of the appropriate size,
    \begin{align} 
        \err(\Rhat^{-\top}\mat{A}^\top\vec{z}) &= \norm{\vec{z}}\cdot \Rhat^{-\top}\vec{e}, \label{eq:WT-analysis} \\
        \norm{\err(\Rhat^{-\top}\mat{A}^\top\mat{A}\Rhat^{-1}\vec{z})} &\lesssim \kappa u \norm{\vec{z}}.\label{eq:M-analysis}
    \end{align}
\end{proposition}

\begin{proof}
    We begin by proving \cref{eq:WT-analysis}. 
    By \cref{fact:basic-stability}, we have
    \begin{equation*}
        \err(\mat{A}^\top \vec{z}) = \norm{\vec{z}} \cdot \vec{e}_1.
    \end{equation*}
    By \cref{prop:backward-triangular}, we obtain
    \begin{equation} \label{eq:WT-multiply-error}
        \err(\Rhat^{-\top}\mat{A}^\top\vec{z}) = \norm{\vec{z}} \cdot \Rhat^{-\top}\vec{e}_1 + \norm{\Rhat^{-\top}(\mat{A}^\top \vec{z} + \norm{\vec{z}} \cdot \vec{e}_1)} \cdot \Rhat^{-\top}\vec{e}_2.
    \end{equation}
    By the triangle inequality, the submultiplicative property, and \cref{lem:stability-sketching}, we have
    \begin{equation*}
        \norm{\Rhat^{-\top}\left(\mat{A}^\top \vec{z} + \norm{\vec{z}} \cdot \vec{e}_1\right)} \le \left(\norm{\mat{A}\Rhat^{-1}}+\norm{\Rhat^{-\top}\vec{e}_1}\right) \norm{\vec{z}} \lesssim (1 + \kappa u)\norm{\vec{z}}\lesssim\norm{\vec{z}}.
    \end{equation*}
    This establishes that the right-hand side of \cref{eq:WT-multiply-error} has the form $\norm{\vec{z}}\cdot \Rhat^{-\top}\vec{e}$, proving \cref{eq:WT-analysis}.

    We now prove \cref{eq:M-analysis}.
    By \cref{prop:backward-triangular},
    \begin{equation*}
        \err(\Rhat^{-1}\vec{z}) = \norm{\Rhat^{-1}\vec{z}} \cdot \Rhat^{-1}\vec{e}_3 = \kappa \norm{\vec{z}} \cdot \Rhat^{-1}\vec{e}_4.
    \end{equation*}
    In the second equality, we used the fact that $\norm{\smash{\Rhat^{-1}\vec{z}}}\lesssim \kappa \norm{\vec{z}}$ by \cref{lem:stability-sketching}.
    By \cref{fact:basic-stability}, we have
    \begin{equation*}
        \err(\mat{A}\Rhat^{-1}\vec{z}) = \kappa \norm{\vec{z}} \cdot \mat{A}\Rhat^{-1}\vec{e}_4 + \norm{\Rhat^{-1}(\vec{z} + \kappa \norm{\vec{z}} \cdot \vec{e}_4)} \cdot \vec{e}_5.
    \end{equation*}
    By \cref{lem:stability-sketching}, we have $\norm{\mat{A}\Rhat^{-1}\vec{e}_4} \lesssim u$ and 
    \begin{equation*}
        \norm{\Rhat^{-1}(\vec{z} + \kappa \norm{\vec{z}} \cdot \vec{e}_4)} \le \sigma_{\rm min}(\Rhat)^{-1} (1 + \kappa \norm{\vec{e}_4}) \norm{\vec{z}} \lesssim \kappa \norm{\vec{z}}.
    \end{equation*}
    Here, we used the assumption that $\kappa \norm{\vec{e}_4} \lesssim \kappa u \lesssim 1$.
    Thus, we have shown $\norm{\err(\mat{A}\Rhat^{-1}\vec{z})} \lesssim \kappa u \norm{\vec{z}}$.
    Combining with \cref{eq:WT-analysis}, we obtain
    \begin{align*}
        &\norm{\err(\Rhat^{-\top}\mat{A}^\top \mat{A}\Rhat^{-1}\vec{z})} 
        = \norm{ \norm{\mat{A}\Rhat^{-1}\vec{z} + \err(\mat{A}\Rhat^{-1}\vec{z})} \cdot \Rhat^{-\top}\vec{e}_6 + \Rhat^{-\top}\mat{A}^\top \err(\mat{A}\Rhat^{-1}\vec{z})} \\
        &\qquad\le (\norm{\mat{A}\Rhat^{-1}}\norm{\vec{z}} + \norm{\err(\mat{A}\Rhat^{-1}\vec{z})}) \norm{\smash{\Rhat^{-\top}}}\norm{\vec{e}_6} + \norm{\Rhat^{-\top}\mat{A}^\top} \norm{\err(\mat{A}\Rhat^{-1}\vec{z})} \lesssim \kappa u\norm{\vec{z}}.
    \end{align*}
    The final inequality uses \cref{lem:stability-sketching} and the bound $\norm{\err(\mat{A}\Rhat^{-1}\vec{z})} \lesssim \kappa u \norm{\vec{z}}$.
\end{proof}

\subsection{The error formula}

We shall provide backward stability using \cref{cor:proving-backward}.
To do so, we establish an error formula of the form \cref{eq:backward-error-from-formula} for \cref{alg:meta-algorithm}:

\begin{lemma}[Error formula] \label{lem:error-formula}
    Instate the assumptions of \rrr{\cref{thm:fossils-stability}}{\cref{thm:meta-stability}} and define $\vec{r}_i \coloneqq \vec{b} - \mat{A}\vec{\hat{x}}_i$.
    Then we have the following error formula
    \begin{equation*} \label{eq:error-formula}
        \vec{x} - \vec{\hat{x}}_{i+1} = \left(1+\norm{\vec{x}} + \norm{\vec{\hat{x}}_i} + \kappa^2 u \norm{\vec{r}_i} + \kappa \norm{\mat{A}(\vec{x}-\vec{\hat{x}}_i)} \right) \cdot \Rhat^{-1}\vec{e} + \norm{\vec{r}_i} \cdot (\mat{A}^\top \mat{A})^{-1} \vec{e}'.
    \end{equation*}
\end{lemma}

\begin{proof}
    We proceed step by step through \cref{alg:meta-algorithm}. 
    By \cref{fact:basic-stability}, we have
    \begin{equation*}
        \err(\vec{b} - \mat{A}\vec{\hat{x}}_i) = (\norm{\vec{r}_i} + \norm{\vec{\hat{x}}_i})\vec{e}_1.
    \end{equation*}
    In particular, since
    \begin{equation} \label{eq:r-norm-bound}
       \norm{\vec{r}_i} \le \norm{\vec{b}} + \norm{\mat{A}}\norm{\vec{\hat{x}}_i} = 1 + \norm{\vec{\hat{x}}_i},
    \end{equation}
    the norm of the computed residual satisfies 
    \begin{equation} \label{eq:comp-r-norm-bound}
        \norm{\vec{\hat{r}}_i} \lesssim \norm{\vec{r}_i} + (\norm{\vec{r}_i} + \norm{\vec{\hat{x}}_i})u \lesssim \norm{\vec{r}_i} + (1 + \norm{\vec{\hat{x}}_i})u.
    \end{equation}
    Let $\vec{c}_i = \Rhat^{-\top} \mat{A}^\top \vec{r}_i$.
    By \cref{prop:whitened-stability},
    \begin{equation*}
        \err(\vec{c}_i) = \err(\Rhat^{-\top}\mat{A}^\top\vec{r}_i) = (\norm{\vec{r}_i} + \norm{\vec{x}_i}) \cdot \Rhat^{-\top}\mat{A}^\top\vec{e}_1 + \norm{\vec{\hat{r}}_i} \cdot \Rhat^{-\top} \vec{e}_2.
    \end{equation*}
    Using \cref{eq:AR-whiten-fp}, \cref{eq:r-norm-bound}, and \cref{eq:comp-r-norm-bound}, this result simplifies to
    \begin{equation*}
        \err(\vec{c}_i) = (1 + \norm{\vec{\hat{x}}_i}) \vec{e}_3 + [\norm{\vec{r}_i}+(1 + \norm{\vec{\hat{x}}_i})u] \Rhat^{-\top} \vec{e}_4.
    \end{equation*}
    By \cref{eq:R-whiten-fp}, $\norm{\Rhat^{-\top} \vec{e}_4} \lesssim \kappa u \ll 1$, so we can further consolidate
    \begin{equation} \label{eq:err-c}
        \err(\vec{c}_i) = (1 + \norm{\vec{\hat{x}}_i}) \vec{e}_5 + \norm{\vec{r}_i}\cdot \Rhat^{-\top} \vec{e}_4.
    \end{equation}
    By assumption (C), the value of
    \begin{equation} \label{eq:y}
        \vec{\delta y}_{i} = (\Rhat^{-\top}\mat{A}^\top\mat{A}\Rhat^{-1})^{-1}\vec{c}_i = \Rhat (\mat{A}^\top \mat{A})^{-1} \mat{A}^\top \vec{r}_i = \Rhat(\vec{x} - \vec{\hat{x}}_i).
    \end{equation}
    computed by the \Call{IterativeSolver}{} subroutine satisfies
    \begin{equation*}
        \norm{\vec{\hat{\delta y}}_i - (\Rhat^{-\top}\mat{A}^\top\mat{A}\Rhat^{-1})^{-1}\vec{\hat{c}}_i} \lesssim \kappa u \norm{\vec{\hat{c}}_i} \le \kappa u \norm{\vec{c}_i} + \kappa u \norm{\err(\vec{c}_i)}.
    \end{equation*}
    We compute the two norms in the right-hand side of this expression.
    First, in view of the identity $\mat{A}^\top \vec{r}_i = \mat{A}^\top \mat{A}(\vec{x} - \vec{\hat{x}}_i)$, we compute
    \begin{equation} \label{eq:c-norm}
        \norm{\vec{c}_i} = \norm{\Rhat^{-\top}\mat{A}^\top \mat{A}(\vec{x} - \vec{\hat{x}}_i)} \lesssim \norm{\mat{A}(\vec{x} - \vec{\hat{x}}_i)}.
    \end{equation}
    The last line is \cref{eq:AR-whiten-fp}.
    We bound $\norm{\err(\vec{c}_i)}$ using \cref{eq:err-c}, the triangle inequality, and \cref{eq:R-whiten-fp},
    \begin{equation*}
        \kappa u\norm{\err(\vec{c}_i)} \lesssim \kappa u(1 + \norm{\vec{\hat{x}}_i})u  + \kappa u\norm{\vec{r}_i} \cdot \norm{\Rhat^{-\top}\vec{e}_4} \lesssim \left[1 + \norm{\vec{\hat{x}}_i}  + \kappa^2 u\norm{\vec{r}_i}\right]u.
    \end{equation*}
    Combining the three previous displays, we get
    \begin{equation*}
        \norm{\vec{\hat{\delta y}}_i - (\Rhat^{-\top}\mat{A}^\top\mat{A}\Rhat^{-1})^{-1}\vec{\hat{c}}_i} \lesssim \left[1 + \norm{\vec{\hat{x}}_i}  + \kappa^2 u\norm{\vec{r}_i} + \kappa \norm{\mat{A}(\vec{x} - \vec{x}_i)}\right]u.
    \end{equation*}
    Therefore, we have
    \begin{equation*}
        \err(\vec{\delta y}_{i}) = \left[1 + \norm{\vec{\hat{x}}_i}  + \kappa^2 u\norm{\vec{r}_i} + \kappa \norm{\mat{A}(\vec{x} - \vec{\hat{x}}_i)}\right]\vec{e}_6 + (\Rhat^{-\top} \mat{A}^\top \mat{A} \Rhat^{-1})^{-1} \err(\vec{c}_i).
    \end{equation*}
    Substituting \cref{eq:err-c} and simplifying, we obtain
    \begin{equation} \label{eq:err-y}
        \err(\vec{\delta y}_{i}) = \left[(1 + \norm{\vec{\hat{x}}_i})  + \kappa^2 u\norm{\vec{r}_i} + \kappa \norm{\mat{A}(\vec{x} - \vec{\hat{x}}_i)}\right]\vec{e}_7 + \norm{\vec{r}_i} \cdot \Rhat (\mat{A}^\top\mat{A})^{-1} \vec{e}_5.
    \end{equation}
    Now we treat $\Rhat^{-1}\, \vec{\delta y}_{i}$.
    By \cref{prop:backward-triangular}, we have
    \begin{equation*}
        \err(\Rhat^{-1}\,\vec{\delta y}_{i}) = \norm{\Rhat^{-1}\,\vec{\hat{\delta y}}_i} \cdot \Rhat^{-1} \vec{e}_8 + \Rhat^{-1} \err(\vec{\hat{\delta y}}_i).
    \end{equation*}
    Using \cref{eq:y,eq:err-c}, we may bound $\norm{\Rhat^{-1}\vec{\hat{\delta y}}_i}$ as
    \begin{align*}
        \norm{\Rhat^{-1}\,\vec{\hat{\delta y}}_i} &\le \norm{\Rhat^{-1}\, \vec{\delta y}_{i}} + \norm{\Rhat^{-1}\err(\vec{\delta y}_{i})} \lesssim \norm{\vec{x} - \vec{\hat{x}}_i} + \left[1 + \norm{\vec{\hat{x}}_i} + \kappa \norm{\mat{A}(\vec{x} - \vec{x}_i)}\right]u  + \kappa^2 u\norm{\vec{r}_i} \\
        &\lesssim \norm{\vec{x}} + \norm{\vec{\hat{x}}} + \left[1 + \kappa \norm{\mat{A}(\vec{x} - \vec{x}_i) } + \kappa^2 \norm{\vec{r}_i}\right]u
    \end{align*}
    Combining the two previous displays and \cref{eq:err-y} then simplifying, we obtain
    \begin{equation*}
        \err(\Rhat^{-1}\, \vec{\delta y}_{i}) = \left( 1 + \norm{\vec{\hat{x}}_i} + \norm{\vec{x}} + \kappa^2 u \norm{\vec{r}_i} + \kappa \norm{\mat{A}(\vec{x} - \vec{\hat{x}}_i)} \right) \Rhat^{-1}\vec{e}_9 + \norm{\vec{r}_i} \cdot (\mat{A}^\top\mat{A})^{-1}\vec{e}_5.
    \end{equation*}
    Finally, using the bound invoking \cref{fact:basic-stability} for the addition $\vec{x} = \vec{\hat{x}}_i + \Rhat^{-1}\, \vec{\delta y}_{i}$ gives the formula \cref{eq:error-formula}.
\end{proof}

\subsection{\rr{A sufficient condition for backward stability}}

\rr{To prove complete the proof of backward stability, we will use the following result, which follows as a corollary of \cref{thm:backward-componentwise}:}

\begin{corollary}[Proving backward stability] \label{cor:proving-backward}
    Let $\Rhat$ denote the numerically computed preconditioner in \cref{eq:sketch_qr}.
    Then $\vec{\hat{x}}$ is a backward stable solution to \cref{eq:ls} if it can be written in the form 
    \begin{equation} \label{eq:backward-error-from-formula}
        \vec{x} - \vec{\hat{x}} = (1 + \norm{\vec{\hat{x}}}) \cdot \Rhat^{-1}\vec{f} + \norm{\vec{b} - \mat{A}\vec{\hat{x}}} \cdot (\mat{A}^\top \mat{A})^{-1}\vec{f}' \quad \text{for } \norm{\vec{f}},\norm{\vec{f}'} \lesssim u.
    \end{equation}
\end{corollary}

\Cref{cor:proving-backward} is proven in \cref{app:proving-backward}.

\subsection{Completing the proof}

With the error formula \cref{eq:error-formula} in hand, we present a proof of \cref{thm:meta-stability}.
We will make use of the following result for sketch-and-solve, which follows directly from \cite[Lem.~8]{Epp24} \rr{and the standing assumption that $\kappa u \ll 1$}:

\begin{fact}[Sketch-and-solve: Finite precision] \label{fact:sketch-solve-numerical}
    The numerically computed sketch-and-solve solution $\vec{\hat{x}}_0$ satisfies:
    \begin{equation*}
        \norm{\mat{A}(\vec{x} - \vec{\hat{x}}_0)} \lesssim \norm{\vec{b} - \mat{A}\vec{x}} + \norm{\vec{x}}u.
    \end{equation*}
\end{fact}

\begin{proof}[Proof of\cref{thm:meta-stability}]
    Let us first consider the first iterate $\vec{\hat{x}}_1$.
    By \cref{lem:error-formula}, we have
    \begin{equation*}
        \vec{x} - \vec{\hat{x}}_1 = \left(1 + \norm{\vec{x}} + \norm{\vec{\hat{x}}_0} + \kappa^2 u \norm{\vec{r}_0} + \kappa \norm{\mat{A}(\vec{x}-\vec{\hat{x}}_0)} \right) \cdot \Rhat^{-1}\vec{e}_1 + \norm{\vec{r}_0} \cdot (\mat{A}^\top \mat{A})^{-1} \vec{e}_2,
    \end{equation*}
    where we continue to use the shorthand $\vec{r}_i \coloneqq \vec{b} - \mat{A}\vec{x}_i$.
    To bound $\norm{\vec{r}_0}$, we use \cref{fact:sketch-solve-numerical} and the Pythagorean theorem:
    \begin{equation*}
        \norm{\vec{r}_0} = \left( \norm{\vec{b} - \mat{A}\vec{x}}^2 + \norm{\mat{A}(\vec{x} - \vec{\hat{x}}_0)}^2 \right)^{1/2} \lesssim \norm{\vec{b} - \mat{A}\vec{x}} + \norm{\vec{x}} u.
    \end{equation*}
    To bound $\norm{\vec{x}_0}$, we use \cref{fact:sketch-solve-numerical}, the triangle inequality, and the fact that $\sigma_{\rm min}(\mat{A}) = 1/\kappa$:
    \begin{equation*}
        \norm{\vec{\hat{x}}_0} \le \norm{\vec{x}} + \norm{\vec{x} - \vec{\hat{x}}_0} \le \norm{\vec{x}} + \kappa \norm{\mat{A}(\vec{x} - \vec{x}_0)} \lesssim \norm{\vec{x}} + \kappa \norm{\vec{b} - \mat{A}\vec{x}}.
    \end{equation*}
    Further, we bound $1$ by the triangle inequality:
    \begin{equation*}
        1 = \norm{\vec{b}} \le \norm{\vec{b} - \mat{A}\vec{x}} + \norm{\mat{A}\vec{x}} \le \norm{\vec{b} - \mat{A}\vec{x}} + \norm{\vec{x}}.
    \end{equation*}
    Combining the four previous displays and the hypothesis $\kappa u\ll 1$, we obtain
    \begin{equation*}
        \vec{x} - \vec{\hat{x}}_1 = \left(\norm{\vec{x}} + \kappa \norm{\vec{b}-\mat{A}\vec{x}} \right) \cdot \Rhat^{-1}\vec{e}_3 + (\norm{\vec{b} - \mat{A}\vec{x}} + \norm{\vec{x}}u) \cdot (\mat{A}^\top \mat{A})^{-1} \vec{e}_4.
    \end{equation*}
    Taking the norm of this formula and using \cref{eq:AR-whiten-fp}, we conclude that the first solution produced by \cref{alg:meta-algorithm} is strongly forward stable:
    \begin{align}
        \norm{\vec{x} - \vec{\hat{x}}_1} &\lesssim \kappa (\norm{\vec{x}} + \kappa \norm{\vec{b} - \mat{A}\vec{x}}) u, \\
        \norm{\mat{A}(\vec{x} - \vec{\hat{x}}_1)} &\lesssim (\norm{\vec{x}} + \kappa \norm{\vec{b} - \mat{A}\vec{x}}) u. \label{eq:first-iterate-A}
    \end{align}
    In particular, by the Pythagorean theorem, we have
    \begin{equation}
        \norm{\vec{r}_1} = \left( \norm{\vec{b} - \mat{A}\vec{x}}^2 + \norm{\mat{A}(\vec{x} - \vec{\hat{x}}_1)}^2 \right)^{1/2} \lesssim \norm{\vec{b} - \mat{A}\vec{x}} + \norm{\vec{x}}u.\label{eq:first-iterate-r}
    \end{equation}
    Now, we treat the second iterate $\vec{\hat{x}}_2$.
    By \cref{lem:error-formula}, we have 
    \begin{equation*}
        \vec{x} - \vec{\hat{x}}_2 = \left(1 + \norm{\vec{x}} + \norm{\vec{\hat{x}}_1} + \kappa^2 u \norm{\vec{r}_1} + \kappa \norm{\mat{A}(\vec{x}-\vec{\hat{x}}_1)} \right) \cdot \Rhat^{-1}\vec{e}_5 + \norm{\vec{r}_1} \cdot (\mat{A}^\top \mat{A})^{-1} \vec{e}_6.
    \end{equation*}
    Substituting in \cref{eq:first-iterate-A,eq:first-iterate-r} and simplifying, we obtain
    \begin{equation*}
        \vec{x} - \vec{\hat{x}}_2 = \left(1 + \norm{\vec{x}} + \kappa^2 u \norm{\vec{b} - \mat{A}\vec{x}} \right) \cdot \Rhat^{-1}\vec{e}_7 + \norm{\vec{b} - \mat{A}\vec{x}} \cdot (\mat{A}^\top \mat{A})^{-1} \vec{e}_8.
    \end{equation*}
    By the non-pathological rounding error assumption, $\norm{\vec{x}} + \kappa^2 u \norm{\vec{b} - \mat{A}\vec{x}} \lesssim \norm{\vec{\hat{x}}_2}$.
    Thus,
    \begin{equation*}
        \vec{x} - \vec{\hat{x}}_2 = \left(1 + \norm{\vec{\hat{x}}_2} \right) \cdot \Rhat^{-1}\vec{e}_9 + \norm{\vec{b} - \mat{A}\vec{x}} \cdot (\mat{A}^\top \mat{A})^{-1} \vec{e}_8.
    \end{equation*}
    By \cref{cor:proving-backward}, we conclude that \cref{alg:meta-algorithm} is backward stable.
\end{proof}

\subsection{Generalization: Beyond randomized preconditioning} \label{sec:generalizing}

Upon inspection of the proof, we see that we use minimal properties about the randomized construction of the preconditioner $\mat{R}$ and initialization $\vec{x}_0$.
Thus, the analysis we have already done immediately yields the following generalization of \cref{thm:meta-stability}:

\begin{theorem}[Beyond randomized preconditioning]
    Assume conditions (A), (B), and (D) of \cref{thm:meta-stability} and the alternate condition:
    \begin{enumerate}[label=(\Alph*')]
        \setcounter{enumi}{2}
        \item \textbf{Good preconditioner, initialization.} The numerically computed preconditioner $\Rhat$ is triangular and the preconditioned matrix has condition number $\cond(\mat{A}\Rhat^{-1})$ bounded by an absolute constant.
        The numerically computed initialization $\vec{\hat{x}}_0$ satisfies the conclusions of \cref{fact:sketch-solve-numerical}.
    \end{enumerate}
    Then lines 4--8 of \cref{alg:meta-algorithm} produce a strongly forward stable solution $\vec{x}_1$ and a backward stable solution $\vec{x}_2$.
\end{theorem}

\section{FOSSILS: Stability analysis} \label{sec:FOSSILS-analysis}

We now prove use \cref{thm:meta-stability} to analyze FOSSILS.
To simplify the analysis, we will analyze the algorithm with conservative parameter choices \cref{eq:parameters_proof,eq:iterations}:

\begin{theorem}[FOSSILS: backward stability] \label{thm:fossils-stability}
    Assume conditions (A)--(C) of \cref{thm:meta-stability} and 
    \begin{enumerate}[label=(\Alph*)]
        \setcounter{enumi}{3}
        \item \textbf{Parameters for Polyak solver.} Set $\alpha$ and $\beta$ according to a distortion $1.01\eta$, i.e., 
        \begin{equation} \label{eq:parameters_proof}
            \alpha = (1 - (1.01\eta)^2)^2, \quad \beta = (1.01\eta)^2
        \end{equation}
        and assume $|\hat{\alpha} - \alpha|, |\hat{\beta} - \beta| \lesssim u$.
        Set the iteration count 
        \begin{equation} \label{eq:iterations}
            q \ge q_0 \coloneqq \max\left(1+2 \frac{\log(1/(\kappa u))}{\log(1/(1.01\eta))},11\right).
        \end{equation}
    \end{enumerate}
    Then FOSSILS is backward stable.
    In particular, choosing $\mat{S}$ to be a sparse sign embedding with distortion $\eta = 1/2$ with the parameter settings \cref{eq:cohen}, FOSSILS produces a backward stable solution in $\order(mn\log(n/u) + n^3\log n)$ operations.
\end{theorem}

\subsection{Exact arithmetic} \label{sec:analysis-exact}

We begin by analyzing FOSSILS in exact arithmetic.
In exact arithmetic, a single call to the FOSSILS subroutine is sufficient to obtain any desired level of accuracy. 
We have the following result:

\begin{theorem}[FOSSILS: Exact arithmetic] \label{thm:fossils-exact}
    Consider FOSSILS (\cref{alg:FOSSILS-basic}) with an arbitrary initialization $\vec{x}_0$ and suppose that the Polyak solver (\cref{alg:polyak}) is terminated at iteration $j$.
    Then
    \begin{align*}
        \norm{\vec{x} - \vec{x}_1} \le \kappa \cdot \frac{4\sqrt{2}}{1-\eta} \cdot j\eta^{j-2}\cdot \norm{\mat{A}(\vec{x} - \vec{x}_0)}, \quad \norm{\mat{A}(\vec{x} - \vec{x}_1)} \le \frac{4\sqrt{2}}{1-\eta} \cdot j\eta^{j-2}\cdot \norm{\mat{A}(\vec{x} - \vec{x}_0)}.
    \end{align*}
    In particular, with the sketch-and-solve initialization, 
    \begin{equation} \label{eq:fossils-sketch-and-solve-exact}
        \norm{\vec{x} - \vec{x}_1} \le \kappa \cdot \frac{16\sqrt{\eta}}{(1-\eta)^2} \cdot j\eta^{j-2}\cdot \norm{\vec{b}-\mat{A}\vec{x}}, \quad
        \norm{\mat{A}(\vec{x} - \vec{x}_1)} \le \frac{16\sqrt{\eta}}{(1-\eta)^2} \cdot j\eta^{j-2}\cdot \norm{\vec{b}-\mat{A}\vec{x}}.
    \end{equation}
\end{theorem}

\begin{algorithm}[t]
	\caption{Polyak heavy ball method for solving $(\mat{R}^{-\top}\mat{A}^\top\mat{A}\mat{R}^{-1})\, \vec{\delta y} = \vec{c}$} \label{alg:polyak}
	\begin{algorithmic}[1]
		\Require Matrix $\mat{A}\in\real^{m\times n}$, preconditioner $\mat{R} \in \real^{n\times n}$, right-hand side $\vec{c} \in \real^n$, coefficients $\alpha, \beta > 0$, and iteration count $q > 0$ \Comment{$q \approx \log(1/u)$}
		\Ensure Solution $\vec{\delta y}_{(q)}\in\real^n$
		\State $\vec{\delta y}_{(1)}, \vec{\delta y}_{(0)} \gets \vec{c}$
		\For{$j=1,2,\ldots,q-1$} 
        \State $\vec{\delta y}_{(j+1)}\gets \vec{\delta y}_{(j)} + \alpha(\vec{c} - \mat{R}^{-\top}(\mat{A}^\top (\mat{A}(\mat{R}^{-1}\, \vec{\delta y}_{(j)})))) + \beta(\vec{\delta y}_{(j)} - \vec{\delta y}_{(j-1)})$
		\EndFor
	\end{algorithmic}
\end{algorithm}

The main difficulty in proving this result is to analyze the Polyak iteration, which we show in \cref{alg:polyak}.
In anticipation of our analysis of FOSSILS in finite precision, the following lemma characterizes the behavior of the Polyak iteration with a perturbation at each step.

\begin{lemma}[Polyak with noise] \label{lem:inner-noise}
    Let $\eta > 0$. 
    Assume the preconditioner $\mat{R}$ satisfies \cref{eq:AR-whiten} and that $\alpha,\beta$ are given by \cref{eq:optimal-coeffs}.
    Consider a noisy version of the Polyak iteration
    \begin{equation*}
    \vec{\delta y}_{(j+1)} = \vec{\delta y}_{(j)} + \alpha (\vec{c} - \mat{R}^{-\top}(\mat{A}^\top (\mat{A}(\mat{R}^{-1}\vec{\delta y}_{(j)})))) + \beta(\vec{\delta y}_{(j)} - \vec{\delta y}_{(j-1)}) + \vec{f}_{(j)}
    \end{equation*}
    and assume the initializations $\vec{\delta y}_{(1)} = \vec{\delta y}_{(0)} = \vec{c}$.
    Then the error for the $j$th iterate satisfies
    \begin{equation*}
        \norm{\vec{\delta y}_{(j)} - (\mat{R}^{-\top}\mat{A}^\top \mat{A}\mat{R}^{-1})^{-1}\vec{c}} \le 4\sqrt{2}\cdot j\eta^{j-2}\norm{\vec{c}} + \frac{9}{(1-\eta)^2} \cdot \max_{1 \le k \le j-1} \norm{\vec{f}_{(k)}}.
    \end{equation*}
\end{lemma}

\begin{proof}
    Denote the error $\vec{d}_{(j)} \coloneqq (\mat{R}^{-\top}\mat{A}^\top \mat{A}\mat{R}^{-1})^{-1}\vec{c}-\vec{\delta y}_{(j)}$, which satisfies the recurrence
    \begin{equation*}
        \twobyone{\vec{d}_{(j+1)}}{\vec{d}_{(j)}} = \mat{T}  \twobyone{\vec{d}_{(j)}}{\vec{d}_{(j-1)}} - \twobyone{\vec{f}_{(j)}}{\vec{0}} \quad \text{for } \mat{T} = \twobytwo{(1+\beta)\Id - \alpha\mat{R}^{-\top}\mat{A}^\top\mat{A}\mat{R}^{-1}}{-\beta \Id}{\Id}{\mat{0}}.
    \end{equation*}
    Therefore, expanding this recurrence and using the initial conditions $\vec{d}_{(1)} = \vec{d}_{(0)} = \vec{c}$, we obtain
    \begin{equation*}
        \twobyone{\vec{d}_{(j)}}{\vec{d}_{(j-1)}} = \mat{T}^{j-1} \twobyone{\vec{c}}{\vec{c}} + \sum_{k=1}^{j-1} \mat{T}^{j-1-k} \twobyone{\vec{f}_{(k)}}{\vec{0}}.
    \end{equation*}
    From the proof of \cite[Thm.~10]{Epp24}, we have the bound $\norm{\mat{T}^k}\le 4(k+1)\eta^{k-1}$.
    Thus, we have
    \begin{align*}
        \norm{\vec{d}_{(j)}} &\le \norm{\twobyone{\vec{d}_{(j)}}{\vec{d}_{(j-1)}}} \le \norm{\mat{T}^{j-1}} \norm{\twobyone{\vec{c}}{\vec{c}}} + \sum_{k=0}^{j-1} \norm{\mat{T}^{j-1-k}} \norm{\twobyone{\vec{f}_{(k)}}{\vec{0}}} \\
        &\le 4\sqrt{2}\cdot j\eta^{j-2}\norm{\vec{c}} + \left(1+4\sum_{k=1}^{j-2} (j-k) \eta^{(j-2)-k}\right) \cdot \max_{1\le k \le j-1} \norm{\vec{f}_{(k)}}\\
    \end{align*}
    We have established the desired conclusion.
\end{proof}

\begin{proof}[Proof of \cref{thm:fossils-exact}]
    By \cref{eq:precond-normal}, we have that
    \begin{equation*}
        \vec{x} = \vec{x}_0 + \mat{R}^{-1} \left(\mat{R}^{-\top}\mat{A}^\top\mat{A}\mat{R}^{-1}\right)^{-1}\mat{R}^{-\top}\mat{A}^\top (\vec{b}-\mat{A}\vec{x}_0).
    \end{equation*}
    Thus,
    \begin{equation} \label{eq:x-error-exact}
        \vec{x} - \vec{x}_1 = \mat{R}^{-1} \left(\left(\mat{R}^{-\top}\mat{A}^\top\mat{A}\mat{R}^{-1}\right)^{-1}\mat{R}^{-\top}\mat{A}^\top (\vec{b}-\mat{A}\vec{x}_0) - \vec{\delta y}_{(j)}\right).
    \end{equation}
    By \cref{lem:inner-noise}, we have
    \begin{equation*}
        \norm{\left(\mat{R}^{-\top}\mat{A}^\top\mat{A}\mat{R}^{-1}\right)^{-1}\mat{R}^{-\top}\mat{A}^\top (\vec{b}-\mat{A}\vec{x}_0) - \vec{\delta y}_{(j)}} \le 4\sqrt{2} \cdot j \eta^{j-2} \norm{\mat{R}^{-\top}\mat{A}^\top(\vec{b} - \mat{A}\vec{x}_0)}.
    \end{equation*}
    By \cref{fact:whitening} and the identity $\mat{A}^\top (\vec{b} - \mat{A}\vec{x}_0) = \mat{A}^\top \mat{A} (\vec{x} - \vec{x}_0)$, the right-hand side can be bounded:
    \begin{equation*}
        \norm{\left(\mat{R}^{-\top}\mat{A}^\top\mat{A}\mat{R}^{-1}\right)^{-1}\mat{R}^{-\top}\mat{A}^\top (\vec{b}-\mat{A}\vec{x}_0) - \vec{\delta y}_{(j)}} \le \frac{4\sqrt{2}}{1-\eta} \cdot j\eta^{j-2}\cdot \norm{\mat{A}(\vec{x} - \vec{x}_0)}.
    \end{equation*}
    Thus, plugging into \cref{eq:x-error-exact} and using \cref{fact:whitening}, we obtain
    \begin{align*}
        \norm{\vec{x} - \vec{x}_1} &\le \norm{\mat{R}^{-1}} \cdot \frac{4\sqrt{2} \cdot j \eta^{j-2}}{1-\eta} \norm{\mat{A}(\vec{x} - \vec{x}_0)} \le \kappa \cdot \frac{4\sqrt{2}}{1-\eta} \cdot j\eta^{j-2}\cdot \norm{\mat{A}(\vec{x} - \vec{x}_0)}, \\
        \norm{\mat{A}(\vec{x} - \vec{x}_1)} &\le \norm{\mat{A}\mat{R}^{-1}} \cdot \frac{4\sqrt{2} \cdot j \eta^{j-2}}{1-\eta} \norm{\mat{A}(\vec{x} - \vec{x}_0)} \le \frac{4\sqrt{2}}{1-\eta} \cdot j\eta^{j-2}\cdot \norm{\mat{A}(\vec{x} - \vec{x}_0)}.
    \end{align*}
    The bounds \cref{eq:fossils-sketch-and-solve-exact} then follow from known bounds for sketch-and-solve (e.g., \cite[Fact 4]{Epp24}).
\end{proof}

\subsection{Stability of Polyak heavy ball iteration}

We have the following guaratee for the Polyak solver in finite precision:

\begin{lemma}[Stability of Polyak heavy ball] \label{lem:polyak}
	Instate the assumptions of \cref{thm:fossils-stability}.
	Then the output of \cref{alg:polyak} satisfies
	\begin{equation*}
		\norm{ \fl(\Call{Polyak}{\vec{c}}) - (\Rhat^{-\top}\mat{A}^\top\mat{A}\Rhat^{-1})^{-1}\vec{c} } \lesssim \kappa u \cdot \norm{\vec{c}}.
	\end{equation*}
\end{lemma}

Combining \cref{lem:polyak} with \cref{thm:meta-stability} proves \cref{thm:fossils-stability}.

\begin{proof}
    By \cref{lem:stability-sketching} and the assumption $\kappa u \ll 1$, the numerically computed preconditioner satisfies \cref{eq:AR-whiten} with distortion parameter $1.01\eta$.
    Thus, we are in the setting to apply \cref{lem:inner-noise}.

    The Polyak iterates $\vec{\hat{\delta y}}_{(j)}$ evaluated in floating point satisfy a recurrence
    \begin{equation*}
        \vec{\hat{\delta y}}_{(j+1)} = \vec{\hat{\delta y}}_{(j)} + \alpha (\vec{c} - \Rhat^{-\top}\mat{A}^\top \mat{A}\Rhat^{-1}\vec{\hat{\delta y}}_{(j)}) + \beta(\vec{\hat{\delta y}}_{(j)} - \vec{\hat{\delta y}}_{(j-1)}) + \vec{f}_{(j)},
    \end{equation*}
    where $\vec{f}_{(j)}$ is a vector of floating point errors incurred in applying the recurrence.
    By \cref{prop:whitened-stability},
    \begin{equation*}
        \norm{\err(\Rhat^{-\top}\mat{A}^\top\mat{A}\Rhat^{-1}\, \vec{\hat{\delta y}}_{(j)})} \lesssim \kappa u \norm{\vec{\hat{\delta y}}_{(j)}}.
    \end{equation*}
    In particular, since $\kappa u \ll 1$, $\norm{\fl(\Rhat^{-\top}\mat{A}^\top\mat{A}\Rhat^{-1}\smash{\vec{\hat{\delta y}}_{(j)}})} \lesssim \norm{\smash{\vec{\hat{\delta y}}_{(j)}}}$.
    By several applications of \cref{fact:basic-stability}, we conclude
    \begin{equation*}
        \norm{\vec{f}_{(j)}} = \norm{\err(\smash{\vec{\hat{\delta y}}}_{(j+1)})} \lesssim (\kappa \norm{\vec{\smash{\hat{\delta y}}}_{(j)}} + \norm{\vec{c}} + \norm{\smash{\vec{\hat{\delta y}}_{(j-1)}}})u \lesssim \kappa u ( \norm{\smash{\vec{\hat{\delta y}}}_{(j)}} + \norm{\vec{c}} + \norm{\smash{\vec{\hat{\delta y}}}_{(j-1)}}).
    \end{equation*}
    More precisely, let us say that
    \begin{equation} \label{eq:f-bound}
        \norm{\vec{f}_{(j)}} \le C\kappa u ( \norm{\vec{\hat{\delta y}}_{(j)}} + \norm{\vec{c}} + \norm{\vec{\hat{\delta y}}_{(j-1)}})
    \end{equation}
    for an appropriate quantity $C$ depending polynomially on $m$ and $n$.

    Now, we establish the floating point errors remain bounded across iterations; specifically,
    \begin{equation} \label{eq:fp-error-bound-inductive}
        \norm{\vec{f}_{(j)}} \le 141 C\kappa u \norm{\vec{c}} \quad \text{for } j =1,2,\ldots.
    \end{equation}
    We prove this claim by induction. For the base case $j = 1$, we have
    \begin{equation*}
        \norm{\vec{f}_{(1)}} \le C\kappa u (\norm{\smash{\vec{\hat{\delta y}}}_{(1)}} + \norm{\vec{c}} + \norm{\smash{\vec{\hat{\delta y}}}_{(0)}}) = 3C\kappa u < 141C\kappa u.
    \end{equation*}
    Here, we used the initial conditions $\vec{\hat{y}}_{(0)} = \vec{\hat{y}}_{(1)} = \vec{c}$.
    Now suppose that \cref{eq:fp-error-bound-inductive} holds up to $i$.
    By \cref{lem:inner-noise}, we have the following bound for $j\le i+1$:
    \begin{equation*}
        \norm{\vec{\hat{\delta y}}_{(j)}} \le 4\sqrt{2} \cdot j(1.01\eta)^{j-1} \norm{\vec{c}} + \frac{9}{(1-1.01\eta)^2} \cdot \max_{1\le k \le j-1} \norm{\vec{f}_{(k)}} + \norm{(\Rhat^{-\top}\mat{A}^\top\mat{A}\Rhat^{-1})^{-1}\vec{c}}.
    \end{equation*}
    Using the induction hypothesis \cref{eq:fp-error-bound-inductive}, the hypothesis $\eta\le 0.9$, the numerical inequality $j(1.01)^{j-1} < 11$ for all $j\ge 0$, and the preconditioning result $\sigma_{\rm min}(\mat{A}\Rhat^{-1}) \ge (1+1.01\eta)^{-1}$, we obtain
    \begin{equation*}
        \norm{\smash{\vec{\hat{\delta y}}}_{(j)}} \le 44\sqrt{2} \norm{\vec{c}} + 1087 \cdot (141C \kappa u \norm{\vec{c}}) + 3.7\norm{\vec{c}} \le 70 \norm{\vec{c}} \quad \text{for } j \le i+1.
    \end{equation*}
    In the last bound, we used the assumption $\kappa u \ll 1$.
    Plugging into \cref{eq:f-bound}, we obtain
    \begin{equation*}
        \norm{\vec{f}_{(j+1)}} \le C\kappa u( \norm{\smash{\vec{\hat{\delta y}}}_{(j+1)}} + \norm{\vec{c}} + \norm{\smash{\vec{\hat{\delta y}}}_{(j)}}) \le 141 C\kappa u.
    \end{equation*}
    This completes the proof of the claim \cref{eq:fp-error-bound-inductive}.

    Finally, using \cref{eq:fp-error-bound-inductive} and \cref{lem:inner-noise}, we obtain
    \begin{equation} \label{eq:intermediate-bound}
        \norm{\smash{\vec{\hat{\delta y}}}_{(q)} - (\Rhat^{-\top}\mat{A}^\top\mat{A}\Rhat^{-1})^{-1}\vec{c}} \le 4\sqrt{2} \cdot q (1.01\eta)^{q-1} \norm{\vec{c}} + 141C \kappa u\norm{\vec{c}}.
    \end{equation}
    The function $q\mapsto q(1.01)^{q-1}$ is monotone decreasing for $q\ge 11$.
    There are two cases.
    \begin{itemize}
        \item \textbf{Case 1.} Suppose $q_0 = 11$.
        Then $2 \log(1/(\kappa u)) / \log(1/(1.01\eta)) \le 11$ so $(1.01\eta)^{5.5} \le \kappa u$.
        Thus,
        \begin{equation*}
            q (1.01\eta)^{q-1} \le q_0 (1.01\eta)^{q_0-1} = 11(1.01\eta)^{4.5} (1.01\eta)^{5.5} \le 7.2 \kappa u. 
        \end{equation*}
        Substituting this bound in \cref{eq:intermediate-bound}, we obtain the desired conclusion 
        \begin{equation} \label{eq:desired-conclusion}
            \norm{\smash{\vec{\hat{\delta y}}}_{(q)} - (\Rhat^{-\top}\mat{A}^\top\mat{A}\Rhat^{-1})^{-1}\vec{c}} \lesssim \kappa u \norm{\vec{c}}.
        \end{equation}
        \item \textbf{Case 2.} Suppose $q_0 > 11$.
        Then
        \begin{align*}
            q (1.01\eta)^{q-1} &\le q_0 (1.01\eta)^{q_0-1} = \left(1+2 \frac{\log(1/(\kappa u))}{\log(1/(1.01\eta))}\right) (1.01\eta)^{2 \log(1/(\kappa u))/\log(1/(1.01\eta))} \\
            &= \left(1+2 \frac{\log(1/(\kappa u))}{\log(1/(1.01\eta))}\right) (\kappa u)^2 \le \left(1+21 \log(1/(\kappa u))\right)(\kappa u)^2 \le 2\kappa u.
        \end{align*}
        In the last inequality, we use the hypothesis $\kappa u \ll 1$.
        Substituting this bound in \cref{eq:intermediate-bound}, we obtain the desired conclusion \cref{eq:desired-conclusion}.
    \end{itemize}
    Having exhausted both cases, we conclude that \cref{eq:desired-conclusion} always holds, completing the proof.
\end{proof}

\section{SPIR: Stability analysis} \label{sec:SPIR-analysis}

As discussed in \cref{sec:existing-stability-least-squares} and shown by the numerical experiments in this paper, empirical results suggest that sketch-and-precondition has the following stability properties:
\begin{enumerate}[label=(\roman*)]
    \item With the zero initialization, sketch-and-precondition is numerically unstable.
    \item With the sketch-and-solve initialization $\vec{x}_0$, sketch-and-precondition produces a solution $\vec{x}_1$ that is strongly forward stable, but not backward stable.
    \item When initialized with $\vec{x}_1$, sketch-and-precondition produces a backward stable solution $\vec{x}_2$.
\end{enumerate}
Point (iii) forms the basis for the SPIR method, introduced in this paper.
Previous to this work, no explanation was known for the significant stability differences between different initializations (i)--(iii) for sketch-and-precondition.
In this section, we will explain these different stability properties using \cref{thm:meta-stability}.

\subsection{Lanczos-based implementation of SPIR}

We have now seen two variants of sketch-and-precondition using conjugate gradient and LSQR.
These are the implementations we recommend in practice and, as demonstrated in \cref{fig:spir}, they empirically possess similar stability properties.
For our stability analysis, we shall analyze a \emph{third} implementation of sketch-and-precondition using the Lanczos method as the linear solver, given by replacing conjugate gradient in \cref{alg:SPIR-basic} with the Lanczos linear solve routine in \cref{alg:lanczos}. 

The Lanczos method is most well-known as a method for computing a select number of eigenpairs of a Hermitian matrix $\mat{M}$, but it can also be used to evaluate expressions of the form $f(\mat{M})\vec{c}$.
Here, $f : \set{I} \to \real$ is a function defined on an interval $\set{I} \subseteq \real$, extended to matrix inputs in the standard way.
In particular, with the choice $f(x) = x^{-1}$, the Lanczos method can be used to solve linear systems of equations $\mat{M}\vec{y} = \vec{c}$.
For a symmetric positive definite matrix $\mat{M}$ and \textbf{in exact arithmetic, the Lanczos linear solver and conjugate gradient method produce the same outputs} \cite[p.~44]{Gre97a}.
The Lanczos method for linear solves is shown in \cref{alg:lanczos}.
(Note that we use Lanczos with no additional reorthogonalization, which is known to be stable for computation of $f(\mat{M})\vec{c}$ \cite{DK92,MMS18}.)

\begin{algorithm}[t]
	\caption{Lanczos method for linear solves $\mat{M}^{-1}\vec{c}$} \label{alg:lanczos}
	\begin{algorithmic}[1]
		\Require Matrix--vector product $\vec{z} \mapsto \mat{M}\vec{z}$ subroutine \Call{Apply}{}, right-hand side $\vec{c}$, number of steps $k$
		\Ensure Approximate solution $\vec{y} \approx \mat{M}^{-1}\vec{c}$
		\State $\vec{q}_0 \gets \vec{0}$, $\gamma \gets \norm{\vec{c}}$, $\vec{q}_1\gets \vec{c}/\gamma$, $\beta_1 \gets 0$
        \For{$i = 1,\ldots,k$}
        \State $\vec{q}_{i+1} \gets \Call{Apply}{\vec{q}_i} - \beta_i \vec{q}_{i-1}$
        \State $\alpha_i \gets \vec{q}_{i+1}^\top \vec{q}_i^{\vphantom{\top}}$
        \State $\vec{q}_{i+1} \gets \vec{q}_{i+1} - \alpha_i \vec{q}_i$
        \State $\beta_{i+1} \gets \norm{\vec{q}_{i+1}}$
        \State \textbf{if} $\beta_{i+1} = 0$, set $k\gets i$ and \textbf{break}
        \State $\vec{q}_{i+1} \gets \vec{q}_{i+1} / \beta_{i+1}$
        \EndFor
        \State $\mat{T} \gets \begin{bmatrix} \alpha_1 & \beta_2 \\ 
        \beta_2 & \alpha_2 & \ddots \\
        &\ddots & \ddots & \beta_k \\ && \beta_k & \alpha_k\end{bmatrix}$, $\mat{Q} \gets \begin{bmatrix} \vec{q}_1 & \cdots & \vec{q}_k \end{bmatrix}$
        \State $\vec{y} \gets \gamma \cdot \mat{Q} (\mat{T}^{-1}\mathbf{e}_1)$
	\end{algorithmic}
\end{algorithm}

\subsection{Stability of SPIR}

Now, we turn our attention to analyzing the stability of SPIR (\cref{alg:SPIR-basic}) with Lanczos linear solves.
There exists a significant gap between the observed performance of the Lanczos method in practice and the best-known theoretical results for the method in finite precision.
While the existing bounds are crude, they are sufficient for our purposes to establish backward stability of SPIR.
Specifically, we will make use of the following simplification and special case of the results developed in \cite[sec.~6]{MMS18}:

\begin{fact}[Stability of Lanczos linear solves] \label{fact:stable-lanczos}
    Let $\mat{M} \in \real^{n\times n}$ be a positive definite matrix and assume $\mat{M}$ is well-conditioned, $\cond(\mat{M}) \le 40$.
    Assume the \Call{Apply}{} operation satisfies
    \begin{equation} \label{eq:stable-apply}
        \norm{\fl(\Call{Apply}{\vec{z}}) - \mat{M}\vec{z}} \lesssim \norm{\mat{M}}\norm{\vec{z}}\tilde{u}
    \end{equation}
    for some effective precision $\tilde{u} > 0$ and assume that $n\ge \log(1/\tilde{u})$.
    Then the Lanczos algorithm \cref{alg:lanczos} ran for $k = \order(\log(1/\tilde{u}))$ steps produces an output satisfying
    \begin{equation} \label{eq:lanczos-stability}
        \norm{\mat{M}} \cdot \norm{\fl(\Call{LanczosLinearSolve}{\textsc{Apply},\vec{c}}) - \mat{M}^{-1}\vec{c}} \lesssim \norm{\vec{c}} \tilde{u}.
    \end{equation}
\end{fact}

We describe how \cref{fact:stable-lanczos} can be distilled from \cite[sec.~6]{MMS18} in \cref{app:MMS}.
A slightly stronger version of this result, better in ways suppressed by our $\lesssim$ notation, appears in \cite[Thm.~2.2]{DGK98}.
With \cref{fact:stable-lanczos} in place, we analyze SPIR:

\begin{theorem}[SPIR: backward stability] \label{thm:spir-stability}
    Assume conditions (A)--(C) of \cref{thm:meta-stability} and consider SPIR (\cref{alg:SPIR-basic}) run with Lanczos linear solves (\cref{alg:lanczos}) for $k = \order(\log(1/\kappa u))$ steps.
    Assume $n \ge \log(1/\kappa u)$.
    Then:
    \begin{enumerate}
        \item The first solution $\vec{x}_1$ (i.e., sketch-and-precondition with the sketch-and-solve initialization) is strongly forward stable, and
        \item The output $\vec{x}_2$ of SPIR is backward stable.
    \end{enumerate}
    In particular, choosing $\mat{S}$ to be a sparse sign embedding with distortion $\eta = 1/2$ with the parameter settings \cref{eq:cohen}, SPIR produces a backward stable solution in $\order(mn\log(n/u) + n^3\log n)$ operations.
\end{theorem}

\begin{proof}
    Define $\mat{M} = \Rhat^{-\top}\mat{A}^\top\mat{A}\Rhat^{-1}$. 
    By \cref{lem:stability-sketching} and the hypothesis  $\kappa u \ll 1$, $\cond(\mat{M}) \le 40$.
    By \cref{prop:whitened-stability}, the apply operation $\Call{Apply}{\vec{c}} = \Rhat^{-\top}(\mat{A}^\top(\mat{A}(\Rhat^{-1}\vec{c})))$ satisfies the stability guarantee \cref{eq:stable-apply} with $\tilde{u} = \kappa u$.
    Ergo, \cref{fact:stable-lanczos} implies that
    \begin{equation*}
        \norm{\fl(\Call{LanczosLinearSolve}{\textsc{Apply},\vec{c}}) - \mat{M}^{-1}\vec{c}} \lesssim \kappa u\norm{\vec{c}}.
    \end{equation*}
    Thus, the hypotheses of \cref{thm:meta-stability} hold, from which the desired conclusion follows.
\end{proof}

\Cref{thm:spir-stability} effectively resolves the previously mysterious stability properties of sketch-and-precondition, at least up to the use of Lanczos in place of conjugate gradient or LSQR.
Since Lanczos, conjugate gradient, and LSQR are equivalent in exact arithmetic and have similar stabilities in practice (in our experiments, at least), we believe \cref{thm:spir-stability} provides stong evidence for the backward stability of conjugate gradient- and LSQR-based implementations of SPIR.

\section*{Acknowledgements}

We thank Deeksha Adil, Anne Greenbaum, Christopher Musco, Mert Pilanci, Joel Tropp, Madeleine Udell, and Robert Webber for helpful discussions.
ENE acknowledges support by the DOE CSGF.\footnote[1]{This report was prepared as an account of work sponsored by an agency of the United States Government. Neither the United States Government nor any agency thereof, nor any of their employees, makes any warranty, express or implied, or assumes any legal liability or responsibility for the accuracy, completeness, or usefulness of any information, apparatus, product, or process disclosed, or represents that its use would not infringe privately owned rights. Reference herein to any specific commercial product, process, or service by trade name, trademark, manufacturer, or otherwise does not necessarily constitute or imply its endorsement, recommendation, or favoring by the United States Government or any agency thereof. The views and opinions of authors expressed herein do not necessarily state or reflect those of the United States Government or any agency thereof.}

\bibliographystyle{amsplain}
\bibliography{refs}

\providecommand{\bysame}{\leavevmode\hbox to3em{\hrulefill}\thinspace}
\providecommand{\MR}{\relax\ifhmode\unskip\space\fi MR }
\providecommand{\MRhref}[2]{%
  \href{http://www.ams.org/mathscinet-getitem?mr=#1}{#2}
}
\providecommand{\href}[2]{#2}
\begin{thebibliography}{10}

\bibitem{AMT10}
Haim Avron, Petar Maymounkov, and Sivan Toledo, \emph{Blendenpik:
  {{Supercharging LAPACK}}'s {{Least-Squares Solver}}}, SIAM Journal on
  Scientific Computing \textbf{32} (2010), no.~3, 1217--1236.

\bibitem{BSW14}
Pierre Baldi, Peter Sadowski, and Daniel Whiteson, \emph{Searching for exotic
  particles in high-energy physics with deep learning}, Nature Communications
  \textbf{5} (2014), no.~1, 4308.

\bibitem{Bjo96}
{\AA}ke Bj{\"o}rck, \emph{Numerical methods for least squares problems}, SIAM,
  1996.

\bibitem{CEMT25}
Chris Cama\~no, Epperly~E. Epperly, Raphael~A. Meyer, and Joel~A. Tropp,
  \emph{Faster linear algebra with structured random matrices}, Manuscript in
  preparation (2025).

\bibitem{carson2024comparison}
Erin Carson and Ieva Dau{\v z}ickait{\.e}, \emph{A comparison of mixed
  precision iterative refinement approaches for least-squares problems}, SIAM
  Journal on Matrix Analysis and Applications \textbf{46} (2025), no.~2,
  1117--1144.

\bibitem{CFS21}
Coralia Cartis, Jan Fiala, and Zhen Shao, \emph{Hashing embeddings of optimal
  dimension, with applications to linear least squares}, arXiv preprint
  arXiv:2105.11815 (2021).

\bibitem{CDD25}
Shabarish Chenakkod, Micha{\l} Derezi{\'n}ski, and Xiaoyu Dong, \emph{Optimal
  oblivious subspace embeddings with near-optimal sparsity}, Proceedings of the
  52nd EATCS International Colloquium on Automata, Languages, and Programming,
  2025, Available at arXiv:2411.08773.

\bibitem{CDDR23}
Shabarish Chenakkod, Micha{\l} Derezi{\'n}ski, Xiaoyu Dong, and Mark Rudelson,
  \emph{Optimal embedding dimension for sparse subspace embeddings},
  Proceedings of the 56th {{Annual ACM Symposium}} on {{Theory}} of
  {{Computing}}, ACM, June 2024, pp.~1106--1117.

\bibitem{CDD+23}
Younghyun Cho, James Demmel, Micha{\l} Derezi{\'n}ski, Haoyun Li, Hengrui Luo,
  Michael Mahoney, and Riley Murray, \emph{Surrogate-based autotuning for
  randomized sketching algorithms in regression problems}, SIAM Journal on
  Matrix Analysis and Applications \textbf{46} (2025), no.~2, 1247--1279.

\bibitem{Coh16}
Michael~B. Cohen, \emph{Nearly tight oblivious subspace embeddings by trace
  inequalities}, Proceedings of the {{Twenty-Seventh Annual ACM-SIAM
  Symposium}} on {{Discrete Algorithms}}, {Society for Industrial and Applied
  Mathematics}, January 2016, pp.~278--287.

\bibitem{DH11}
Timothy Davis and Yifan Hu, \emph{The {{University}} of {{Florida}} sparse
  matrix collection}, ACM Transactions on Mathematical Software (TOMS)
  \textbf{38} (2011), no.~1, 1--25.

\bibitem{DDH07}
James Demmel, Ioana Dumitriu, and Olga Holtz, \emph{Fast linear algebra is
  stable}, Numerische Mathematik \textbf{108} (2007), no.~1, 59--91.

\bibitem{DEF+23}
Mateo D{\'i}az, Ethan~N. Epperly, Zachary Frangella, Joel~A. Tropp, and
  Robert~J. Webber, \emph{Robust, randomized preconditioning for kernel ridge
  regression}, arXiv preprint arXiv:2304.12465 (2023).

\bibitem{DELZ24}
Zhiyan Ding, Ethan~N. Epperly, Lin Lin, and Ruizhe Zhang, \emph{The {{ESPRIT}}
  algorithm under high noise: {{Optimal}} error scaling and noisy
  super-resolution}, Proceedings of the 65th {{Annual Symposium}} on
  {{Foundations}} of {{Computer Science}}, 2024, pp.~2344--2366.

\bibitem{DM23}
Yijun Dong and Per-Gunnar Martinsson, \emph{Simpler is better: A comparative
  study of randomized pivoting algorithms for {{CUR}} and interpolative
  decompositions}, Advances in Computational Mathematics \textbf{49} (2023),
  no.~4, 66.

\bibitem{drineas2011faster}
Petros Drineas, Michael~W Mahoney, Shan Muthukrishnan, and Tam{\'a}s
  Sarl{\'o}s, \emph{Faster least squares approximation}, Numer. Math.
  \textbf{117} (2011), no.~2, 219--249.

\bibitem{DGK98}
Vladimir~L. Druskin, Anne Greenbaum, and Leonid~A. Knizhnerman, \emph{Using
  nonorthogonal {Lanczos} vectors in the computation of matrix functions}, SIAM
  Journal on Scientific Computing \textbf{19} (1998), no.~1, 38--54.

\bibitem{DK92}
Vladimir~L. Druskin and Leonid~A. Knizhnerman, \emph{Error bounds in the simple
  {{Lanczos}} procedure for computing functions of symmetric matrices and
  eigenvalues}, Computational Mathematics and Mathematical Physics \textbf{31}
  (1992), no.~7, 20--30.

\bibitem{Epp24}
Ethan~N. Epperly, \emph{Fast and forward stable randomized algorithms for
  linear least-squares problems}, SIAM Journal on Matrix Analysis and
  Applications (2024), 1782--1804.

\bibitem{Epp25}
\bysame, \emph{Make the most of what you have: Resource-efficient randomized
  algorithms for matrix computations}, Ph.D. thesis, California Institute of
  Technology, 2025.

\bibitem{epperly2025stable}
Ethan~N Epperly, Anne Greenbaum, and Yuji Nakatsukasa, \emph{Stable algorithms
  for general linear systems by preconditioning the normal equations}, arXiv
  preprint arXiv:2502.17767 (2025).

\bibitem{ET24}
Ethan~N. Epperly and Joel~A. Tropp, \emph{Efficient error and variance
  estimation for randomized matrix computations}, SIAM Journal on Scientific
  Computing \textbf{46} (2024), no.~1, A508--A528.

\bibitem{FM11}
David Chin-Lung Fong and Michael Saunders, \emph{{LSMR: A}n iterative algorithm
  for sparse least-squares problems}, SIAM Journal on Scientific Computing
  \textbf{33} (2011), no.~5, 2950--2971.

\bibitem{golub1966note}
Gene~H. Golub and James~H. Wilkinson, \emph{Note on the iterative refinement of
  least squares solution}, Numer. Math. \textbf{9} (1966), no.~2, 139--148.

\bibitem{GJT12}
Serge Gratton, Pavel Jir{\'a}nek, and David {Titley-Peloquin}, \emph{On the
  accuracy of the {{Karlson--Wald{\'e}n}} estimate of the backward error for
  linear least squares problems}, SIAM Journal on Matrix Analysis and
  Applications \textbf{33} (2012), no.~3, 822--836.

\bibitem{Grc03}
Joseph~F. Grcar, \emph{Optimal sensitivity analysis of linear least squares},
  Lawrence Berkeley National Laboratory, Report LBNL-52434 \textbf{99} (2003).

\bibitem{GSS07}
Joseph~F. Grcar, Michael~A. Saunders, and Zheng Su, \emph{Estimates of optimal
  backward perturbations for linear least squares problems}, Tech. report,
  Lawrence Berkeley National Lab, 2007.

\bibitem{Gre97a}
Anne Greenbaum, \emph{Iterative methods for solving linear systems}, Frontiers
  in Applied Mathematics, no.~17, {Society for Industrial and Applied
  Mathematics}, Philadelphia, PA, 1997.

\bibitem{Gu98}
Ming Gu, \emph{Backward perturbation bounds for linear least squares problems},
  SIAM Journal on Matrix Analysis and Applications \textbf{20} (1998), no.~2,
  363--372.

\bibitem{HMT11}
Nathan Halko, Per-Gunnar Martinsson, and Joel~A. Tropp, \emph{Finding structure
  with randomness: {{Probabilistic}} algorithms for constructing approximate
  matrix decompositions}, SIAM Review \textbf{53} (2011), no.~2, 217--288.

\bibitem{HS52}
Magnus~Rudolph Hestenes and Eduard Stiefel, \emph{Methods of conjugate
  gradients for solving linear systems}, vol.~49, NBS Washington, DC, 1952.

\bibitem{Hig02}
Nicholas~J. Higham, \emph{Accuracy and stability of numerical algorithms},
  SIAM, 2002.

\bibitem{HM19}
Nicholas~J. Higham and Theo Mary, \emph{A new approach to probabilistic
  rounding error analysis}, SIAM Journal on Scientific Computing \textbf{41}
  (2019), no.~5, A2815--A2835.

\bibitem{KW97}
Rune Karlson and Bertil Wald{\'e}n, \emph{Estimation of optimal backward
  perturbation bounds for the linear least squares problem}, BIT Numerical
  Mathematics \textbf{37} (1997), no.~4, 862--869.

\bibitem{KT24}
Anastasia Kireeva and Joel~A. Tropp, \emph{Randomized matrix computations:
  {{Themes}} and variations}, 2023 CIME Summer School on Machine Learning
  (2024), Available at arXiv:2402.17873.

\bibitem{KMC+22}
Katherine Klymko, Carlos {Mejuto-Zaera}, Stephen~J. Cotton, Filip Wudarski,
  Miroslav Urbanek, Diptarka Hait, Martin {Head-Gordon}, K.~Birgitta Whaley,
  Jonathan Moussa, Nathan Wiebe, Wibe~A. {de Jong}, and Norm~M. Tubman,
  \emph{Real-time evolution for ultracompact {{Hamiltonian}} eigenstates on
  quantum hardware}, PRX Quantum \textbf{3} (2022), no.~2, 020323.

\bibitem{LP20}
Jonathan Lacotte and Mert Pilanci, \emph{Optimal randomized first-order methods
  for least-squares problems}, Proceedings of the 37th {{International
  Conference}} on {{Machine Learning}}, 2020.

\bibitem{LP21}
\bysame, \emph{Faster least squares optimization}, arXiv preprint
  arXiv:1911.02675 (2021).

\bibitem{LNY23a}
Haoya Li, Hongkang Ni, and Lexing Ying, \emph{Adaptive low-depth quantum
  algorithms for robust multiple-phase estimation}, Physical Review A
  \textbf{108} (2023), no.~6, 062408.

\bibitem{MT20}
Per-Gunnar Martinsson and Joel~A. Tropp, \emph{Randomized numerical linear
  algebra: {{Foundations}} and algorithms}, Acta Numerica \textbf{29} (2020),
  403--572.

\bibitem{MNTW24}
Maike Meier, Yuji Nakatsukasa, Alex Townsend, and Marcus Webb, \emph{Are
  sketch-and-precondition least squares solvers numerically stable?}, SIAM
  Journal on Matrix Analysis and Applications \textbf{45} (2024), no.~2,
  905--929.

\bibitem{MSM14}
Xiangrui Meng, Michael~A. Saunders, and Michael~W. Mahoney, \emph{{{LSRN}}:
  {{A}} parallel iterative solver for strongly over- or underdetermined
  systems}, SIAM Journal on Scientific Computing \textbf{36} (2014), no.~2,
  C95--C118.

\bibitem{MMS18}
Cameron Musco, Christopher Musco, and Aaron Sidford, \emph{Stability of the
  {{Lanczos}} method for matrix function approximation}, Proceedings of the
  2018 {{Annual ACM-SIAM Symposium}} on {{Discrete Algorithms}}, SIAM, January
  2018, pp.~1605--1624.

\bibitem{NT24}
Yuji Nakatsukasa and Joel~A. Tropp, \emph{Fast and accurate randomized
  algorithms for linear systems and eigenvalue problems}, SIAM Journal on
  Matrix Analysis and Applications \textbf{45} (2024), no.~2, 1183--1214.

\bibitem{OPA19}
Ibrahim~Kurban Ozaslan, Mert Pilanci, and Orhan Arikan, \emph{Iterative
  {{Hessian}} sketch with momentum}, {IEEE} International Conference on
  Acoustics, Speech and Signal Processing, 2019, pp.~7470--7474.

\bibitem{PS75}
Christopher~C Paige and Michael~A Saunders, \emph{Solution of sparse indefinite
  systems of linear equations}, SIAM journal on numerical analysis \textbf{12}
  (1975), no.~4, 617--629.

\bibitem{PS82a}
Christopher~C. Paige and Michael~A. Saunders, \emph{Algorithm 583: {{LSQR}}:
  Sparse linear equations and least squares problems}, ACM Transactions on
  Mathematical Software \textbf{8} (1982), no.~2, 195--209.

\bibitem{parlettsym}
Beresford~N. Parlett, \emph{The symmetric eigenvalue problem}, SIAM, 1998.

\bibitem{PW16}
Mert Pilanci and Martin~J. Wainwright, \emph{Iterative {{Hessian}} sketch:
  {{Fast}} and accurate solution approximation for constrained least-squares},
  The Journal of Machine Learning Research \textbf{17} (2016), no.~1,
  1842--1879.

\bibitem{Pol64}
Boris~T. Polyak, \emph{Some methods of speeding up the convergence of iteration
  methods}, Ussr computational mathematics and mathematical physics \textbf{4}
  (1964), no.~5, 1--17.

\bibitem{RT08}
Vladimir Rokhlin and Mark Tygert, \emph{A fast randomized algorithm for
  overdetermined linear least-squares regression}, Proceedings of the National
  Academy of Sciences \textbf{105} (2008), no.~36, 13212--13217.

\bibitem{Sar06}
Tam{\'a}s Sarl{\'o}s, \emph{Improved approximation algorithms for large
  matrices via random projections}, 2006 47th {{Annual IEEE Symposium}} on
  {{Foundations}} of {{Computer Science}}, October 2006, pp.~143--152.

\bibitem{Stewart1}
Gilbert~W. Stewart, \emph{Matrix algorithms volume {I}: Basic decompositions},
  SIAM, 1998.

\bibitem{Tre19}
Lloyd~N. Trefethen, \emph{Approximation {{Theory}} and {{Approximation
  Practice}}, {{Extended Edition}}}, SIAM, Philadelphia, PA, January 2019.

\bibitem{Tro25}
Joel~A. Tropp, \emph{Comparison theorems for the minimum eigenvalue of a random
  positive-semidefinite matrix}, arXiv preprint arXiv:2501.16578 (2025).

\bibitem{TW23}
Joel~A. Tropp and Robert~J. Webber, \emph{Randomized algorithms for low-rank
  matrix approximation: {{Design}}, analysis, and applications}, arXiv preprint
  arXiv:2306.12418 (2023).

\bibitem{TYUC19}
Joel~A. Tropp, Alp Yurtsever, Madeleine Udell, and Volkan Cevher,
  \emph{Streaming low-rank matrix approximation with an application to
  scientific simulation}, SIAM Journal on Scientific Computing \textbf{41}
  (2019), no.~4, A2430--A2463.

\bibitem{sluis1969condition}
Abraham van~der Sluis, \emph{Condition numbers and equilibration of matrices},
  Numerische Mathematik \textbf{14} (1969), no.~1, 14--23.

\bibitem{WKS95}
Bertil Wald{\'e}n, Rune Karlson, and Ji-Guang Sun, \emph{Optimal backward
  perturbation bounds for the linear least squares problem}, Numerical Linear
  Algebra with Applications \textbf{2} (1995), no.~3, 271--286.

\bibitem{Wed73}
Per-{\AA}ke Wedin, \emph{Perturbation theory for pseudo-inverses}, BIT
  Numerical Mathematics \textbf{13} (1973), no.~2, 217--232.

\bibitem{WEB24}
Heather Wilber, Ethan~N. Epperly, and Alex~H. Barnett, \emph{Superfast direct
  inversion of the nonuniform discrete fourier transform via hierarchically
  semiseparable least squares}, SIAM Journal on Scientific Computing (2025),
  A1702--A1732.

\bibitem{Wil63}
James~H Wilkinson, \emph{{Rounding Errors in Algebraic Processes}},
  Prentice-Hall, 1963.

\bibitem{XXCB14}
Yuanzhe Xi, Jianlin Xia, Stephen Cauley, and Venkataramanan Balakrishnan,
  \emph{Superfast and stable structured solvers for {{Toeplitz}} least squares
  via randomized sampling}, SIAM Journal on Matrix Analysis and Applications
  \textbf{35} (2014), no.~1, 44--72.

\end{thebibliography}

\appendix

\section{Auxilliary proofs} \label{app:FOSSILS}

In this section, we prove \cref{thm:backward-componentwise,cor:proving-backward,prop:sketched-KW}.

\subsection{\texorpdfstring{Proof of \cref{thm:backward-componentwise}}{Proof of Theorem 2.8}} \label{app:backward-componentwise}

    Let us compute the Karlson--Wald\'en estimate $\hat{\BE}_1(\vec{\hat{x}})$.
    Expanding $\vec{\hat{x}} = \sum_{i=1}^n (\vec{v}_i^\top \vec{\hat{x}})\vec{v}_i^{\vphantom{\top}}$, we have
    \begin{equation*}
        \mat{A}^\top(\vec{b} - \mat{A}\vec{\hat{x}}) = \sum_{i=1}^n \sigma_i^{\vphantom{\top}} \left(\vec{u}_i^\top \vec{b} - \sigma_i^{\vphantom{\top}} (\vec{v}_i^\top \vec{\hat{x}})\right) \vec{v}_i^{\vphantom{\top}} = \mat{A}^\top(\vec{b} - \mat{A}\vec{\hat{x}}) = \sum_{i=1}^n \sigma_i^2  \vec{v}_i^\top (\vec{x}- \vec{\hat{x}}) \cdot \vec{v}_i^{\vphantom{\top}}.
    \end{equation*}
    For the second equality, we used the fact that the true solution $\vec{x}$ to the least-squares proble satisfies $\vec{u}_i^\top \vec{b} = \sigma_i^{\vphantom{\top}} \vec{v}_i^\top \vec{x}$.
    We now evaluate the matrix expression in the Karlson--Wald\'en estimate:
    \begin{equation*}
        \left( \mat{A}^\top\mat{A} + \frac{\norm{\vec{b} - \mat{A}\vec{\hat{x}}}^2}{1+\norm{\vec{\hat{x}}}^2} \Id \right)^{-1/2} = \sum_{i=1}^n \left( \sigma_i^2 + \frac{\norm{\vec{b} - \mat{A}\vec{\hat{x}}}^2}{1+\norm{\vec{\hat{x}}}^2} \right)^{-1/2} \vec{v}_i^{\vphantom{\top}}\vec{v}_i^\top
    \end{equation*}
    Combining the two previous displays and simplifying yields
    \begin{equation} \label{eq:KW-expanded}
        \hat{\BE}_1(\vec{\hat{x}})^2 = \sum_{i=1}^n \frac{\sigma_i^4}{(1+\norm{\vec{\hat{x}}}^2)\sigma_i^2 +\norm{\vec{b} - \mat{A}\vec{\hat{x}}}^2}\left(\vec{v}_i^\top (\vec{x}- \vec{\hat{x}})\right)^2.
    \end{equation}
    
    Now, we make a chain of deductions.
    By \cref{fact:KW-estimate}, $\BE_1(\vec{\hat{x}}) \lesssim u$ if and only if $\hat{\BE}_1(\vec{\hat{x}})\lesssim u$.
    This, in turn, occurs if and only if each summand in \cref{eq:KW-expanded} satisfies
    \begin{equation}\label{eq:simplified-backward-condition}
        \frac{\sigma_i^4}{(1+\norm{\vec{\hat{x}}}^2)\sigma_i^2 +\norm{\vec{b} - \mat{A}\vec{\hat{x}}}^2}\left(\vec{v}_i^\top (\vec{x}- \vec{\hat{x}})\right)^2\lesssim u^2 \quad \text{for } i =1,2,\ldots,n.
    \end{equation}
    Rearranging and using the fact that $\sqrt{\alpha + \beta}$ and $\sqrt{\alpha} + \sqrt{\beta}$ are within a factor $\sqrt{2}$, we see that \cref{eq:simplified-backward-condition} holds if and only if
    \begin{equation*}
        \left|\vec{v}_i^\top (\vec{x} - \vec{\hat{x}})\right| \lesssim \frac{1}{\sigma_i} (1 + \norm{\vec{\hat{x}}})u + \frac{1}{\sigma_i^2} \norm{\vec{b} - \mat{A}\vec{\hat{x}}}u.
    \end{equation*}
    This is the desired conclusion \cref{eq:componentwise-error}, under the standing normalization $\norm{\mat{A}} = \norm{\vec{b}} = 1$. \hfill $\qed$

\subsection{\texorpdfstring{\rr{Proof of \cref{cor:forward-componentwise}}}{Proof of Proposition 2.9}} \label{sec:forward-componentwise}

\rr{First, assume (1).
Then}
\begin{equation*}
    \rr{\left| \vec{v}_i^\top (\vec{\hat{x}} - \vec{x}) \right| = \sigma_i^{-1}\cdot \left| \vec{u}_i^\top \mat{A}(\vec{\hat{x}} - \vec{x}) \right| \le \sigma_i^{-1} \cdot \norm{\mat{A}(\vec{\hat{x}} - \vec{x})}.}
\end{equation*}
\rr{By strong forward stability, the right-hand side is bounded as
\begin{equation} \label{eq:first-res-bound}
    \left| \vec{v}_i^\top (\vec{\hat{x}} - \vec{x}) \right| \lesssim \left(\sigma_i^{-1} \cdot  \norm{\vec{x}} + \sigma_i^{-1}\sigma_n^{-1} \cdot \norm{\vec{b} - \mat{A}\vec{x}}\right) u.
\end{equation}
This establishes (2).}

\rr{Now assume (2).
(Equivalently, assume \cref{eq:first-res-bound}.)
We may bound $\norm{\vec{x}}$ as
\begin{equation} \label{eq:norm-bound}
    \norm{\vec{x}} \le \norm{\vec{\hat{x}}} + \norm{\vec{\hat{x}} - \vec{x}} \le \norm{\vec{\hat{x}}} + \sigma_n^{-1} \cdot \norm{\mat{A}(\vec{\hat{x}} - \vec{x})} \le \norm{\vec{\hat{x}}} + \sigma_n^{-1} \cdot \norm{\vec{b} - \mat{A}\vec{\hat{x}}}.
\end{equation}
The first inequality is the triangle inequality, the second inequality follows since $\sigma_n$ is the minimum singular value of $\mat{A}$, and the last inequality follows from the Pythagorean identity
\begin{equation*}
    \norm{\vec{b} - \mat{A}\vec{\hat{x}}}^2 = \norm{\vec{b} - \mat{A}\vec{x}}^2 + \norm{\mat{A}(\vec{\hat{x}} - \vec{x})}^2.
\end{equation*}
Substituting \cref{eq:norm-bound} into \cref{eq:first-res-bound} and using the optimality of the residual $\norm{\vec{b} - \mat{A}\vec{x}} \le \norm{\vec{b} - \mat{A}\vec{\hat{x}}}$, we obtain
\begin{equation*}
    \left| \vec{v}_i^\top (\vec{\hat{x}} - \vec{x}) \right| \lesssim \left(\sigma_i^{-1} \cdot  \norm{\vec{\hat{x}}} + \sigma_i^{-1}\sigma_n^{-1} \cdot \norm{\vec{b} - \mat{A}\vec{\hat{x}}}\right) u.
\end{equation*}
Therefore, \cref{eq:strong-forward-componentwise} holds, showing (3).}

\rr{Now suppose (3).
Then,
\begin{align*}
    \norm{\mat{A}(\vec{\hat{x}}-\vec{x})} &= \left| \sum_{i=1}^n \left[\sigma_i \cdot \vec{u}_i^{\vphantom{\top}} \vec{v}_i^\top (\vec{\hat{x}} - \vec{x})\right] \right| \le \sum_{i=1}^n \left[\sigma_i \cdot \left|\vec{v}_i^\top (\vec{\hat{x}} - \vec{x})\right] \right| \\
    &\le \left(\sum_{i=1}^n [(1+\norm{\vec{\hat{x}}}) + \sigma_n^{-1} \cdot \norm{\vec{b} - \mat{A}\vec{\hat{x}}}]\right)u \lesssim \left[1 + \norm{\vec{\hat{x}}} + \sigma_n^{-1} \cdot \norm{\vec{b} - \mat{A}\vec{\hat{x}}}\right]u
\end{align*}
Using the same argument as \cref{eq:norm-bound}, we have $\norm{\vec{\hat{x}}} \le \norm{\vec{x}} + \sigma_n^{-1} \cdot \norm{\vec{b} - \mat{A}\vec{\hat{x}}}$.
Thus,
\begin{equation*}
    \norm{\mat{A}(\vec{\hat{x}}-\vec{x})} \lesssim \left[1 + \norm{\vec{x}} + \sigma_n^{-1} \cdot \norm{\vec{b} - \mat{A}\vec{\hat{x}}}\right]u \lesssim \left[\norm{\vec{x}} + \sigma_n^{-1} \cdot \norm{\vec{b} - \mat{A}\vec{x}}\right]u + \sigma_n^{-1} u \cdot \norm{\mat{A}(\vec{\hat{x}}-\vec{x})}.
\end{equation*}
In the second inequality, we use the triangle inequality and the inequality 
\begin{equation*}
    1 = \norm{\vec{b}} \le \norm{\mat{A}\vec{x}} + \norm{\vec{b} - \mat{A}\vec{x}} \le \norm{\vec{x}} + \norm{\vec{b} - \mat{A}\vec{x}}.
\end{equation*}
Using the hypothesis $\sigma_n^{-1} u \ll 1$, we have the bound
\begin{equation*}
    \norm{\mat{A}(\vec{\hat{x}}-\vec{x})} \le c \left[\norm{\vec{x}} + \sigma_n^{-1} \cdot \norm{\vec{b} - \mat{A}\vec{x}}\right]u + \frac{1}{2}\norm{\mat{A}(\vec{\hat{x}}-\vec{x})}.
\end{equation*}
for a prefactor $c$ polynomially large in $m$ and $n$.
Thus,
\begin{equation*}
    \norm{\mat{A}(\vec{\hat{x}}-\vec{x})} \le 2c\left[\norm{\vec{x}} + \sigma_n^{-1} \cdot \norm{\vec{b} - \mat{A}\vec{x}}\right]u,
\end{equation*}
so $\vec{\hat{x}}$ is strongly forward stable, showing (1).\hfill $\qed$}

\subsection{\texorpdfstring{\rr{Proof of \cref{cor:strong-forward-enough}}}{Proof of Corollary 2.10}} \label{sec:strong-forward-enough}

\rr{Suppose $\vec{\hat{x}}$ is strongly forward stable.
By \cref{cor:forward-componentwise}, it follows that
\begin{equation*}
    \left| \vec{v}_i^\top (\vec{\hat{x}} - \vec{x}) \right| \lesssim \left(\sigma_i^{-1} \cdot  \norm{\vec{x}} + \sigma_i^{-1}\sigma_n^{-1} \cdot \norm{\vec{b} - \mat{A}\vec{x}}\right) u \quad \text{for } i=1,2\ldots,n.
\end{equation*}
By assumption, $\sigma_n^{-1} \cdot \norm{\vec{b} - \mat{A}\vec{x}} \lesssim 1 + \norm{\vec{x}}$.
Thus,
\begin{equation*}
    \left| \vec{v}_i^\top (\vec{\hat{x}} - \vec{x}) \right| \lesssim \sigma_i^{-1} \cdot (1+\norm{\vec{x}})u \quad \text{for } i=1,2\ldots,n.
\end{equation*}
In particular, using the hypothesis $\cond(\mat{A})u \ll 1$, we have $\norm{\vec{\hat{x}} - \vec{x}} \le \cond(\mat{A}))u \cdot (1+\norm{\vec{x}})\le 0.5 + 0.5\norm{\vec{x}}$.
Thus, 
\begin{equation*}
    \norm{\vec{x}} \le \norm{\vec{\hat{x}} - \vec{x}} + \norm{\vec{\hat{x}}} \le 0.5 + 0.5\norm{\vec{x}}+ \norm{\vec{\hat{x}}}.
\end{equation*}
Ergo, $\norm{\vec{x}} \le 1 + 2 \norm{\vec{\hat{x}}}$. 
We conclude
\begin{equation*}
    \left| \vec{v}_i^\top (\vec{\hat{x}} - \vec{x}) \right| \lesssim \sigma_i^{-1} \cdot (1+\norm{\vec{\hat{x}}})u \quad \text{for } i=1,2\ldots,n,
\end{equation*}
so $\vec{\hat{x}}$ is backward stable by \cref{thm:backward-componentwise}.\hfill$\qed$}

\subsection{\texorpdfstring{Proof of \cref{prop:sketched-KW}}{Proof of Proposition 4.2}} \label{app:sketched-KW-proof}

    The proof is essentially the same as \cite[Thm.~2.3]{DEF+23}.
    For conciseness, drop the subscript $\theta$ for the quantities $\hat{\BE}_{\theta}$ and $\hat{\BE}_{\theta,{\rm sk}}$.
    The squared ratio of the Karlson--Wald\'en estimate and its sketched variant is
    \begin{equation*}
        \left(\frac{\hat{\BE}_{\rm sk}(\vec{\hat{x}})}{\hat{\BE}(\vec{\hat{x}})}\right)^2 = \frac{\vec{z}^\top \left( (\mat{S}\mat{A})^\top(\mat{S}\mat{A}) + \alpha \Id \right)^{-1}\vec{z}}{\vec{z}^\top \left( \mat{A}^\top\mat{A} + \alpha \Id \right)^{-1}\vec{z}} \quad \text{with } \alpha = \frac{\theta^2 \norm{\vec{b} - \mat{A}\vec{\hat{x}}}^2}{1 + \theta^2 \norm{\vec{\hat{x}}}^2}, \: \vec{z} = \mat{A}^\top(\vec{b}-\mat{A}\vec{\hat{x}}).
    \end{equation*}
    By the Rayleigh--Ritz principle \cite[Ch.~8]{parlettsym} and some algebra, we have
    \begin{equation*}
        \left(\frac{\hat{\BE}(\vec{\hat{x}})}{\hat{\BE}_{\rm sk}(\vec{\hat{x}})}\right)^2\ge \lambda_{\rm min} \left( \left( \mat{A}^\top\mat{A} + \alpha \Id \right)^{-1/2} \left( (\mat{S}\mat{A})^\top(\mat{S}\mat{A}) + \alpha \Id \right) \left( \mat{A}^\top\mat{A} + \alpha \Id \right)^{-1/2} \right).
    \end{equation*}
    Thus, by the Rayleigh--Ritz principle, we have
    \begin{equation*}
        \left(\frac{\hat{\BE}(\vec{\hat{x}})}{\hat{\BE}_{\rm sk}(\vec{\hat{x}})}\right)^2 \ge \min_{\vec{w} \ne \vec{0}} \frac{\vec{w}^\top \left( (\mat{S}\mat{A})^\top(\mat{S}\mat{A}) + \alpha \Id \right)\vec{w}}{\vec{w}^\top \left( \mat{A}^\top\mat{A} + \alpha \Id \right) \vec{w}} \ge \min \left\{ \min_{\vec{w} \ne \vec{0}} \frac{\norm{\mat{S}\mat{A} \vec{w}}^2}{\norm{\mat{A} \vec{w}}^2}, 1 \right\}.
    \end{equation*}
    The second inequality is a consequence of the numerical inequality $(a+x)/(1+x) \ge \min \{1,a\}$.
    Finally, invoking the \QR decomposition \cref{eq:sketch_qr} and \cref{fact:whitening}, we continue
    \begin{equation*}
        \left(\frac{\hat{\BE}(\vec{\hat{x}})}{\hat{\BE}_{\rm sk}(\vec{\hat{x}})}\right)^2 \ge \min \left\{ \min_{\vec{w} \ne \vec{0}} \frac{\norm{\mat{R} \vec{w}}^2}{\norm{\mat{A} \vec{w}}^2}, 1 \right\} = \min \left\{ \frac{1}{\sigma_{\rm max}^2(\mat{A}\mat{R}^{-1})}, 1\right\} \ge (1-\eta)^2.
    \end{equation*}
    Combining this with \cref{fact:KW-estimate} gives the lower bound.
    The upper bound is proven similarly. \hfill $\qed$

\subsection{\texorpdfstring{\rr{Proof of~\cref{lem:stability-sketching}}}{Proof of Lemma 6.3}}\label{app:stability-sketching-proof}
    By assumption \cref{eq:stable-sketch} and the normalization $\norm{\mat{A}} = 1$, we have
    \begin{equation*}
        \norm{\fl(\mat{S}\mat{A}) - \mat{S}\mat{A}} \lesssim u.
    \end{equation*}
    Thus, by this bound and the backward stability of Householder \QR factorization \cite[Thm.~19.4]{Hig02}, there exists a perturbation $\mat{E}$ of size $\norm{\mat{E}} \lesssim u$ and a matrix $\mat{\tilde{Q}}$ with orthonormal columns such that
    \begin{equation} \label{eq:perturbed-qr}
        \mat{S}\mat{A} + \mat{E} = \mat{\tilde{Q}} \mat{\hat{R}}.
    \end{equation}

    We begin by proving \cref{eq:R-whiten-fp}.
    Throughout, let $\mat{QR} = \mat{SA}$ be the exact QR decomposition of $\mat{SA}$.
    Note that \cref{eq:AR-whiten} implies the bounds
    \begin{equation} \label{eq:R-whiten}
        (1-\eta)\sigma_{\rm min}(\mat{A}) \le \norm{\smash{\mat{R}^{-1}}}^{-1} \le \norm{\mat{R}} \le (1+\eta) \norm{\mat{A}}. 
    \end{equation}
    Applying Weyl's inequality \cite[Sec.~4.3]{Stewart1} to \cref{eq:perturbed-qr}, we obtain
    \begin{equation*}
        \sigma_{\rm min}(\mat{\hat{R}}) =  \sigma_{\rm min}(\mat{\tilde{Q}}\mat{\hat{R}}) \ge \sigma_{\rm min}(\mat{S}\mat{A}) - \norm{\mat{E}} = \sigma_{\rm min}(\mat{R}) - \norm{\mat{E}} \ge \frac{1-\eta}{\kappa} - \norm{\mat{E}} \ge \frac{1}{20\kappa}.
    \end{equation*}
    In the penultimate inequality, we use the fact that $\sigma_{\rm min}(\mat{A}) = 1/\kappa$ and \cref{eq:R-whiten}.
    In the final inequality, we use the fact that $\eta \le 0.9$ and use the standing assumption $\kappa u \ll 1$ to enforce $\norm{\mat{E}} \le (20\kappa)^{-1}$.
    The first inequality of \cref{eq:R-whiten-fp} is proven.
    To obtain the second inequality, we take norms of \cref{eq:perturbed-qr}:
    \begin{equation*}
        \norm{\Rhat} = \norm{\smash{\mat{\tilde{Q}} \Rhat}} \le \norm{\mat{S}\mat{A}} + \norm{\mat{E}} = \norm{\mat{R}} + \norm{\mat{E}}\le (1+\eta) + \norm{\mat{E}} \le 2.
    \end{equation*}
    The penultimate inequality is \cref{eq:R-whiten}.
    In the final inequality, we use the assumption $\eta \le 0.9$ and the standing assumption $u \ll 1$ to enforce $\norm{\mat{E}} \le 0.1$.
    This completes the proof of \cref{eq:R-whiten-fp}.

    Now we prove \cref{eq:AR-whiten-fp}.
    By the embedding property and \cref{eq:perturbed-qr}, we have
    \begin{equation*} 
        \begin{split}
            \norm{\mat{A}\smash{\mat{\hat{R}}}^{-1}} &= \max_{\norm{\vec{u}} = 1} \norm{\mat{A}\smash{\mat{\hat{R}}}^{-1}\vec{u}} \le \frac{1}{1-\eta} \cdot \max_{\norm{\vec{u}} = 1} \norm{\mat{S}\mat{A}\smash{\mat{\hat{R}}}^{-1}\vec{u}} \\
            &= \frac{1}{1-\eta} \cdot \norm{\smash{\mat{\tilde{Q}} - \mat{E}\smash{\mat{\hat{R}}}^{-1}}} \le \frac{1}{1-\eta} + 200\kappa \norm{\mat{E}}.
        \end{split}
    \end{equation*}
    In the last inequality, we use the triangle inequality and \cref{eq:R-whiten-fp}.
    Rearranging gives the second inequality of \cref{eq:AR-whiten-fp}.
    The proof of the first inequality of \cref{eq:AR-whiten-fp} is similar and is omitted. \hfill $\qed$

\subsection{\texorpdfstring{\rr{Proof of~\cref{prop:backward-triangular}}}{Proof of Proposition 6.5}}\label{app:backward-triangular-proof}
The classical statement of the backward stability of triangular solves \cite[sec.~8.1]{Hig02} is that there exist a perturbation $\mat{\Delta R}$ such that
    \begin{equation*}
        (\Rhat + \mat{\Delta R}) (\Rhat^{-1}\vec{z} + \err(\Rhat^{-1}\vec{z})) = \vec{z}, \quad \norm{\mat{\Delta R}} \lesssim \norm{\Rhat}u.
    \end{equation*}
    Subtracting $\vec{z} + \mat{\Delta R}\cdot\Rhat^{-1}\vec{z}$
    from both sides yields
    \begin{equation*}
        (\Rhat + \mat{\Delta R}) \err(\Rhat^{-1}\vec{z}) = - \mat{\Delta R}\cdot \Rhat^{-1}\vec{z}.
    \end{equation*}
    Multiplying by $(\Rhat+\mat{\Delta R})^{-1} = \Rhat^{-1}(\Id + \mat{\Delta R}\cdot\Rhat^{-1})^{-1}$ yields
    \begin{equation*}
        \err(\Rhat^{-1}\vec{z}) = \Rhat^{-1}(\Id + \mat{\Delta R}\cdot\Rhat^{-1})^{-1}\left(- \mat{\Delta R}\cdot  \Rhat^{-1}\vec{z}\right).
    \end{equation*}
    The matrix $\Id + \Rhat^{-1}\cdot \mat{\Delta R}$ is invertible because $\norm{\mat{\Delta R}\cdot \Rhat^{-1}} \lesssim \kappa u \ll 1$ by \cref{lem:stability-sketching}.
    Thus, we are free to assume $\norm{\Rhat^{-1}\cdot \mat{\Delta R}} \le 1/2$ and thus $\norm{(\Id + \Rhat^{-1}\cdot \mat{\Delta R})^{-1}} \le 2$. Therefore, 
    \begin{equation*}
        \norm{(\Id + \mat{\Delta R}\cdot\Rhat^{-1})^{-1}\left(- \mat{\Delta R}\cdot  \Rhat^{-1}\vec{z}\right)} \le 2 \norm{\mat{\Delta R}}\norm{\Rhat^{-1}\vec{z}} \lesssim \norm{\Rhat^{-1}\vec{z}}u.
    \end{equation*}
    Here, we used the bound \cref{eq:R-whiten-fp}.
    Finally, use the bound $\norm{\vec{z}} \le \norm{\Rhat} \norm{\Rhat^{-1}\vec{z}} \lesssim \norm{\Rhat^{-1}\vec{z}}$ to reach the stated result for $\err(\Rhat^{-1}\vec{z})$.
    The result for $\err(\Rhat^{-\top}\vec{z})$ is proven in the same way.

\subsection{\texorpdfstring{Proof of \cref{cor:proving-backward}}{Proof of Corollary 6.8}} \label{app:proving-backward}

Let $\mat{A} = \sum_{i=1}^n \sigma_i^{\vphantom{\top}} \vec{u}_i^{\vphantom{\top}} \vec{v}_i^\top$ be a singular value decomposition.
Multiply \cref{eq:backward-error-from-formula} by $\vec{v}_i^\top$, then use \cref{eq:AR-whiten-fp} and the identity $\vec{v}_i^\top = \sigma_i^{-1} \vec{u}_i^\top \mat{A} = \sigma_i^{-2} \vec{v}_i^\top \mat{A}^\top \mat{A}$.
Finally, appeal to \cref{thm:backward-componentwise}. \hfill $\qed$
    
\section{Prony's method and quantum eigenvalue algorithms} \label{app:prony}

Recently, there has been interest in using frequency estimation algorithms such as Prony's method, the matrix pencil method \cite{KMC+22}, and ESPRIT \cite{LNY23a,DELZ24} to compute eigenvalues with quantum computers.
We will consider the simplest of these, Prony's method, shown in \cref{alg:prony}.
Observe that Prony's method involves two least-squares solves, which could be tackled with randomized methods.
We mention off-handedly that there exist structure-exploiting least-squares solvers for Toeplitz \cite{XXCB14} and Vandermonde \cite{WEB24} least-squares problems with an asympotically faster $\order(m\log^2(n) \log^2(1/u))$ runtime; these algorithms might or might not be faster than randomized methods for small values of $n$.

\begin{algorithm}[t]
	\caption{Prony's method} \label{alg:prony}
	\begin{algorithmic}[1]
		\Require Noisy measurements of the form $f_j = \sum_{i=1}^n \mu_i z_i^j + \varepsilon_j \in \complex$ for $0\le j \le m+n-1$
		\Ensure Vectors of estimates $\vec{\hat{\mu}},\vec{\hat{z}}\in\complex^n$ with $\hat{z}_i \approx z_i$, $\hat{\mu}_i \approx \mu_i$ (up to permutation)
		\State Form $\mat{A} = (f_{n-1+(i-j)})_{1\le i\le m,1\le j\le n}$ and $\vec{b} = (f_{n-1+i})_{1\le i\le m}$
        \State $\vec{p} \gets \Call{LeastSquaresSolver}{\mat{A},\vec{b}}$
        \State $\vec{\hat{z}} \gets \Call{PolynomialRoots}{z^n - p_1 z^{n-1} - p_2 z^{n-2} - \cdots - p_n z^0}$
        \State Form $\mat{A} = (z_j^{i-1})_{1\le i\le m+n,1\le j\le n}$ and $\vec{b} = (f_{i-1})_{1\le i\le m+n}$
        \State $\vec{\hat{\mu}} \gets \Call{LeastSquaresSolver}{\mat{A},\vec{b}}$
	\end{algorithmic}
\end{algorithm}

Prony's method can be used to estimate eigenvalues of a Hermitian matrix $\mat{H}$ as follows.
Suppose we use a quantum computer to collect measurements of the form
\begin{equation*}
    f_j = \vec{x}_0^* \e^{-\mathrm{i} (\Delta t) j \mat{H}} \vec{x}_0^{\vphantom{*}} + \varepsilon_j \quad \text{for } j=0,\ldots,t,
\end{equation*}
    where $\vec{x}_0$ is an appropriately chosen vector and $\varepsilon_j$ are independent random errors.
    By running Prony's method on this data and normalizing the outputs to have unit modulus $\hat{z}_i / |\hat{z}_i| = \e^{-\mathrm{i} (\Delta t) \lambda_i}$, we obtain an approximations $\lambda_i$ to the eigenvalues of $\mat{H}$.

With this context, the least-squares problem shown in the right panel of \cref{fig:dense} is constructed as follows.
We let $\mat{H} = \mat{H}(h)$ be the transverse field Ising mode Hamiltonian on $16$ sites with magnetization $h=1$, choose $\vec{x}_0$ to be the minimum-eigenvalue eigenvector of $\mat{H}(h=0)$, and set the noise $\varepsilon_j$ to be independent complex Gaussians with mean zero and standard deviation $10^{-6}$.
In the right panel of \cref{fig:dense}, we construct $\mat{A}$ and $\vec{b}$ as in line 1 of \cref{alg:prony}; thus, we consider only the \emph{first} least-squares solve of Prony's method.

\section{\texorpdfstring{Distillation of \cref{fact:stable-lanczos} from \cite{MMS18}}{Distillation of Fact 8.1 from [43]}} \label{app:MMS}

In this section, we describe how \cref{fact:stable-lanczos} can be recovered by the analysis in \cite[sec.~6]{MMS18}.
Assume the hypotheses of \cref{fact:stable-lanczos} and assume, without loss of generality, that $\norm{\mat{M}} = \norm{\vec{c}} = 1$.
Denote $\vec{y} = \mat{M}^{-1}\vec{c}$ and let $\fl(\vec{y})$ denote the output of Lanczos run for $k$ steps.
By \cite[Lem.~6.3]{MMS18} (with $\eta = \lambda_{\rm min}(\mat{M})$) and the proof of \cite[Thm.~6.2]{MMS18}, we have that
\begin{equation*}
    \norm{\fl(\vec{y}) - \vec{y}} \lesssim \left[k \min_{\deg p \le k-1} \max_{x \in [\lambda_{\rm min}(\mat{M})/2,1+\lambda_{\rm min}(\mat{M})/2]} |p(x) - x^{-1}| + \frac{k^4}{\lambda_{\rm min}(\mat{M})^2} \tilde{u} \right].
\end{equation*}
The minimum is taken over polynomials of degree $k-1$.
Note that, since $\cond(\mat{M}) \le 40$, we must have $\lambda_{\rm min}(\mat{M}) \ge 1/40$.
Ergo,
\begin{equation*}
    \norm{\fl(\vec{y}) - \vec{y}} \lesssim \left[k \min_{\deg p \le k-1} \max_{x \in [1/80,81/80]} |p(x) - x^{-1}| + k^4 \tilde{u} \right].
\end{equation*}
Berntein's theorem for polynomial approximation of analytic functions\cite[Thm.~8.2]{Tre19} implies
\begin{equation*}
    \min_{\deg p \le k-1} \max_{x \in [1/80,81/80]} |p(x) - x^{-1}| \le \tilde{u} \quad \text{for } k = \order(\log(1/\tilde{u})).
\end{equation*}
Thus, for $k = \order(\log(1/\tilde{u}))$.
\begin{equation*}
    \norm{\fl(\vec{y}) - \vec{y}} \lesssim \tilde{u} \log^4(1/\tilde{u}) \lesssim \tilde{u}.
\end{equation*}
In the last line, we use the hypothesis $\log(1/\tilde{u}) \le n$ to suppress the polylogarithmic factors in $1/\tilde{u}$.

We mention offhandedly that \cite[Thm.~2.2]{DGK98} gives an analysis of Lanczos linear solves with a smaller polylogarithmic factor $\log^{0.5}(1/\tilde{u})$, but this result does not treat the floating point errors incurred on line 11 of \cref{alg:lanczos}.
Since we only need the qualitative conclusion \cref{eq:lanczos-stability} for present purposes, we use the end-to-end but less sharp analysis of \cite{MMS18} instead of \cite[Thm.~2.2]{DGK98}.

\end{document}